\documentclass[12pt]{amsart}
\usepackage{amssymb,amsmath,amsthm}
\usepackage[msc-links]{amsrefs}
\usepackage[all,cmtip]{xy}
\usepackage{graphicx}
\calclayout
\newtheorem{theorem}{Theorem}
\newtheorem{lemma}[theorem]{Lemma}
\newtheorem{remark}[theorem]{Remark}
\newtheorem{corollary}[theorem]{Corollary}
\newtheorem{definition}[theorem]{Definition}
\newtheorem{proposition}[theorem]{Proposition}
\numberwithin{theorem}{section} \numberwithin{equation}{section}
\newcommand\scalemath[2]{\scalebox{#1}{\mbox{\ensuremath{\displaystyle #2}}}}
\newcommand*{\MyScaleMedium}{0.87}

\newcommand*{\MyScaleTiny}{0.7}

\newcommand{\im}{\textnormal{Im}\;\;\!\!\!}
\newcommand{\modvar}{\varpi}
\begin{document}
\title[Six Line Configurations]{Six line configurations and string dualities}
\author{A. Clingher}
\address{Department of Mathematics and Computer Science
University of Missouri
St. Louis MO 63121}
\email{clinghera@umsl.edu}
\author{A. Malmendier}
\address{Department of Mathematics and Statistics, Utah State University,
Logan, UT 84322}
\email{andreas.malmendier@usu.edu}
\author{T. Shaska}
\address{Department of Mathematics and Statistics, Oakland University, Rochester, MI 48309}
\email{shaska@oakland.edu}
\begin{abstract}
We study the family of K3 surfaces of Picard rank sixteen associated with the double cover of the projective plane branched along the union of six lines, and the family of its Van~\!Geemen-Sarti partners, i.e., K3 surfaces with special Nikulin involutions, such that quotienting by the involution and blowing up recovers the former. We prove that the family of Van~\!Geemen-Sarti partners is a four-parameter family of K3 surfaces with $H \oplus E_7(-1) \oplus E_7(-1)$ lattice polarization. We describe explicit Weierstrass models on both families using even modular forms on the bounded symmetric domain of type $IV$. We also show that our construction provides a geometric interpretation, called geometric two-isogeny, for the F-theory/heterotic string duality in eight dimensions.
\end{abstract}
\maketitle
\section{Introduction}
In this article, we consider configurations of six lines in general position on the projective plane. The double cover of the plane branched along their union is a K3 surface after resolving only ordinary double points. The moduli space of such K3 surfaces was described in~\cite{MR1136204}. Kloosterman classified all possible types of elliptic fibrations with a section on them in~\cite{MR2254405}. In \cite{MR3201823}, the authors consider K3 surfaces which are double covers of a blow-up of $\mathbb{P}^2$, branched along rational curves. They classified the elliptic fibrations on such surfaces and their van Geemen-Sarti involutions.
\par The assumption that the six lines are in general position implies that the Picard rank of the resulting K3 surface is sixteen. In the special case when the six lines are tangent to a conic, the Picard rank is, generically, seventeen and one obtains as K3 surface a Kummer surface $\mathrm{Kum}(\operatorname{Jac}\mathcal{C})$ of the Jacobian $\operatorname{Jac}(\mathcal{C})$ of a generic genus-two curve $\mathcal{C}$. There is then, as shown in \cites{MR2427457, MR2824841, MR2935386, MR3366121}, a closely related K3 surface, called the \emph{Shioda-Inose surface} $\mathrm{SI}(\operatorname{Jac}\mathcal{C})$, which carries a \emph{Nikulin involution}, i.e., an automorphism of order two preserving the holomorphic two-form, such that quotienting by this involution and blowing up the fixed points recovers the Kummer surface. The Shioda-Inose surface $\mathrm{SI}(\operatorname{Jac}\mathcal{C})$ carries a canonical lattice polarization of type $H \oplus E_8(-1) \oplus E_7(-1)$ and is part of \emph{a geometric two-isogeny}:
 \begin{equation}
\xymatrix{
\mathrm{Kum}(\operatorname{Jac}\mathcal{C})  \ \  \ar @/_0.5pc/ @{-->} _{} [rr]
&
& \ \ \mathrm{SI}(\operatorname{Jac}\mathcal{C}) \ar @/_0.5pc/ @{-->} _{} [ll]
}
\end{equation}
establishing a one-to-one correspondence between two different types of surfaces with the same Hodge-theoretic data: principally polarized abelian surfaces and algebraic K3 surfaces polarized the special lattice $H \oplus E_8(-1) \oplus E_7(-1)$. The key geometric ingredient in this construction is a normal form equation for an elliptically fibered K3 surface whose periods determine a point $\tau$ in the Siegel upper-half space $\mathbb{H}_2$, with the coefficients in the equation being Siegel modular forms. The normal form equation, as well as the two-isogeny construction, are due in different forms to Kumar \cite{MR2427457} and to Clingher and Doran \cite{MR2935386}.
\par In this article, we extend the notion of geometric two-isogeny to K3 surfaces with Picard rank sixteen. In this context, Kummer surfaces are replaced by what we shall refer to as \emph{double sextic surfaces} - K3 surfaces $\mathcal{Y}$ obtained as minimal resolutions of double covers of the projective plane branched along a configuration of six distinct lines. The Shioda-Inose surfaces from above are then replaced, as shown by Clingher and Doran in \cite{MR2824841} by K3 
surfaces $\mathcal{X}$ polarized by the rank-sixteen lattice $H \oplus E_7(-1) \oplus E_7(-1)$. Similarly to the Shioda-Inose case, each of these K3 surfaces $\mathcal{X}$ carries a special Nikulin involution, $\jmath_{\mathcal{X}} $ called \emph{Van~\!Geemen-Sarti involution}. When quotienting by the involution $\jmath_{\mathcal{X}} $ and blowing up the fixed locus, one recovers the corresponding double-sextic surface $\mathcal{Y}$ together with a rational double cover map $ \Phi \colon \mathcal{X} \dashrightarrow \mathcal{Y}$. However, the Van~\!Geemen-Sarti involutions $\jmath_{\mathcal{X}} $ no longer determine Shioda-Inose structures. Instead, they appear as fiber-wise translation by two-torsion in a suitable Jacobian elliptic fibration $\pi^{\mathcal{X}}_{\mathrm{alt}}$. The geometric two-isogeny picture is then given by the diagram below:
 \begin{equation}
\xymatrix{
\mathcal{X} \ar @(dl,ul) _{\jmath_{\mathcal{X}}} \ar [dr] _{\pi^{\mathcal{X}}_{\mathrm{alt}}} \ar @/_0.5pc/ @{-->} _{\hat{\Phi}} [rr]
&
& \mathcal{Y} \ar @(dr,ur) ^{\jmath_{\mathcal{Y}}} \ar [dl] ^{\pi^{\mathcal{Y}}_{\mathrm{alt}}} \ar @/_0.5pc/ @{-->} _{\Phi} [ll] \\
& \mathbb{P}^1 }
\end{equation}
\noindent We shall refer to the K3 surfaces $\mathcal{X}$ as the \emph{Van~\!Geemen-Sarti partners} of the double sextic surface $\mathcal{Y}$. From a physics point of view, Jacobian elliptic fibrations on K3 surfaces correspond to a certain subclass of eight-dimensional compactifications of the type IIB string in which the axio-dilaton field varies over a base~\cites{MR1403744,MR1409284,MR1412112,MR1408164}. The period lattice of the Van~\!Geemen-Sarti partners describes physical models dual to the $\mathfrak{e}_8\oplus \mathfrak{e}_8$ heterotic string, with an unbroken gauge algebra $\mathfrak{e}_7\oplus \mathfrak{e}_7$ ensuring that two Wilson line expectation values are non-zero.  A similar result holds for the $\mathfrak{so}(32)$ heterotic string with an unbroken gauge algebra $\mathfrak{so}(24)\oplus \mathfrak{su}(2)^{\oplus 2}$. The function field of the Narain moduli space of these heterotic theories turns out to be the ring of modular forms of \emph{even characteristic} on the bounded symmetric domain of type $IV$ introduced by Matsumoto et al.~\!\cite{MR1204828}.  Geometric two-isogeny provides a more refined and geometric understanding for this string duality on a natural sub-space of the full eighteen dimensional moduli space \cites{MR2369941, MR2826187, MR3366121}:  by taking the K3 surface to be the Shioda-Inose surface $\mathrm{SI}(\operatorname{Jac}\mathcal{C})$, the F-theory/heterotic string duality is manifested as the aforementioned geometric two-isogeny.
\par This article is structured as follows: in Section~2 we review the work of Dolgachev and Ortland \cite{MR1007155} and the moduli space associated with six-line configurations in the projective plane.  We define new invariants of six-line configurations that generalize the Igusa invariants of binary sextics. We construct the function field of the moduli space explicitly, by determining a complete set of generators for the ring of modular forms of even characteristic. In Section~3 we construct explicit Weierstrass models for three Jacobian elliptic fibrations on the family of double-sextic surfaces $\mathcal{Y}$. One of them, which we call the \emph{alternate fibration}, is of particular importance: the coefficients in its Weierstrass equation are the generators of the ring of modular forms derived before. In Section~4 we construct the family of Van~\!Geemen-Sarti partners $\mathcal{X}$ of the double-sextic surfaces $\mathcal{Y}$ polarized by the lattice $H \oplus E_7(-1) \oplus E_7(-1)$. There are four non-isomorphic elliptic fibrations on $\mathcal{X}$; three will be important for the considerations in this article, and Weierstrass models will be constructed for them. Using the Van~\!Geemen-Sarti involution, we will determine the coefficients of these Weierstrass models in terms of the modular forms found in Section~2. Using a result of Vinberg \cite{MR3235787} and its interpretation in string theory in \cite{MR3366121}, we prove that the function field of the Narain moduli space of quantum-exact heterotic string compactifications with two non-vanishing Wilson lines  is the ring of Siegel modular forms of {\it even weight.}  In Section~5 we discuss the specialization of six-line configurations tangent to a common conic and the associated K3 surfaces. We find perfect agreement in this case with the results in~\cites{MR2824841,MR3712162,MR3366121}. 
\section{Invariants of six-line configurations in the projective plane}
\label{sec:invariants}
The Pl\"ucker embedding algebraically embeds the Grassmannian $\operatorname{Gr}(k,n;\mathbb{C})$ of all $k$-dimensional sub-spaces of an $n$-dimensional complex vector space $V$ as a sub-variety of the projective space $\mathbb{P}(\wedge^k V)$. The homogeneous coordinates of the image under the Pl\"ucker embedding, with respect to the natural basis of the exterior space $\wedge^k V$ relative to a chosen basis in $V$, are called \emph{Pl\"ucker coordinates}. The image of the Pl\"ucker embedding is an intersection of a number of quadrics defined by the so called \emph{Pl\"ucker relations}. 
\par We consider the situation $k=3$ and $n=6$ with $\dim \operatorname{Gr}(k,n;\mathbb{C})=9$. We start with the geometric setup of an \emph{ordered} configuration of six lines in general position in the projective plane $\mathbb{P}^2$. We write each line in the form $\ell_i: a_i z_1 + b_i  z_2 + c_i z_3 =0$ for $i=1, \dots , 6$ with $[z_1:z_2:z_3] \in \mathbb{P}^2$.  The coefficients of the lines are assembled in vectors $\mathbf v_i = \langle a_i, b_i, c_i \rangle^t$ and form a matrix $\mathbf{A} \in \operatorname{Mat}(3,6;\mathbb{C})$ given by $\mathbf{A}= [ \mathbf v_1 | \cdots | \mathbf v_6]$. Let  $\mathbf{A}_{ijk} = [ \mathbf v_i | \mathbf v_j | \mathbf v_k]$ and $D_{ijk} = \det \mathbf{A}_{ijk}$ be the Pl\"ucker coordinates derived from $\mathbf{A} \in \operatorname{Mat}(3,6;\mathbb{C})$ considered as an element of the Grassmannian $\operatorname{Gr}(3,6;\mathbb{C})$.  
\par We consider the following cases of configurations of six lines in $\mathbb{P}^2$: 
\begin{definition}
\label{def:sstable}
We consider configurations of six lines in $\mathbb{P}^2$ that
\begin{enumerate}
\item[(0)] contain six lines in general position,
\item are tangent to a common conic,
\item contain three lines which are coincident in one point,
\item contain one line which is coincident with two different pairs of lines in two different points,
\item contain three lines pairwise coincident in three different points, and each of the three remaining lines is coincident in one intersection point,
\item are combinations of case (1) and cases (2) through (4),
\item[(6a)] contain four lines which intersect in one point,
\item[(6b)] contain one double line.
\end{enumerate}
\end{definition}
Configurations that include cases (0) through (6a) and (6b) are called \emph{semi-stable} configurations. On configurations of six lines we have a right action of  $(\mathbb{C}^*)^6$ given by rescaling each line separately, and the obvious left action of $\operatorname{GL}_3(\mathbb{C})$ by acting on $[z_1:z_2:z_3]\in \mathbb{P}^2$. Next, we want to describe the isomorphism classes of such configurations of six lines. We define the so called degree-one Dolgachev-Ortland coordinates \cite{MR1007155} for configurations of six lines in $\mathbb{P}^2$ to be given by
\begin{equation}
\label{eqn:DOcoords}
\begin{array}{lclclcl}
 t_1 & = & D_{135} D_{246}, &\quad& t_2 & = & D_{145} D_{236}, \\
 t_3 & = & D_{146} D_{235}, &\quad& t_4 & = & D_{136} D_{245},\\
 t_5 & = & D_{125} D_{346}, &\quad& t_6 & = & D_{126} D_{345}, \\
 t_7 & = & D_{134} D_{256}, &\quad& t_8 & = & D_{124} D_{356},\\
 t_9 & = & D_{156} D_{234},  &\quad& t_{10} & = & D_{123} D_{456}.
\end{array}
\end{equation}
We have the following:
\begin{lemma}
The degree-one coordinates $t_1, \dots, t_{10}$ satisfy the relations
\begin{equation}
\label{eqn:PlueckerRelations}
\begin{array}{lll}
t_1-t_2-t_5-t_9,& 
t_1-t_2-t_6-t_7, &
t_1-t_3-t_5-t_{10},\\
t_1-t_3-t_6-t_8,&
t_1-t_4-t_7-t_{10}, &
t_1-t_4-t_8-t_9,\\ 
t_2-t_3+t_7-t_8, &
t_2-t_3+t_9-t_{10},& 
t_2-t_4+t_5-t_8,\\
t_2-t_4+t_6-t_{10},& 
t_3-t_4+t_5-t_7, &
t_3-t_4+t_6-t_9,\\
t_5-t_6-t_7+t_9,&
t_5-t_6-t_8+ t_{10}, &
t_7-t_8-t_9+t_{10}.
\end{array}
\end{equation}
In particular, only five relations among the fifteen relations are linearly independent.
\end{lemma}
\begin{proof}
The proof follows by explicit computation for any matrix $\mathbf{A} \in \operatorname{Mat}(3,6;\mathbb{C})$.
\end{proof}
\noindent
One also introduces the degree-two Dolgachev-Ortland coordinate given by
\begin{equation}
\label{Eqn:moduli}
 R= D_{123}  D_{145} D_{246} D_{356} - D_{124} D_{135} D_{236} D_{456}. 
\end{equation}
We have the following:
\begin{lemma}
The degree-two coordinate $R$ satisfies
\begin{equation}
\label{eqn:PlueckerRelations2}
 R^2 = \frac{1}{12} \left( \Big(\sum_{i=1}^{10} t_i^2\Big)^2 -  4\sum_{i=1}^{10} t_i^4 \right).
 \end{equation}
\end{lemma}
\begin{proof}
The proof follows by explicit computation for any matrix $\mathbf{A} \in \operatorname{Mat}(3,6;\mathbb{C})$.
\end{proof}
The different strata in the moduli space can now be characterized as follows:
\begin{lemma}
\label{lem:equiv}
In Definition~\ref{def:sstable} we have the following:
\begin{enumerate}
 \item[(0)] $\Leftrightarrow$ no element of $(t_i)_{i=1}^{10}$ vanishes and $R\not =0$,
 \item $\Leftrightarrow$ no element of $(t_i)_{i=1}^{10}$ vanishes and $R=0$,
 \item  $\Leftrightarrow$ exactly one element of $(t_i)_{i=1}^{10}$ vanishes,
 \item  $\Leftrightarrow$ exactly two elements of $(t_i)_{i=1}^{10}$ vanish,
 \item  $\Leftrightarrow$ exactly three elements of $(t_i)_{i=1}^{10}$ vanish,
 \item  $\Leftrightarrow$ up to three elements of $(t_i)_{i=1}^{10}$ vanish and $R=0$,
 \item  $\Leftrightarrow$ exactly four elements of $(t_i)_{i=1}^{10}$ vanish and $R=0$.
\end{enumerate}
\end{lemma}
\begin{proof}
Configurations of six lines no three of which are concurrent have four homogeneous moduli which we denote by $a, b, c, d$. A general matrix $\mathbf{A} \in \operatorname{Mat}(3,6;\mathbb{C})$ is written in terms of only $a, b, c, d$ using a $\operatorname{GL}_3(\mathbb{C})$ transformation. The lines are then in the form of Equations~(\ref{lines}). We discuss the details in Section~\ref{ssec:natural_fibration}. Equations~(\ref{compare1}) determine the Dolgachev-Ortland coordinates in terms of these moduli. We can easily check necessary and sufficient conditions for cases (1) through (6). It follows from Equation~(\ref{compare1}) and \cite{MalmendierClingher:2018}*{Prop.~5.13} that $R=0$ in Equation~(\ref{Eqn:moduli}) if and only if the six lines in general position are tangent to a common conic. 
\end{proof}
We have the following:
\begin{lemma}
\label{lem:invariance}
For a configuration of six lines in $\mathbb{P}^2$ the point
\begin{equation}
 [t_1: \dots : t_{10}: R] \in \mathbb{P}(1, \dots, 1, 2)
\end{equation}
in complex weighted projective space, is well-defined and invariant under the right action of  $(\mathbb{C}^*)^6$ and the left action of $\operatorname{GL}_3(\mathbb{C})$ on $\mathbf{A}$.
\end{lemma}
\begin{proof}
The point in weighted projective space is well-defined because of Lemma~\ref{lem:equiv}. The invariance under the right action of $(\mathbb{C}^*)^6$ on $\mathbf{A}$ is immediate. The invariance under the left action of $\operatorname{GL}_3(\mathbb{C})$ follows from a computation showing that the coordinates $t_i$ for $1\le i \le 10$ and $R$ rescale by the determinant with weight two and four, respectively, and the point in weighted projective space remains invariant.
\end{proof}
For more details we refer to \cites{MR1007155, MR1828467}.  The following is a corollary of Lemma~\ref{lem:invariance}:
\begin{corollary}[\cite{MR1007155}]
The moduli space of configurations of six lines in $\mathbb{P}^2$ is isomorphic to the algebraic variety in $\mathbb{P}(1, \dots, 1, 2)$ with the coordinates $[t_1: \dots : t_{10}: R]$ given by Equations~(\ref{eqn:PlueckerRelations}) and (\ref{eqn:PlueckerRelations2}), and $R \not =0$ and $t_i\not =0$ for all $i \in \{1, \dots, 10\}$.
\end{corollary}
We define the moduli space $\mathfrak{M}(2)^+$ to be the moduli space of \emph{ordered} configurations of six lines in $\mathbb{P}^2$ that fall into cases (0) through (5) in Definition~\ref{def:sstable}, i.e., 
\begin{equation}
\label{eqn:M2+}
 \mathfrak{M}(2)^+ = \left\lbrace \big[ t_1 : \dots : t_{10} : R ]  \ \Big\vert \begin{array}{l}\text{$t_i=0$ for at most three $i \in \{1, \dots, 10\}$},\\ \text{Eqns.~(\ref{eqn:PlueckerRelations}) and (\ref{eqn:PlueckerRelations2}) hold.}\end{array}  \right\rbrace.
\end{equation}
The notation $\mathfrak{M}(2)^+$ indicates (i) the existence of a level-two structure obtained by splitting up six indices into two pairs of three, and (ii) the fact that we include all cases (1) through (5) in Definition~\ref{def:sstable} in addition to case (0). 
\par For a given ordered configuration of lines $\{\ell_1, \dots , \ell_6\}$ in general position, let us fix six out of fifteen points of intersection, namely the points
\begin{equation}
\begin{split}
 p_1 = \ell_2 \cap \ell_3, \qquad  p_2 & = \ell_1 \cap \ell_3, \qquad  p_3 = \ell_1 \cap \ell_2, \\
 p_4 = \ell_5 \cap \ell_6, \qquad  p_5 & = \ell_4 \cap \ell_6, \qquad  p_6 = \ell_4 \cap \ell_5.
\end{split}
\end{equation}
Given any non-singular conic $C \subset \mathbb{P}^2$, we define the dual of a point $p_i \not \in C$ to be the line $\ell'_i$ that joins the two points of $C$ on the two tangent lines of $C$ passing through $p_i$; if $p_i \in C$ we define $\ell'_i$ to be the tangent line of $C$ at $p_i$. Changing the conic $C$ to another non-singular conic $C'$ in this construction simply transforms the lines $\ell'_i$ by a projective automorphism of $\mathbb{P}^2$. We then say that the two configurations $\{\ell'_1, \dots , \ell'_6\}$ and $\{\ell_1, \dots , \ell_6\}$ are \emph{in association}.  It was proved in \cite{MR1828467} that $\{\ell'_1, \dots , \ell'_6\}$ and $\{\ell_1, \dots , \ell_6\}$ are associated if and only if their respective matrices $\mathbf{A}'$ and $\mathbf{A}$ satisfy  $\mathbf{A}' \cdot D \cdot \mathbf{A}^t=0$ for some diagonal matrix $D$ with $\det D \not =0$. 
\par Mapping an ordered configuration of six lines to an associated ordered configuration defines an involution $\imath$ on $\mathfrak{M}^+(2)$ with a fixed point set that consists of configurations of six lines tangent to a common conic, and in terms of the Dolgachev-Ortland coordinates it is given by
\begin{equation}
\label{eqn:i_action}
 \imath: \, [\ t_1: \dots : t_{10}: R \ ] \to [\ t_1: \dots : t_{10}: -R \ ].
\end{equation}
We define a four-dimensional sub-space $\mathfrak{M}(2)$ of $\mathbb{P}^9$ by setting
\begin{equation}
\label{M_2}
 \mathfrak{M}(2) = \left\lbrace \big[ t_1 : \dots : t_{10} ] \in \mathbb{P}^9 \ \Big\vert \begin{array}{l} \text{$t_i=0$ for at most three $i \in \{1, \dots, 10\}$},\\ \text{and Eqns.~\!(\ref{eqn:PlueckerRelations}) hold.}\end{array}  \right\rbrace.
\end{equation}
We also set
\begin{equation}
 \overline{\mathfrak{M}(2)} = \left\lbrace \big[ t_1 : \dots : t_{10} ] \in \mathbb{P}^9 \ \Big\vert \begin{array}{l} \text{Eqns.~\!(\ref{eqn:PlueckerRelations}) hold.}\end{array} \! \right\rbrace.
\end{equation}
Notice that, apart from the six-line configurations listed in Definition~\ref{def:sstable}, there are more degenerate configurations: there are configurations such that exactly six elements of $(t_i)_{i=1}^{10}$ vanish; there are also configurations such that exactly four elements of $(t_i)_{i=1}^{10}$ vanish, $R\not =0$, and all non-vanishing $t_i$'s equal $\pm 1$. Since $\mathfrak{M}(2)$ is a four-dimensional linear sub-space of $\mathbb{P}^9$, it is easy to show \cite{MR1204828}*{Sec.~3.2} that $\overline{\mathfrak{M}(2)}$ is in fact isomorphic to $\mathbb{P}^4$.
\par We take the map $\operatorname{pr}$ to be the projection from $\mathbb{P}(1,\dots,1,2) \backslash \{[0:\dots:0:1]\} \to \mathbb{P}^9$ given by
$[t_1: \dots : t_{10}: R] \mapsto [t_1: \dots : t_{10}]$. We have the following:
\begin{lemma}
\label{lem:M2}
We have $\operatorname{pr} =  \operatorname{pr}\circ \, \imath: \mathfrak{M}(2)^+ \to \mathfrak{M}(2)$ and $\operatorname{pr}(\mathfrak{M}(2)^+)\cong \mathfrak{M}(2)$.
\end{lemma}
\subsection{The modular description}
\label{ssec:modular_description}
The moduli spaces $\mathfrak{M}(2)$ and $\mathfrak{M}(2)^+$ have modular descriptions based on the seminal work in \cite{MR1204828}. By $\mathbf{H}_2$ we denote the set of all complex two-by-two matrices $\modvar$ over $\mathbb{C}$ such that the hermitian matrix $(\modvar-\modvar^\dagger)/(2i)$ is positive definite, i.e.,
\begin{equation}
\label{Siegel_tau}
 \mathbf{H}_2 = \left\lbrace  \left( \begin{array}{cc} \tau_1 & z_1 \\ z_2 & \tau_2\end{array} \right) \in \operatorname{Mat}(2,2;\mathbb{C}) \; \Big|
 \; 4 \, \im{\tau_1} \, \im{\tau_2} > |z_1 - \bar{z}_2|^2, \; \im{\tau_2} > 0 \right\rbrace ,
\end{equation}
and the modular group $\Gamma \subset \operatorname{U}(2,2)$ given by
\begin{equation}
\label{modular_group}
 \Gamma =  \left\lbrace G \in \operatorname{GL}_4\big(\mathbb{Z}[i]\big) \,  \Big| \, G^\dagger \cdot \left( \begin{array}{cc} 0 & \mathbb{I}_2 \\  -\mathbb{I}_2 & 0 \end{array}\right) \cdot G = \left( \begin{array}{cc} 0 & \mathbb{I}_2 \\  -\mathbb{I}_2 & 0 \end{array}\right)  \right\rbrace.
\end{equation}
The modular group acts on $\modvar \in \mathbf{H}_2$ by 
\begin{gather*}
\forall \, G= \left(\begin{matrix} A & B \\ C & D \end{matrix} \right) \in \Gamma: \quad G\cdot \modvar=(A\cdot \modvar+B)(C\cdot \modvar+D)^{-1}.
\end{gather*}
It was shown in \cite{MR1204828}*{Prop.~\!1.5.1} that $\Gamma$ is generated by the five elements $G_1$, $G_2$, $G_3$, $G_4$, $G_5$ given by
\begin{equation}
\label{eqn:generators}
\scalemath{\MyScaleTiny}{
\begin{array}{ccccc}
\left(\begin{array}{rrrr}  i 	&	&	&	\\	& 1	& 	&	\\ 	&	& i	&	\\	&	&	& 1 \end{array}\right)	, &
\left(\begin{array}{rrrr} 1 	&1	&	&	\\ 0	& 1	& 	&	\\ 	&	& 1	& 0	\\	&	& -1	& 1 \end{array}\right)	, &
\left(\begin{array}{rrrr} 0 	&1	&	&	\\ 1	& 0	& 	&	\\ 	&	& 0	& 1	\\	&	& 1	& 0 \end{array}\right)	, &
\left(\begin{array}{rrrr} 1 	&0	&1	&0 	\\ 0	& 1	& 0	&0	\\ 	&	& 1	& 	\\	&	& 	& 1 \end{array}\right)	, &
\left(\begin{array}{rrrr}   	& 	& 1	&  	\\ 	& 	& 	&1	\\ -1	&	& 	& 	\\	&-1	& 	&  \end{array}\right),
\end{array}}
\end{equation}
with determinants $\det{(G_1)}=-1$ and $\det{(G_k)}=1$ for $k=2,\dots, 5$. We also introduce the principal modular sub-group of \emph{complex level} $1+i$ (over the Gaussian integers) given by
\begin{equation}
 \Gamma(1+i) = \left\lbrace G \in \Gamma \,  \Big| \, G \equiv \mathbb{I}_4 \! \!\mod{1+i} \right\rbrace .
\end{equation} 
There is an additional involution $\mathcal{T}$ acting on elements of $\mathbf{H}_2$ by transposition, i.e., $\modvar \mapsto \mathcal{T}\cdot \modvar=\modvar^t$, yielding extended groups obtained from the semi-direct products
\begin{equation}
\label{modular group_extended}
 \Gamma_{\mathcal{T}} = \Gamma \rtimes \langle \mathcal{T} \rangle , \qquad 
 \Gamma_{\mathcal{T}}(1+i) = \Gamma(1+i) \rtimes \langle \mathcal{T} \rangle ,
\end{equation} 
where $\langle \mathcal{T} \rangle$ is the sub-group generated by $\mathcal{T}$. We will always write elements $g \in \Gamma_{\mathcal{T}}$ in the form  $g = G\,\mathcal{T} ^{n}$ with $G \in \Gamma$ and $n \in \{0,1\}$.
A \emph{modular form $f$ of weight $2k$ relative to a finite-index sub-group $\Gamma' \subset \Gamma_{\mathcal{T}}$ with character $\chi_{f}$} is a holomorphic function on $\mathbf{H}_2$ such that
\begin{equation}
 \forall \, \modvar \in \mathbf{H}_2, \ \forall \, g = G \mathcal{T} ^{n} \in \Gamma': \ f\big(g\cdot \modvar\big) = \chi_f(g) \ \det(C\modvar+D)^{2k} \ f(\modvar) .
\end{equation}
There is a well-known isomorphism $\Gamma/\Gamma(1+i)\cong \mathrm{S}_6$ -- since both groups are in fact isomorphic to $\operatorname{Sp}_4(\mathbb{Z}/2\mathbb{Z})$ -- where $\mathrm{S}_6$ is the permutation group of six elements. By $S_G$ we denote the image of $G \in \Gamma$ under the natural quotient map $\Gamma \to \mathrm{S}_6$ and by $\operatorname{sign}\!{(S_G)}$ the sign of this permutation $S_G$.  The following was proven in \cite{MR1204828}:
\begin{theorem}[Props.~3.1.1, 3.1.3, 3.1.5 in \cite{MR1204828}]
\label{thm1}
\hspace{2em}
\begin{enumerate}
\item There are ten theta functions $\theta^2_i(\modvar)$ for $1 \le i \le 10$ which are non-zero modular forms of weight two relative to $ \Gamma_{\mathcal{T}}(1+i)$ and for each $g = G\,\mathcal{T} ^{n}  \in \Gamma_{\mathcal{T}}(1+i)$ with $n \in \lbrace 0,1 \rbrace$ the modular forms $\theta^2_i(\modvar)$ transform with $\chi_{\theta_i}(g)=\det{(G)}$.
\item Any five of the ten functions $\theta^2_i(\modvar)$ for $1 \le i \le 10$ generate the ring of modular forms of level $1+i$ and character $\chi(g)=\det{(G)}$ for all $g \in\Gamma_{\mathcal{T}}(1+i)$.
\item There is a unique function $\Theta(\modvar)$ which is a non-zero modular form of weight four relative to $\Gamma_{\mathcal{T}}$ such that for each $g = G \mathcal{T} ^{n}  \in \Gamma_{\mathcal{T}}$ with $n \in \lbrace 0,1 \rbrace$ the modular form $\Theta(\modvar)$ transforms with character $\chi_{\Theta}(g)=(-1)^n \det{(G)} \, \operatorname{sign}{(S_G)}$ and satisfies
\begin{equation}
\label{eqn:Tsqr}
\Theta(\modvar)^2 = 2^{-6}\cdot3^5\cdot5^2 \left( \big(\sum_{i=1}^{10} \theta_i(\modvar)^4\big)^2 - 4 \sum_{i=1}^{10} \theta_i(\modvar)^8\right) .
\end{equation}
\end{enumerate}
\end{theorem}
In the interest of keeping this section short, we do not give explicit formulas for $\theta^2_i(\modvar)$ with $1 \le i \le 10$. However, just as there are simple sum formulas for theta functions of even and odd characteristic in genus two and genus one, the same holds for the theta functions $\theta^2_i(\modvar)$ in Theorem~\ref{thm1}: they are simply theta functions of complex characteristic. All quadratic relations among the  even theta functions $\theta^2_i(\modvar)$ for $1 \le i \le 10$ can then be derived explicitly. We refer to \cite{MR1204828}*{Sec.~2} for details. 
\begin{remark}
 The space $\mathbf{H}_2$ is a generalization of the Siegel upper-half space $\mathbb{H}_2$. In fact, elements invariant under the involution $\mathcal{T}$ are precisely the two-by-two symmetric matrices over $\mathbb{C}$ whose imaginary part is positive definite, i.e.,
\begin{equation}
\mathbb{H}_2 = \Big\{  \modvar \in \mathbf{H}_2 \, \Big\vert \, \modvar^t=\modvar \Big\} .
\end{equation}
It was proven in \cite{MR1204828}*{Lemma 2.1.1(vi)} that for $\modvar = \tau \in \mathbb{H}_2$ we have $\theta_i(\modvar)=\vartheta_i(\tau)^2$  where $\vartheta_i(\tau)$ for $1 \le i \le 10$ are the even theta functions of genus two. We provide a geometric cross-check for (the squares of) these reduction formulas in Proposition~\ref{prop-5.6}.
\end{remark}
The following describes the action of the \emph{full} modular group on the theta functions:
\begin{lemma}
\label{lem:Gtransfo}
The action of the generators $\mathcal{T}, G_1, \dots, G_5 \in\Gamma_{\mathcal{T}}$ in Equation~(\ref{eqn:generators}) on $\theta_i(\modvar)$ with $1 \le i \le 10$ and $\rho=-\det{(\modvar)}$ is given in the following table:
\begin{equation}
\label{eqn:Gtransfo}
\scalemath{\MyScaleMedium}{
\begin{array}{l|rrrrrrrrrr}
	& \theta_1		& \theta_2		& \theta_3		& \theta_4		& \theta_5		& \theta_6		& \theta_7		& \theta_8		& \theta_9		& \theta_{10}\\
\hline
\mathcal{T} & \theta_1		& \theta_2		& \theta_3		& \theta_4		& \theta_5		& \theta_6		& \theta_7		& \theta_8		& \theta_9		& \theta_{10}\\
G_1^{\pm1}	& \theta_1		& \theta_2		& \theta_3		& \theta_4		& \theta_5		& \theta_6		& \theta_7		& \theta_8		& \theta_9		& -\theta_{10}\\
G_2^{\pm1}	& \theta_1		& \theta_4		& \theta_3		& \theta_2		& \theta_8		& \theta_{10}	& \theta_7		& \theta_5		& \theta_9		&  \theta_{6}\\
G_3^{\pm1}	& \theta_1		& \theta_2		& \theta_4		& \theta_3		& \theta_7		& \theta_9		& \theta_5		& \theta_8		& \theta_6		& \theta_{10}\\
G_4^{\pm1}	& \theta_3		& \theta_4		& \theta_1		& \theta_2		&\pm i\,\theta_5&\pm i\,\theta_6	& \theta_9		&\pm i\,\theta_8	& \theta_7		&\pm i\,\theta_{10}\\
G_5^{\pm1}	& \rho^{\pm 1} \theta_1&\rho^{\pm 1}\theta_8	&\rho^{\pm 1} \theta_5	&\rho^{\pm1}\theta_7	&\rho^{\pm1}\theta_3	&\rho^{\pm1}\theta_9	&\rho^{\pm1}\theta_4	&\rho^{\pm1}\theta_2	&\rho^{\pm1}\theta_6	&\rho^{\pm1}\theta_{10}
\end{array}}
\end{equation}
\end{lemma}
\begin{proof}
The proof follows from an explicit computation applying the formulas in Lemmas~2.1.1(ii) and Lemma~2.1.2(viii)-(x) in \cite{MR1204828}.
\end{proof}
\begin{lemma}
\label{lem:J4vanish}
Under the action $\modvar \mapsto \mathcal{T} \cdot \modvar =\modvar^t$ we have
\begin{equation}
\label{eqn:I_action}
  \Big( \theta_1(\modvar) , \dots  , \theta_{10}(\modvar) , \Theta(\modvar) \Big) \ \mapsto \ \Big( \theta_1(\modvar) , \dots  , \theta_{10}(\modvar) , - \Theta(\modvar) \Big).
\end{equation}
\end{lemma}
\begin{proof}
The transformation for $\Theta(\modvar)$ was proven in \cite{MR1204828}*{Cor.~\!3.1.4}.
\end{proof}
\begin{lemma}
\label{lem:DAvanish}
Under the action $\modvar \mapsto M_i \cdot G_1 \cdot M_i^{-1} \cdot \modvar$ we have 
\begin{equation}
 \Big[ \theta_1(\modvar): \dots :  \theta_{10}(\modvar) \Big] \mapsto \Big[ (-1)^{\delta_{i,1}} \theta_1(\modvar): \dots :   (-1)^{\delta_{i,10}}   \theta_{10}(\modvar) \Big],
\end{equation} 
where $M_i \in \Gamma$ with $\det{(M_i)}=1$ and $1 \le i \le 10$, $\delta_{\mu,\nu}$ is the Kronecker delta function, and the matrices $M_i$ are given in the following table:
\begin{equation}
\scalemath{\MyScaleMedium}{
\begin{array}{r|r|rrrrrrrrrrr}
i& M_i & \multicolumn{11}{c}{[(-1)^{\delta_{i,1}}, \dots , (-1)^{\delta_{i,10}}]}\\
 \hline
 1&G_4G_3G_5G_4G_3G_2 & [&-1	,&1	,&1	,&1	,&1	,&1	,&1	,&1	,&1	,& 1] \\
 2&G_2G_5G_4G_3G_2 	& [&1	,&-1	,&1	,&1	,&1	,&1	,&1	,&1	,&1	,& 1] \\
 3&G_3G_5G_4G_3G_2 	& [&1	,&1	,&-1	,&1	,&1	,&1	,&1	,&1	,&1	,& 1] \\
 4&G_5G_4G_3G_2 		& [&1	,&1	,&1	,&-1	,&1	,&1	,&1	,&1	,&1	,& 1] \\
 5&G_3G_4G_3G_2 		& [&1	,&1	,&1	,&1	,&-1	,&1	,&1	,&1	,&1	,& 1] \\
 6&G_2 				& [&1	,&1	,&1	,&1	,&1	,&-1	,&1	,&1	,&1	,& 1] \\
 7&G_4G_3G_2 			& [&1	,&1	,&1	,&1	,&1	,&1	,&-1	,&1	,&1	,& 1] \\
 8&G_2G_3G_4G_3G_2 	& [&1	,&1	,&1	,&1	,&1	,&1	,&1	,&-1	,&1	,& 1] \\
 9&G_3G_2 				& [&1	,&1	,&1	,&1	,&1	,&1	,&1	,&1	,&-1	,& 1] \\
10&\mathbb{I}_2 			& [&1	,&1	,&1	,&1	,&1	,&1	,&1	,&1	,&1	,&-1] 
\end{array}}
\end{equation}
In particular, $\theta_{i}(\modvar)$ has simple zeros exactly on the $\Gamma_{\mathcal{T}}(1+i)$-orbit of the fixed locus of $M_i \cdot G_1 \cdot M_i^{-1}$ for $1 \le i \le 10$. \end{lemma}
\begin{proof}
The first part of the proof follows from Lemma~\ref{lem:Gtransfo}. The fact that it is a simple zero must only be proven for one theta function, say $\theta_{10}$. This was done  in \cite{MR1204828}*{Lemma~\!2.3.1}. 
\end{proof}
The groups $\Gamma_{\mathcal{T}}$ and $\Gamma_{\mathcal{T}}(1+i)$ have the index-two subgroups given by
\begin{equation}
\label{index-two}
\begin{split}
 \Gamma^{+}_{\mathcal{T}} =   \Big\lbrace g = G\, \mathcal{T}^{n} \in  \Gamma_{\mathcal{T}}  & \ \Big| \ n \in \lbrace 0,1 \rbrace, \; (-1)^n \, \det{G} =1 \Big\rbrace, \\
  \Gamma^{+}_{\mathcal{T}}(1+i)  & = \Gamma^{+}_{\mathcal{T}} \cap \Gamma_{\mathcal{T}}(1+i).
 \end{split}
\end{equation} 
Obviously, we have the following:
\begin{lemma}
\label{lem:generators}
The group $\Gamma^{+}_{\mathcal{T}}$ is generated by elements $G_1\mathcal{T}$ and $G_2, \dots, G_5$ where $G_k$ for $k=1,\dots, 5$ were given in Equation~\!(\ref{eqn:generators}). 
\end{lemma}
\par We now consider the quotient spaces $\mathbf{H}_2/\Gamma_{\mathcal{T}}(1+i)$ and $ \mathbf{H}_2/\Gamma^+_{\mathcal{T}}(1+i)$, and the Satake compactification
\begin{equation}
 \overline{\mathbf{H}_2/\Gamma_{\mathcal{T}}(1+i)} \;.
\end{equation}
For details on the construction of the Satake-Baily-Borel compactification we refer to \cites{MR0216035,MR0118775}. We define a holomorphic map
\begin{equation}
\label{eqn:map1a}
 \mathcal{F}':  \ \ \mathbf{H}_2  \to  \mathbb{P}^9, \quad \modvar  \mapsto  \Big[ \theta^2_1(\modvar) : \dots  : \theta^2_{10}(\modvar) \Big].
\end{equation}
It follows immediately from Theorem~\!\ref{thm1} that the map $\mathcal{F}'$ descends to a holomorphic map on the quotient space $\mathbf{H}_2/\Gamma_{\mathcal{T}}(1+i)$. Equations~\!(\ref{eqn:PlueckerRelations}) coincide with the quadratic relations between the even theta functions in Theorem~\!\ref{thm1} under the identification given by the map $\mathcal{F}'$. An analysis of the simple zeros of the theta functions in  \cite{MR1136204}*{Lemma~\!2.3.1} shows that the image is in fact contained in $\mathfrak{M}(2)$. Thus, we obtain a holomorphic map 
\begin{equation}
\label{eqn:map1b}
\begin{split}
 \mathcal{F}:  \ \ & \mathbf{H}_2/\Gamma_{\mathcal{T}}(1+i) \  \longrightarrow \ \mathfrak{M}(2) \subset \mathbb{P}^9,\\
  & \modvar  \ \mapsto \  \Big[t_1 : \dots : t_{10} \Big] =  \Big[ \theta^2_1(\modvar) : \dots  : \theta^2_{10}(\modvar) \Big] .
\end{split}
\end{equation}
We have the following:
\begin{theorem}[\cite{MR1204828}*{Thm.~3.2.1}]
\label{thm2}
The map $\mathcal{F}$ in Equation~(\ref{eqn:map1b}) extends to an isomorphism between the Satake compactification of $\mathbf{H}_2/\Gamma_{\mathcal{T}}(1+i)$ and $\overline{\mathfrak{M}(2)}$ given by
\begin{equation*}
\begin{split}
 \mathcal{F}:  \ \ & \overline{\mathbf{H}_2/\Gamma_{\mathcal{T}}(1+i)}\   \overset{\cong}{\longrightarrow}\ \overline{\mathfrak{M}(2)}  \subset \mathbb{P}^9.
 \end{split}
\end{equation*}
\end{theorem}
We also define a holomorphic map
\begin{equation}
\label{eqn:map2a}
 \mathcal{G}':  \ \ \mathbf{H}_2  \to  \mathbb{P}(1,\dots,1,2), \; \modvar  \mapsto  \Big[ \theta^2_1(\modvar) : \dots  : \theta^2_{10}(\modvar) : 2^2 3^{-3} 5^{-2}\Theta(\modvar) \Big].
\end{equation}
We have the following:
\begin{proposition}
The map $\mathcal{G}'$ descends to  a holomorphic map 
\begin{equation}
\label{eqn:map2b}
\begin{split}
 \mathcal{G}:  \ \ &  \mathbf{H}_2/\Gamma^+_{\mathcal{T}}(1+i) \   \longrightarrow \ \mathfrak{M}(2)^{+} \subset \mathbb{P}(1,\dots,1,2),\\
 & \modvar  \ \mapsto \  \Big[t_1 : \dots : t_{10}: R \Big] =  \Big[ \theta^2_1(\modvar) : \dots  : \theta^2_{10}(\modvar): 2^2 3^{-3} 5^{-2}\Theta(\modvar)\Big] .
 \end{split}
 \end{equation}
 Moreover, the following diagram commutes:
 \begin{equation}
\label{pic:modular_maps}
\xymatrix{
\mathbf{H}_2/\Gamma^+_{\mathcal{T}}(1+i)   	\ar[r]^-{\mathcal{G}}	\ar @{->>} [d]		&\mathfrak{M}(2)^{+}	\ar @{->>} [d]^{\operatorname{pr}}\\
 \mathbf{H}_2/\Gamma_{\mathcal{T}}(1+i)	\ar [r]^-{\mathcal{F}}           			&\mathfrak{M}(2)}	
\end{equation}
with $\mathcal{G} \circ \mathcal{T}= \imath \circ \mathcal{G}$ and $\operatorname{pr} \circ\, \imath =\operatorname{pr}$ and $\Gamma^+_{\mathcal{T}}(1+i)$ the index-two sub-group of $\Gamma_{\mathcal{T}}(1+i)$ defined in Equation~(\ref{index-two}).
\end{proposition}
\begin{proof}
For elements $g = G\,\mathcal{T} ^{n} \in \Gamma^+_{\mathcal{T}}(1+i)$ we have by definition $(-1)^n \det{G} =1$ and $\operatorname{sign}\!{(S_G)}=1$. It then follows from Theorem~\ref{thm1} that the map $\mathcal{G}'$ descends to a holomorphic map on the quotient space $\mathbf{H}_2/\Gamma^+_{\mathcal{T}}(1+i) $. Combined with the results in Theorem~\ref{thm2} this shows that the image is contained in $\mathfrak{M}(2)^{+}$. The branching locus of the covering $\mathbf{H}_2/\Gamma^+_{\mathcal{T}}(1+i) \to \mathbf{H}_2/\Gamma_{\mathcal{T}}(1+i)$ is given by the simple zeros of the modular form $\Theta(\modvar)$ of weight four which is unique up to constant; see~Theorem~\ref{thm1}. Moreover, the ratio of the right hand sides of Equation~(\ref{eqn:Tsqr}) and Equation~(\ref{eqn:PlueckerRelations2}) yields $\Theta(\modvar)^2/R^2=2^{-4}\cdot3^{6}\cdot5^2$ under the identification given by the holomorphic map $\mathcal{F}$. Thus, the branching locus is identical. The equivariance follows from Equation~(\ref{eqn:i_action}) and Equation~(\ref{eqn:I_action}) and the fact that their branching locus is identical.
\end{proof}
\subsection{The Satake sextic and ring of modular forms}
\label{ssec:mod_forms}
We introduce the following linear combinations of the degree-one invariants $t_i$ which we call the generalized \emph{level-two Satake coordinate functions} $x_1, \dots, x_6$. We set $x_1 + \dots + x_6=0$. Choosing three Satake roots out of the five roots $x_1, \dots, x_5$, we want to obtain all invariants $t_1, \dots, t_{10}$ by setting
\begin{equation}
\label{InvSatake2Theta}
\begin{array}{llllll}
-3\,t_9 & = x_1&+x_2&+x_3,\\
 \phantom{-}3\,t_8 & = x_1&+x_2&&+x_4,\\
-3\,t_6 &=x_1&+x_2	&&&+x_5,\\
\phantom{-}3\,t_5 &=x_1&	&+x_3&+x_4,\\
-3\,t_{10} & = x_1&&+x_3&&+x_5,\\
 \phantom{-}3\,t_7 & = x_1&&&+x_4&+x_5,\\
 -3\,t_3 &= 	&+x_2	&+ x_3	&+x_4,\\
 -3\,t_1 &= 	&+x_2	&+ x_3	&&+x_5,\\
-3\,t_4 &= 	&+x_2	&&+ x_4	&+x_5,\\
-3\,t_2 &= 	&&+x_3	&+ x_4	&+x_5.
\end{array}
\end{equation}
The $j$-th power sums $s_j$ are defined by $s_j = \sum_{k=1}^6 x_k^j$ for $j=1, \dots,6$. It can be easily checked using Equation~(\ref{InvSatake2Theta}) that $s_1=\sum_k x_k =0$. We combine the level-two Satake functions in a sextic curve, called the \emph{Satake sextic}, given by
\[
 \mathcal{S}(x) = \prod_{k=1}^6 (x -x_k) \;.
\]
The coefficients of the Satake sextic are polynomials in  $\mathbb{Z} \left[ \frac 1 2 , \frac 1 3 , s_2, s_3, s_4, s_5, s_6 \right]$. In fact, we obtain
\[
 \mathcal{S}(x) = x^6 + \sum_{k=1}^6 \frac{(-1)^k}{k!} b_k\, x^{6-k} \,
 \]
 where $b_k$ is the $k$-th Bell polynomials in the variables $\lbrace s_1, -s_2, 2!  s_3, - 3! s_4, 4! s_5, -5! s_6 \rbrace$. The following holds:
 \begin{lemma}\label{Satake6}
The generalized level-two Satake coordinate functions $x_1, \dots , x_6$ are the roots of the Satake  sextic $\mathcal{S}\in \mathbb{Z} \left[ \frac 1 2 , \frac 1 3 , s_2, s_3, s_4, s_5, s_6 \right]\![x]$  given by
\begin{equation}
\label{SatakeSextic}
\begin{split}
 \mathcal{S}(x) & = \mathcal{B}(x)^2-4 \, \mathcal{A}(x),\;  \\[8pt]
 \ \text{with} \qquad\mathcal{B}(x)& =x^3- \frac{s_2}{4} \, x - \frac{s_3}{6},\\[5pt]
 \mathcal{A}(x) &= \frac{4s_4-s_2^2}{64} x^2 - \frac{5s_2 s_3-12s_5}{240} x  + \frac{3s_2^3-4s_3^2-18s_2 s_4+24s_6}{576}.
 \end{split}
 \end{equation}
\end{lemma}
\begin{proof}
The proof follows from the explicit computation of the Bell polynomials using the relation $s_1=0$.
\end{proof}
\par We define new quantities $J_2, J_3, J_4, J_5, J_6$ by setting
\begin{equation}
\label{eqn:Jinvariants}
\begin{split}
J_2 = \dfrac{s_2}{12} , \qquad J_3 &= \dfrac{s_3}{12}, \qquad  J_4 =\dfrac{4s_4-s_2^2}{64}  \\
J_5 = \dfrac{5s_2 s_3-20s_5}{240}, \qquad s_6 &= \dfrac{3 s_2^3 - 4 s_3^2 -18 s_2 s_4+24 s_6}{576},
\end{split}
\end{equation}
such that
\begin{equation}
\label{Jinvariants}
\begin{split}
s_2 = 12 \,  J_2, \qquad s_3 & = 12 \, J_3, \qquad \ s_4 = 36 \, J_2^2+16 \, J_4,  \\
s_5 = 60 \, J_2 J_3-20 \, J_5, \qquad s_6 & = 108 \, J_2^3+144 \, J_2 J_4+24 \, J_3^2+24 \, J_6.
\end{split}
\end{equation}
We have the following:
\begin{lemma}
\label{lem:Jinvariance}
$J_2, J_3, J_4, J_5, J_6$ are polynomials over $\mathbb{Q}$ in $t_i$ for $1 \le i \le10$ that are invariant under the action of the permutation group $\mathrm{S}_6$ on the variables $t_i$ induced by permuting the lines of a six-line configuration.
\end{lemma}
\begin{proof}
Using Equations~(\ref{eqn:PlueckerRelations}) we can solve for $x_1, \dots, x_6$ using any five of the then $t_i$'s and obtain
\begin{equation}
\label{Satake2Theta}
\begin{array}{rrrrrr}
x_1 =&2\,t_1	&+2\,t_5	&-3\,t_6	&-t_7	&-t_8,\\
x_2 =&-t_1	&-t_5	&		&-t_7	&+2\, t_8,\\
x_3 =&-t_1	&+2\,t_5	&		&-t_7	&-t_8,\\
x_4 =&  -t_1	&-t_5	&+3\,t_6	&+2\,t_7	&+2\,t_8,\\
x_5 =& -t_1	&-t_5	&		&+2\,t_7	&-t_8,\\
x_6 =&2t_1	&-t_5	&		&-t_7	&- t_8.
\end{array}
\end{equation}
Plugging Equations~(\ref{Satake2Theta}) into the $j$-th power sums $s_j$ and, in turn, into Equations~(\ref{eqn:Jinvariants}) proves that $J_k$ for $2\le k \le 6$ are polynomials in $t_i$ with $1 \le i \le10$ and rational coefficients. One checks that for a set of generators of the permutation group of six elements $\mathrm{S}_6$, acting on the variables $t_i$ with $1 \le i \le 10$ as defined by permuting lines in Equations~(\ref{eqn:DOcoords}), the polynomials $J_k$ for $2\le k \le 6$ remain invariant.
\end{proof}
\par Notice that there are many notations in the literature for invariants of binary equations. We will show in Section~\ref{6Lrestricted} how our invariants $J_k$ for $2 \le k \le 6$, when restricted to a six-line configuration tangent to a common conic, are related to the Igusa invariants of a binary sextic which we denote by $I_2, I_4, I_6, I_{10}$; see Equation~(\ref{modd_restriced}). 
We have the following:
\begin{lemma}
The moduli space $\mathfrak{M}(2)^{+}$ of six lines in $\mathbb{P}^2$ embeds into the variety in $\mathbb{P}(1, 1, 1, 1, 1, 1, 2)$ with coordinates $[x_1: \dots: x_6: X]$ given by the equations $s_1=0$ and $X^2=4s_4-s_2^2$.
\end{lemma}
\begin{proof}
Given a point in the image, setting $x_6=-(x_1 + \dots +x_5)$, Equations~(\ref{InvSatake2Theta}) for $t_1, t_5, t_6, t_7, t_8$ constitute the inverse transformation to Equations~(\ref{Satake2Theta}). Moreover, one checks that $4s_4-s_2^2=(18 \, R)^2$ in Equation~(\ref{Eqn:moduli}).
\end{proof}
\begin{remark}
The sub-variety defined by $X=0$ comprises an algebraic variety in $\mathbb{P}^4$ given by $s_1=0$ and $s_2^2=4s_4$ known as Igusa's quartic. It corresponds to the moduli space of configurations of six-lines tangent to a conic and is closely related to the moduli space of genus-two curves with level-two structure. We will discuss the details in Section~\ref{6Lrestricted}.
\end{remark}
\par In terms of the invariants $J_k$ with $2\le k\le6$, the Satake sextic is given by
\begin{equation}
\label{SatakeSextic2}
\begin{array}{c}
 \mathcal{S}(x) =\mathcal{B}(x)^2 -4  \, \mathcal{A}(x),\\[8pt]
\ \text{with} \qquad  \mathcal{B}(x)=x^3- 3 \, J_2 x - 2 \, J_3, \quad \quad
 \mathcal{A}(x) = J_4 \, x^2 - J_5 \, x  + J_6.
 \end{array}
 \end{equation}
 One also checks that the square of the degree-two Dolgachev-Ortland invariant in Equation~(\ref{eqn:PlueckerRelations2}) is given by $J_4$, i.e., 
\begin{equation}
\label{eqn:R}
 R^2 = 2^{4} 3^{-4} J_4 \;.
\end{equation}
We introduce three more invariants: the discriminants of the Satake sextic $\mathcal{S}(x)$ and the quadratic polynomial $\mathcal{A}(x)$ which have degrees $30$ and $10$, respectively, as well as the resultant of the polynomials $\mathcal{A}(x)$ and $\mathcal{B}(x)$ which has degree $18$. By construction, they are all homogeneous polynomials in the invariants $\mathcal{J}_k$ for $2\le k \le 6$ with integer coefficients. One checks that
\begin{equation}
\label{eqn:decomp}
\begin{split}
 \operatorname{Disc} (\mathcal{A}) & = J_5^2 -4 \, J_4 J_6 =  2^{-4}3^{10}  \prod_{i=1}^{10} t_i ,\\
 \operatorname{Res} (\mathcal{A},\mathcal{B}) & = 9 \, J_2^2 J_4^2 J_6+6 \, J_2 \, J_3 \, J_4^2 J_5+4 \, J_3^2 J_4^3 +6 \, J_2 \, J_4 \, J_6^2\\
 &-3 \, J_2 \, J_5^2 \,J_6+6 \, J_3\, J_4 \, J_ 5 \, J_6-2 \,J_3 \, J_5^3+J_6^3.
 \end{split}
\end{equation}
For the discriminant of the Satake sextic, i.e., $\mathcal{S}(x)= \prod_{i<j} (x_i - x_j)^2$, we suppress the lengthy polynomial expression in terms of the $J_k$ for $k=2, \dots, 6$. We rather give the following formula in terms of modular forms on $\mathfrak{M}(2)$, namely
\begin{equation}\label{eq:DiscS}
\begin{array}{c}
 \operatorname{Disc} (\mathcal{S})   = 3^{30} \prod_{j=2}^{10} (t_1 -t_j)^{2} \\[6pt]
  \times  (t_2 -t_3)^{2}  (t_3 -t_4)^{2}  (t_4 -t_5)^{2}  (t_5 -t_6)^{2}   (t_2 -t_4)^{2}  (t_4 -t_6)^{2}.
 \end{array}
\end{equation}
We have the following:
\begin{corollary}
\label{lem:Js}
A configuration of six lines in $\mathbb{P}^2$ falls into cases (0) through (5) in Definition~\ref{def:sstable} if and only if the invariants in Equation~(\ref{eqn:Jinvariants}) satisfy
\begin{equation}
 ( J_4 , \  J_5, \ J_6 ) \neq (0,0,0) .
\end{equation}
For cases (1) through (5) in Definition~\ref{def:sstable}, we find the following equivalences (where all invariants not specified remain generic):
\begin{enumerate}
 \item [(1)] $\Leftrightarrow \ J_4 =0$,
 \item [(2)] $\Leftrightarrow \  \operatorname{Disc}(\mathcal{A}) =0$,
 \item [(3, 4)] $\Leftrightarrow \  \operatorname{Disc}(\mathcal{A}) =\operatorname{Res} (\mathcal{A},\mathcal{B}) =0$,
 \item [(5)] $\Leftrightarrow \ J_4 =J_5=0$,
\end{enumerate}
For cases (6a) and (6b), we find  $( J_4 , J_5, J_6 ) = (0,0,0)$.
\end{corollary}
\begin{proof}
One first checks that $ \operatorname{Disc}(\mathcal{A}) =\operatorname{Res} (\mathcal{A},\mathcal{B})=0$ implies $\operatorname{Disc}(\mathcal{S})  =0$. The proof follows the same strategy as the one applied in the proof of Lemma~\ref{lem:equiv}.  We explicitly compute the invariants $J_4 , \  J_5, \ J_6$ in terms of the moduli $a, b, c, d$ and then restrict to cases (1) through (5) in Definition~\ref{def:sstable}.  Moreover, $( J_4 , J_5, J_6 ) = (0,0,0)$ implies $\operatorname{Disc}(\mathcal{A}) = \operatorname{Disc}(\mathcal{S}) =\operatorname{Res} (\mathcal{A},\mathcal{B})= 0$.
\end{proof}
In light of Lemma~\ref{lem:Jinvariance} and Corollary~\ref{lem:Js} we can now define a moduli space of \emph{unordered} configurations of six lines in $\mathbb{P}^2$. We define $\mathfrak{M}$ to be the four-dimensional open complex variety given by
\begin{equation}
\label{modulispace}
\mathfrak{M} =  \Big\{ 
\left [  J_2 : \ J_3 : \  J_4: \ J_5 : \ J_6   \right ]   \in  \mathbb{P}(2,3,4,5,6) \ \Big\vert  \
( J_4 , \  J_5, \ J_6 ) \neq (0,0,0) 
\Big\},
\end{equation}
and we also set $\overline{\mathfrak{M}}=\mathbb{P}(2,3,4,5,6)$. 
\par By construction, the points in projective space $\mathbb{P}^9$ arising as image of the map $\mathcal{F}': \mathbf{H}_2  \to  \mathbb{P}^9$ in Equation~(\ref{eqn:map1a}), i.e., the points given by
 \begin{equation}
\label{eqn:map1aB}
 \mathcal{F}':  \ \ \mathbf{H}_2  \to  \mathbb{P}^9, \quad \modvar  \mapsto  \Big[ \theta^2_1(\modvar) : \dots  : \theta^2_{10}(\modvar) \Big],
\end{equation}
are invariant under the action of the sub-group $\Gamma(1+i)$ of level $(1+i)$; see Theorem~\ref{thm1}. As explained, there is a natural action of the permutation group of six elements $\mathrm{S}_6$ on the variables $t_i$ with $1 \le i \le 10$ induced by permuting the six lines. This action coincides with the action of $\Gamma/\Gamma(1+i)$.
\par Equation~(\ref{eq:DiscS}) provides a geometric characterization of the locus $\operatorname{Disc} (\mathcal{S})=0$ in $\mathfrak{M}(2)$. It turns out that the fifteen components of the vanishing locus are in one-to one correspondence with permutations of the from $\sigma_\alpha=(ij)(kl)(mn)$ where $(ij)=i \leftrightarrow j$ is the permutation of the $i$-th and $j$-th line. We have the following:
\begin{lemma}
\label{lem:Svanish} 
The vanishing locus $\operatorname{Disc} (\mathcal{S}) =0$ is the union of the $\Gamma_{\mathcal{T}}(1+i)$-orbits of the fixed loci of  $\modvar \mapsto S_j \cdot G_2 \mathcal{T} \cdot S_j^{-1} \cdot \modvar$ in $\mathfrak{M}(2)$ where $S_j \in \Gamma$ with $\det{(S_j)}=1$ and $1\le j \le 15$. The fixed loci, the elements $S_j$, and corresponding permutations $\sigma_j$ are given in the following table:
\begin{equation}
\label{eqn:iso_S6}
\scalemath{\MyScaleMedium}{
\begin{array}{r|r|c|rlrrlrrlr}
 j	& S_j 						& \sigma_j 	& \multicolumn{9}{c}{\text{fixed locus}}\\
 \hline
 1	&G_5G_4G_5G_3G_5G_2G_5G_4 	& (15)(26)(34)	& t_1 & = & t_2,&  t_5 &= -&t_9,& t_6 &= -&t_7 \\
 2	&G_5G_2G_5G_4				& (12)(35)(46)	& t_1 & = & t_3,&  t_5 &= -&t_{10}, & t_6 &= -&t_8 \\
 3	&G_5G_4G_3G_5G_2G_5G_4		& (13)(24)(56)	& t_1 & = & t_4,& t_7 &= -&t_{10},& t_8 &= -&t_9\\
 4	&G_3G_5G_2G_5G_4			& (15)(23)(46)	& t_1 & = & t_5,&  t_2 &= -&t_9,& t_3 &= -&t_{10}\\
 5	&G_4G_2G_5G_4				& (14)(26)(35)	& t_1 & = & t_6,& t_2 &= -&t_7,& t_3 &= -&t_8\\
 6	&G_5G_3G_5G_2G_5G_4		& (13)(26)(45)	& t_1 & = & t_7,&  t_2 &= -&t_6,& t_4 &= -&t_{10}\\
 7	&G_4G_3G_5G_2G_5G_4		& (16)(24)(35)	& t_1 & = & t_8,&  t_3 &= -&t_6,& t_4 &= -&t_9\\
 8	&G_4G_5G_4G_5G_4			& (15)(24)(36)	& t_1 & = & t_9,&  t_2 &= -&t_5,& t_4 &= -&t_8\\
 9	&G_4G_5G_4					& (13)(25)(46)	& t_1 & = & t_{10},& t_3 &= -&t_5,& t_4 &= -&t_7\\
 10	&G_2G_5G_4					& (14)(23)(56)	& t_2 & = & t_3,& t_7 &=&t_8,& t_9 &= &t_{10}\\
 11	&G_5G_4G_5G_4				& (12)(36)(45)	& t_2 & = & t_4,&  t_5&= &t_8,& t_6 &= &t_{10}\\
 12	&G_4						& (14)(25)(36)	& t_2 & = & t_8,&  t_3 &=&t_7,& t_4 &= &t_5\\
 13	&G_4G_5G_3G_5G_2G_5G_4		& (16)(23)(45)	& t_2 & = & t_{10},&  t_3 &=&t_9,& t_4 &= &t_6\\
 14	&G_5G_4						& (16)(25)(34)	& t_3 & = & t_4,&  t_5 &=&t_7,& t_6 &= &t_9\\
 15	&\mathbb{I}_4					& (12)(34)(56)	& t_5 & = & t_6,&  t_7 &=&t_9,& t_8 &= &t_{10}
\end{array}}
\end{equation}

\end{lemma}
\begin{proof}
The relation between the components of vanishing locus and  the fixed loci of the transformations $\modvar \mapsto S_j \cdot G_2 \mathcal{T} \cdot S_j^{-1} \cdot \modvar$ with $1\le j \le 15$ follow from Equation~(\ref{eq:DiscS}) and Lemma~\ref{eqn:Gtransfo}. The relation between permutations acting on $t_i$ with $1 \le i \le 10$ and the listed fixed loci is checked directly using Equations~(\ref{eqn:DOcoords}) 
\end{proof}
\noindent
We denote the six-line configuration discussed in Lemma~\ref{lem:Svanish} by (2b), adding to cases (0) through (5) in Definition~\ref{def:sstable}. Equation~(\ref{eqn:iso_S6}) provides the explicit from of the isomorphism $\Gamma/\Gamma(1+i) \cong \mathrm{S}_6$.  We have the following:
\begin{corollary}
\label{cor:VanishingLoci}
The following vanishing loci are fixed loci of elements in $\Gamma_{\mathcal{T}} \backslash \Gamma^+_{\mathcal{T}}$:
\begin{enumerate}
\item[(1)] The locus $J_4=0$ is the fixed locus of $\mathcal{T} \in \Gamma_{\mathcal{T}}$.
\item[(2)] The locus $\operatorname{Disc}(\mathcal{A}) =0$ is the union of the fixed loci of $M_i\cdot G_1\cdot M_i^{-1} \in \Gamma_{\mathcal{T}}$ with $M_i \in \Gamma^+_{\mathcal{T}}$ and $1\le i \le 10$ given in Lemma~\ref{lem:DAvanish}.
\item[(2b)] The locus $\operatorname{Disc}(\mathcal{S}) =0$ is the union of the fixed loci of $S_j\cdot G_2\mathcal{T}\cdot S_j^{-1} \in \Gamma_{\mathcal{T}}$ with $S_j \in \Gamma^+_{\mathcal{T}}$ and $1\le j \le 15$ given in Lemma~\ref{lem:Svanish}.
\end{enumerate}
\end{corollary}
\begin{proof}
Parts (1), (2)  follow from Equations~(\ref{eqn:R}) and (\ref{eqn:decomp}) when using Lemma~\ref{lem:J4vanish} and Lemma~\ref{lem:DAvanish}. Part (3) follows from Lemma~\ref{lem:Svanish}.
\end{proof}
\par If we plug into the expressions for $J_k$ in Equation~(\ref{eqn:Jinvariants}) for $k=2, 3, 4, 5, 6$, the theta functions $t_i=\theta^2_i(\modvar)$ for $1 \le i \le 10$ of Theorem~\ref{thm1}, we obtain modular forms which we will denote by $J_2(\modvar), J_3(\modvar), J_4(\modvar), J_5(\modvar), J_6(\modvar)$. We have the following:
\begin{lemma}
The functions $J_k(\modvar)$ are modular forms of weight $2k$ relative to $\Gamma_{\mathcal{T}}$ with character $\chi_{2k}(g)=\det{(G)}^k$ for all $g \in\Gamma_{\mathcal{T}}$.
\end{lemma}
\begin{proof}
It follows from Lemma~\ref{lem:Jinvariance} that $J_k(\modvar)$ are homogeneous polynomials of degree $k$ in $t_i=\theta^2_i(\modvar)$ for $1 \le i \le 10$.  Using Theorem~\ref{thm1}, we conclude that $J_k(\modvar)$ are modular forms of weight $2k$ relative to $ \Gamma_{\mathcal{T}}(1+i)$ with character $\chi_{2k}(g)=\det{(G)}^k$. The isomorphism $\Gamma/\Gamma(1+i) \cong \operatorname{Sp}_4(\mathbb{Z}/2\mathbb{Z})$ extends the group homomorphism obtained from the projection
\begin{equation}
  \mathbb{Z}[i] \mapsto  \mathbb{Z}[i] / (1+i) \mathbb{Z}[i]  \cong \mathbb{Z}/2\mathbb{Z} .
\end{equation}
On the other hand, there is a group isomorphism $\operatorname{Sp}_4(\mathbb{Z}/2\mathbb{Z}) \cong \mathrm{S}_6$. We showed in Lemma~\ref{lem:Gtransfo} that this action coincides with the natural action of the permutation group of six elements $\mathrm{S}_6$ on the variables $t_i$ with $1 \le i \le 10$ due to Equations~(\ref{eqn:DOcoords}). Since the invariants $J_k$ are polynomials in the symmetric polynomials $s_k$ with $2 \le k \le 6$ given in Equation~(\ref{eqn:Jinvariants}), hence invariant under the action of $\mathrm{S}_6$, 
$J_k(\modvar)$ are modular forms of weight $2k$ relative to the full modular group $ \Gamma_{\mathcal{T}}$.
\end{proof}
We have the following: 
 \begin{theorem}
\label{thm4}
The graded ring of modular forms relative to $\Gamma^+_{\mathcal{T}}$ of even characteristic, i.e., with character $\chi_{2k}(g)=\det{(G)}^k$, is generated over $\mathbb{C}$ by the five algebraically independent modular forms $J_k(\modvar)$ of weight $2k$ with $k=2, \dots, 6$.
\end{theorem}
\begin{proof}
 It follows from~\cite{MR2682724}*{Thm.~\!1} that the ring of modular forms relative to $\Gamma_{\mathcal{T}}$ is generated by five modular forms of weights $4, 6, 8, 10, 12$. For arguments $\modvar=\tau$ invariant under $\mathcal{T}$, the functions $J_k(\modvar)$ for $k=2, 3, 4, 5, 6$ descend to Siegel modular forms of even weight. In fact, we will check in Equation~(\ref{modd_restriced}) that under the restriction from $\mathbf{H}_2/\Gamma_{\mathcal{T}}$ to $\mathbb{H}_2/\operatorname{Sp}_4(\mathbb{Z})$, we obtain
\begin{equation}
\label{modd_restricedb}
\scalemath{\MyScaleMedium}{
\begin{array}{c}
\Big[ J_2(\modvar) : J_3(\modvar) : J_4(\modvar): J_5(\modvar) : J_6(\modvar)  \Big] 
=  \Big[ \psi_4(\tau) : \psi_6(\tau) : 0: 2^{12} 3^5 \chi_{10}(\tau) : 2^{12}3^6 \chi_{12}(\tau)\Big].
\end{array}}
\end{equation}
Here, $\psi_4$, $\psi_6$, $\chi_{10}$, $\chi_{12}$ are Siegel modular forms of respective weights $4, 6, 10, 12$ that, as Igusa proved in~\cites{MR0229643, MR527830}, generate the ring of Siegel modular forms of even weight. Thus, the forms $J_k(\modvar)$ for $k=2, 3, 5, 6$ must be independent.  After looking at some Fourier coefficients to ensure that $J_4$ is not identically zero, we adjoin $J_4$ as a fifth form to the ring generated by $J_k$ for $k=2, 3, 5, 6$. The fundamental theorem of symmetric polynomials establishes the power sums as an algebraic basis for the space of symmetric polynomials. Therefore, $J_k$ for $k=2, 3, 4, 5, 6$ are algebraically independent. Moreover, we check that $J_4(\modvar)=(\Theta(\modvar)/15)^2$ using Theorem~\ref{thm1}. $\Theta$ does \emph{not} have the character $\chi(g)=\det(G)$ for all $g \in\Gamma_{\mathcal{T}}$. It follows that $J_4$ cannot be decomposed further as a modular form relative to $\Gamma^+_{\mathcal{T}}$ with even characteristic.
\end{proof}
We define the holomorphic map
\begin{equation}
\label{eqn:map3a}
 \mathcal{H}':  \mathbf{H}_2  \to  \mathbb{P}^9, \quad \modvar  \mapsto   \Big[  J_2(\modvar) : \ J_3(\modvar) : \  J_4(\modvar): \ J_5(\modvar) : \ J_6(\modvar)   \Big ].
\end{equation}
We have the following:
\begin{corollary}
The map $\mathcal{H}'$ descends to  a holomorphic map 
\begin{equation}
\label{eqn:map3b}
\begin{split}
 \mathcal{H}:  \ \ &  \mathbf{H}_2/\Gamma_{\mathcal{T}} \   \longrightarrow \    \mathfrak{M}  \subset \mathbb{P}(2,3,4,5,6),\\
 & \modvar  \ \mapsto \    \Big[  J_2(\modvar) : \ J_3(\modvar) : \  J_4(\modvar): \ J_5(\modvar) : \ J_6(\modvar)   \Big ].
 \end{split}
\end{equation}
The map $\mathcal{H}$ in Equation~(\ref{eqn:map3b}) extends to an isomorphism between the Satake compactification of $\mathbf{H}_2/\Gamma_{\mathcal{T}}$ and $\overline{\mathfrak{M}}$ given by
\begin{equation*}
\begin{split}
 \mathcal{H}:  \ \ & \overline{\mathbf{H}_2/\Gamma_{\mathcal{T}}}\   \overset{\cong}{\longrightarrow}\ \overline{\mathfrak{M}}  \subset \mathbb{P}(2,3,4,5,6).
 \end{split}
\end{equation*}
\end{corollary}
\begin{proof}
Using Theorem~\ref{thm4} and Corollary~\ref{lem:Js} it follows that $\mathcal{H}'$ descends to the holomorphic map $\mathcal{H}$ as stated. By construction the Satake compactification of $\mathbf{H}_2/\Gamma_{\mathcal{T}}$ is given by
\begin{equation}
  \operatorname{Proj} \mathbb{C}\Big[ J_2, J_3, J_4, J_5, J_6  \Big]\cong \mathbb{P}(2,3,4,5,6).
\end{equation}
\end{proof}
\section{K3 surfaces associated with double covers of six lines}
\label{sec:6lines}
In this section, we discuss several Jacobian elliptic fibrations on the K3 surface associated with configurations of six lines $\ell_i$ in $\mathbb{P}^2$ with $i=1,\dots,6$, no three of which are concurrent.  The double cover branched along six lines $\ell_i$ with $i=1,\dots,6$ given in terms of variables $z_1, z_2, z_3$ of $\mathbb{P}^2$  is the solution of
\begin{equation}
\label{six-lines-cover}
z_4^2 = \ell_1\ell_2\ell_3\ell_4\ell_5\ell_6 
\end{equation}
with $[z_1:z_2:z_3:z_4]\in \mathbb{P}(1,1,1,3)$. It is well known that the minimal resolution is a K3 surface of Picard-rank $16$ which we will always denote by $\mathcal{Y}$. In \cite{MR2254405} Kloosterman classified all Jacobian elliptic fibrations on $\mathcal{Y}$, i.e., elliptic fibrations $\pi^{\mathcal{Y}}: \mathcal{Y} \to \mathbb{P}^1$ together with a section $\sigma: \mathbb{P}^1 \to \mathcal{Y}$ such that $\pi^{\mathcal{Y} }\circ \sigma = 1$. We will construct explicit Weierstrass models for three of these Jacobian elliptic fibrations which we call the \emph{natural} fibration, the \emph{base-fiber-dual} of the natural fibration, and the \emph{alternate} fibration.
\subsection{The natural fibration}
\label{ssec:natural_fibration}
Configurations of six lines no three of which are concurrent have four homogeneous moduli which we will denote by $a, b, c, d$. These moduli can be constructed as follows: the six lines $\ell_i$ for $1\le i \le 6$ can always be brought into the form
\begin{equation}
\label{lines}
 \ell_1=z_1, \ell_2 = z_2, \ell_3= z_3,  \ell_4=z_1 + z_2 + z_3,  \ell_5=z_1+ a z_2 + b z_3, \ell_6=z_1+ c z_2 + d z_3.
\end{equation}
The matrix $\mathbf{A}$ defined in Section~\ref{sec:invariants} for this six-line configuration is
\begin{equation}
\mathbf{A} = \left( \begin{array}{cccccc} 1 & 0 & 0 & 1 & 1 & 1 \\ 0 & 1 & 0 & 1 & a & c \\ 0 & 0 & 1 & 1 & b & d \end{array}\right) \;,
\end{equation}
with Dolgachev-Ortland coordinates given by
\begin{equation}
\label{compare1}
\scalemath{\MyScaleMedium}{
\begin{array}{c}
 t_1=a(d-1), \ t_2=b-a, \ t_3=d-c, \  t_4=c(b-1), \ t_5=b(c-1), \\[4pt]
 t_6=d(a-1), \ t_7=d-b, \ t_8=c-a, \ t_9=ad-bc, \ t_{10}=ad-bc-a+b+c-d, \\[4pt]
 R = -abc+abd+acd-bcd-ad+bc.
\end{array}}
\end{equation}
The six lines intersect as follows:
\begin{equation}
\label{tab:6lines}
\scalemath{\MyScaleTiny}{
\begin{array}{c|c|c|c|c|c|c}
		& \ell_1	& \ell_2	& \ell_3	& \ell_4		& \ell_5			& \ell_6			\\ \hline &&&&&& \\[-0.9em]
\ell_1	& -		& [0:0:1]	& [0:1:0]	& [0:1:-1]		& [0:b:-a]			& [0:d:-c]			\\ &&&&&& \\[-0.9em]
\ell_2	& [0:0:1]	& -		& [1:0:0]	& [1:0:-1]		& [-b:0:1]			& [-d:0:1]			\\ &&&&&& \\[-0.9em]
\ell_3	& [0:1:0]	& [1:0:0]	& -		& [1:-1:0]		& [-a:1:0]			& [-c:1:0]			\\ &&&&&& \\[-0.9em]
\ell_4	& [0:1:-1]	& [1:0:-1]	& [1:-1:0]	& -			& [a-b:b-1:1-a]		& [c-d:d-1:1-c]		\\ &&&&&& \\[-0.9em]
\ell_5	& [0:b:-a]	& [-b:0:1]	& [-a:1:0]	& [a-b:b-1:1-a]	& -				& [ad-bc:b-d:c-a]	\\ &&&&&& \\[-0.9em]
\ell_6	& [0:d:-c]	& [-d:0:1]	& [-c:1:0]	& [c-d:d-1:1-c]	& [ad-bc:b-d:c-a]	& -				\\ &&&&&& \\[-0.9em]
\end{array}}
\end{equation}
\smallskip
Setting 
\begin{equation}
z_1 = \frac{u(u+1)(au+b)(cu+d)}{X-u(au+b)(cu+d)}, \quad z_2=u, \quad z_3=1
\end{equation}
in Equation~(\ref{six-lines-cover}), it is transformed into the Weierstrass equation
\begin{equation}
\label{eqns_Kummer}
  Y^2  = X \Big(X - 2 \,u \big(\mu(u)-\nu(u)\big) \Big)  \Big(X - 2 \,u \big(\mu(u)+\nu(u)\big)  \Big), 
\end{equation} 
with discriminant
\begin{equation}
 \Delta_{\mathcal{Y}}(u) =2^8  u^6 \mu(u)^2 \, \Big(\mu(u)^2-\nu(u)^2\Big)^2
\end{equation}
and
\begin{equation}
\label{coeffs}
\begin{split}
 2(\mu-\nu)  & = \Big(au +b\Big)\,  \Big((c-1) \, u+ (d-1) \Big) ,\\ 
 2(\mu+\nu)  & = \Big(c u +d\Big)\,  \Big((a-1) \, u + (b-1) \Big).
\end{split}
\end{equation}
In this way, the K3 surfaces associated with the double cover of $\mathbb{P}^2$ branched along the union of six lines, no three of which are concurrent, are equipped with an elliptical fibration $\pi^{\mathcal{Y}}_{\mathrm{nat}}: \mathcal{Y} \to \mathbb{P}^1$ with section $\sigma$ given by the point at infinity and with a fiber $\mathcal{Y}_{u}$ given by the minimal Weierstrass equation~(\ref{eqns_Kummer}). We call this fibration the \emph{natural fibration}. Three two-torsion sections are obvious from the explicit Weierstrass points in Equation~(\ref{eqns_Kummer}). The following is immediate:
\begin{lemma}
\label{lem:EllLeft}
Equation~(\ref{eqns_Kummer}) defines a Jacobian elliptic fibration $\pi^{\mathcal{Y}}_{\mathrm{nat}}: \mathcal{Y} \to \mathbb{P}^1$ with six singular fibers of type $I_2$,  two singular fibers of type $I_0^*$, and the Mordell-Weil group of sections $\operatorname{MW}(\pi^{\mathcal{Y}}_{\mathrm{nat}})=(\mathbb{Z}/2\mathbb{Z})^2$.
\end{lemma}
One checks that the Weierstrass model in Equation~(\ref{eqns_Kummer}) is minimal if and only if the configuration of six lines falls into cases (0) through (5) in Definition~\ref{def:sstable}, but not into case (6). In the latter case, the singularities of  Equation~(\ref{eqns_Kummer}) are not canonical singularities.  The following proposition was given in \cite{MR1877757}:
\begin{proposition}[\cite{MR1877757}]
\label{lem_6lines}
For generic parameters $a, b, c, d$, the K3 surface $\mathcal{Y}$  has the transcendental lattice $T_{\mathcal{Y}} \cong H(2) \oplus H(2) \oplus \langle -2 \rangle^{\oplus 2}$.
\end{proposition}
\begin{remark}
\label{rem:GKZ}
We choose $z_1=1$ as an affine chart for Equation~(\ref{six-lines-cover}) with lines given by Equations~(\ref{lines}), and the holomorphic two-form $dz_2 \wedge dz_3/z_4$. After relabelling variables, the period of the holomorphic two-form for the family of K3 surfaces $\mathcal{Y}(a,b,c,d)$ over a transcendental two-cycle $\Sigma \in T_{\mathcal{Y}}$ is given by
\begin{equation}
\label{GKZperiod}
 \iint_\Sigma \; \frac{dz_1}{\sqrt{z_1}} \;  \frac{dz_2}{\sqrt{z_2}} \; \frac{1}{\sqrt{(1+z_1+z_2)(1+a z_1+b z_2)(1+c z_1+d z_2)}} \;.
\end{equation}
We set
\begin{equation}
\begin{split}
 \alpha&=(\alpha_1,\alpha_2,\alpha_3)=\left(-\frac{1}{2}, -\frac{1}{2},-\frac{1}{2}\right), \quad \beta=(\beta_1,\beta_2)=\left(-\frac{1}{2},-\frac{1}{2}\right) 
\end{split}
\end{equation}
and
\begin{equation}
P_1 = 1+z_1+z_2, \quad P_2=1+a z_1+b z_2, \quad P_3=1+c z_1+d z_2.
\end{equation}
We observe that $\forall \, i:  \alpha_i \not \in \mathbb{Z}$, $\forall \, j:  \beta_j \not \in \mathbb{Z}$,  and $\sum_{i} \alpha_i + \sum_{j} \beta_j\not \in \mathbb{Z}$, and write the periods in the form
\begin{equation}
 F_{\Sigma}\Big(\alpha,\beta,\{P_i\}\Big|\, a,b,c,d\Big)=\iint_\Sigma \; \frac{dz_1}{z_1^{-\beta_1}} \;  \frac{dz_2}{z_2^{-\beta_2}}  \; \prod_{i=1}^{3}  \, P_i^{\alpha_i}\;.
\end{equation}
This identifies the periods as so called $\mathcal{A}$-hypergeometric functions that satisfy a system of linear differential equations known as non-resonant GKZ system \cite{MR1080980}. The particular GKZ system satisfied by the periods in Equation~(\ref{GKZperiod}) is a system of differential equations of rank six in four variables known as Aomoto-Gel'fand Hypergeometric System of Type $(3, 6)$. It was studied in great detail in \cites{MR1204828,MR1136204, MR1233442, MR1267602, MR2683512}. 
\par A basis of the solutions $F_1, \dots F_6$ defines a map from the Grassmanian $\operatorname{Gr}(3,6;\mathbb{C})$ into the projective space $\mathbb{P}^5$, more precisely into the domain in $\mathbb{P}^5$ cut out by the Hodge-Riemann relations. The period map is equivariant with respect to the action of $\Gamma_{\mathcal{T}}(1+i)$ on the domain and the monodromy group on the image. In fact, the map $[F_1 : \dots : F_6]\in \mathbb{P}^5$ is a multi-valued vector function from the moduli space $\mathfrak{M}(2)$ to the period domain acted upon by a monodromy group moving the branching locus of six lines around. The map $\mathcal{F}$ in Equation~(\ref{eqn:map1b}) is the inverse of this period mapping. 
\par We obtain a new map by quotienting further by the permutation group $\mathrm{S}_6$. The monodromy group is then generated by reflections and transformations caused by permuting the lines in a six-line configuration. The resulting map is the period map for a family of K3 surfaces $\mathcal{X}$ closely related to the K3 surfaces $\mathcal{Y}$ discussed  in Section~\ref{sec:VGSP}.
\end{remark}
We have the following:
\begin{corollary}
Switching the roles of base and fiber in Equation~(\ref{eqns_Kummer}) defines a second Jacobian elliptic fibration $\check{\pi}^{\mathcal{Y}}_{\mathrm{nat}}: \mathcal{Y} \to \mathbb{P}^1$ with 12 singular fibers of type $I_1$, a fiber of type $I_4$, a fiber of type $I_8$, and $\operatorname{MW}(\check{\pi}^{\mathcal{Y}}_{\mathrm{nat}})_{\mathrm{tor}}=\{ \mathbb{I} \}$ and $\operatorname{rk} \operatorname{MW}(\check{\pi}^{\mathcal{Y}}_{\mathrm{nat}})=4$ and $\det\operatorname{discr} \operatorname{MW}(\check{\pi}^{\mathcal{Y}}_{\mathrm{nat}})=1/2$.
\end{corollary}
\noindent
The elliptic fibration appears in the list of all Jacobian elliptic fibrations in \cite{MR2254405}. We call this fibration the \emph{base-fiber-dual} of the natural fibration.
\begin{proof}
In~\cite{MalmendierClingher:2018} it was proved that a Weierstrass model of the form given by Equation~(\ref{eqns_Kummer}) is equivalent to the genus-one fibration
\begin{equation}
 \eta^2 = \nu(u)^2 \xi^4 + 2 \, u \, \mu(u) \, \xi^2 + u^2,
\end{equation}
with one apparent rational point. Since $\mu, \nu$ are given as polynomials in $u$ in Equation~(\ref{coeffs}), the equation can be rewritten as
\begin{equation}
 \eta^2 = A(\xi) \, u^4 + B(\xi) \, u^3 + C(\xi) \, u^2 + D(\xi) \, u + E(\xi)^2,
\end{equation}
where $A, B, C, D, E$ are polynomials in $\xi$. Because there always is the rational point $(u,\eta)=(0,E(\xi))$, it can be brought into the Weierstrass form
\begin{equation}
\label{fib1}
 y^2 = 4 \, x^3 - g_2(\xi) \, x - g_3(\xi),
\end{equation}
with 
\begin{equation}
\begin{split}
 g_2 & = \frac{16}{3} \, C(\xi)^2+64 \, A(\xi) \, E(\xi)^2-16\, B(\xi) \, D(\xi),\\
 g_3 & = -\frac{64}{27} \, C(\xi)^3+\frac{256}{3} A(\xi) \, C(\xi)\, E(\xi)^2+\frac{32}{3} B(\xi) \, C(\xi) \, D(\xi)\\
 & \qquad-32\ A(\xi) \, D(\xi)^2-32\, B(\xi)^2 \, E(\xi)^2.
\end{split}
\end{equation}
This is a Weierstrass model with 12 singular fibers of type $I_1$, a fiber of type $I_4$, a fiber of type $I_8$, and $\operatorname{MW}(\check{\pi}^{\mathcal{Y}}_{\mathrm{nat}})_{\mathrm{tor}}=\{ \mathbb{I} \}$. In Proposition~\ref{lem_6lines} we showed that the transcendental lattice of the K3 surfaces $\mathcal{Y}$ has rank six and is given by
\begin{equation}
\begin{split}
 T_{\mathcal{Y}}& \cong H(2) \oplus H(2) \oplus \langle -2 \rangle^{\oplus 2}.
\end{split}
\end{equation}
Therefore, the determinant of the discriminant group for the rank-six lattice $T_{\mathcal{Y}}$ is $\det(\operatorname{discr}T_{\mathcal{Y}})=2^6$. The root lattice associated with the singular fibers in Equation~(\ref{fib1}) has rank ten and determinant $2^5$. The claim follows.
\end{proof}
\subsection{The alternate fibration}
In the list of all possible fibrations on the K3 surface $\mathcal{Y}$ associated with the double cover branched along the union of six lines given in \cite{MR2254405} we find the following fibration which we call the \emph{alternate} fibration:
\begin{corollary}[\cite{MR2254405}]
\label{cor:Kum_alternate}
On the K3 surface $\mathcal{Y}$ there is a Jacobian elliptic fibration $\pi^{\mathcal{Y}}_{\mathrm{alt}}: \mathcal{Y} \to \mathbb{P}^1$ with six singular fibers of type $I_2$,  two singular fibers of type $I_1$, one singular fiber of type $I_4^*$, and the Mordell-Weil group of sections $\operatorname{MW}(\pi^{\mathcal{Y}}_{\mathrm{alt}})=\mathbb{Z}/2\mathbb{Z}$.
\end{corollary}
The alternate fibration on the K3 surface $\mathcal{Y}$ can be obtained explicitly from Equation~(\ref{six-lines-cover}). In fact, we obtained the Weierstrass model for this fibration using a 2-neighbor-step procedure applied twice, a method explained in \cite{MR3263663}, starting with the natural fibration.  The details of this computation will be published in another forthcoming article \cite{ClingherDoran:2019}. We have the following:
\begin{theorem}
\label{thm:6lines_alt}
The K3 surface $\mathcal{Y}$ associated with the double cover branched along six lines admits the Jacobian elliptic fibration $\pi^{\mathcal{Y}}_{\mathrm{alt}}: \mathcal{Y} \to \mathbb{P}^1$ with the Weierstrass equation
\begin{equation}
\label{eqns_Kummer_alt}
  Y^2   = X \Big(X^2 - 2 \, \mathcal{B}(t)  X + \mathcal{B}(t)^2 - 4 \, \mathcal{A}(t) \Big), 
\end{equation} 
with discriminant
\begin{equation}
\label{eqns_Kummer_alt_discr}
   \Delta_{\mathcal{Y}}(t) = 16 \, \mathcal{A}(t) \, \Big(\mathcal{B}(t)^2 - 4 \, \mathcal{A}(t)\Big)^2 ,
\end{equation} 
and
\begin{equation}
\begin{split}
 \mathcal{B}(t)=t^3- 3 \, J_2 \, t - 2 \, J_3, \ & \qquad 
 \mathcal{A}(t) = J_4 \, t^2 - J_5 \, t  + J_6,
 \end{split}
 \end{equation}
where $J_k$ for $k=2, \dots, 6$ are the invariants of the configuration of six lines defined in Equations~(\ref{Jinvariants}).
\end{theorem}
\begin{proof}
The invariants $J_2, J_3, J_4, J_5, J_6$ defined by Equations~(\ref{Jinvariants}) are determined by the symmetric polynomials in terms of the five degree-one invariants $t_1, t_5, t_6, t_7, t_8$ which in turn are given in terms of the affine moduli $a, b, c, d$ using Equations~(\ref{compare1}). These  affine moduli $a, b, c, d$ were defined by arranging the six lines to be in the form of Equation~(\ref{lines}). On the other hand, the 2-neighbor-step procedure, when applied twice, gives the Weierstrass model in Equation~(\ref{eqns_Kummer_alt}) with 
\begin{equation}
 \mathcal{B}(t)=t^3-3\, J'_2 \, t - 2\, J'_3, \qquad
 \mathcal{A}(t) = J'_4 \, t^2 - J'_5 \, t  + J'_6.
 \end{equation}
We computed the coefficients $J'_i$ following the same procedure as the one outlined in \cite{MR3263663} for a general Kummer surface using a computer algebra system. At the end of the computation, the coefficients $J'_i$ for $2 \le 6$ are obtained directly in terms of the affine moduli $a, b, c, d$. The resulting expressions for the coefficients are then given in Equations~(\ref{QP_inv}). One easily checks that $J_i=J'_i$ for $2 \le i \le 6$.
\end{proof}
\section{The Van~Geemen-Sarti partner}
\label{sec:VGSP}
To extend the notion of geometric two-isogeny to Picard rank 16, we replaced the Kummer surfaces by the K3 surface $\mathcal{Y}$ associated with the double cover branched along the union of six lines discussed in Section~\ref{sec:6lines}. The Shioda-Inose surface is now replaced by a K3 surface $\mathcal{X}$ introduced by Clingher and Doran in \cite{MR2824841}. The K3 surface occurs as the general member of a six-parameter family of K3 surfaces polarized by the lattice $H \oplus E_7(-1) \oplus E_7(-1)$. Each K3 surface in the family carries a special Nikulin involution, called \emph{Van~\!Geemen-Sarti involution}, such that quotienting by this involution and blowing up fixed points recovers a double-sextic surface. 
\subsection{A Six-Parameter Family of K3 Surfaces}
Let $(\alpha, \beta, \gamma, \delta , \varepsilon, \zeta) \in \mathbb{C}^6 $. We consider the projective quartic surface 
$ {\rm Q}(\alpha, \beta, \gamma, \delta , \varepsilon, \zeta) \subset  \mathbb{P}^3(x,y,z,w) $
defined by the homogeneous equation:   
\begin{equation}
\label{mainquartic}
\scalemath{\MyScaleMedium}{
\begin{array}{c}
\mathbf{Y}^2 \mathbf{Z} \mathbf{W}-4 \mathbf{X}^3 \mathbf{Z}+3 \alpha \mathbf{X} \mathbf{Z} \mathbf{W}^2+\beta \mathbf{Z} \mathbf{W}^3+\gamma \mathbf{X} \mathbf{Z}^2 \mathbf{W}- \frac{1}{2} \left (\delta \mathbf{Z}^2 \mathbf{W}^2+ \zeta \mathbf{W}^4 \right )+ \varepsilon \mathbf{X} \mathbf{W}^3 = 0.
\end{array}}
\end{equation}
The family in Equation~(\ref{mainquartic}) was first introduced by Clingher and Doran in \cite{MR2824841} as a generalization of the Inose quartic introduced in \cite{MR578868}. We denote by $\mathcal{X}(\alpha, \beta, \gamma, \delta , \varepsilon, \zeta)$ the smooth complex surface obtained as the minimal resolution of  ${\rm Q}(\alpha, \beta, \gamma, \delta , \varepsilon, \zeta)$. We have the following:
\begin{theorem}
\label{thm:polarization}
Assume that $(\gamma, \delta) \neq (0,0)$ and $ (\varepsilon, \zeta) \neq (0,0)$. Then, the surface 
$\mathcal{X}(\alpha, \beta, \gamma, \delta , \varepsilon, \zeta)$ obtained as the minimal resolution of 
${\rm Q}(\alpha, \beta, \gamma, \delta , \varepsilon, \zeta)$ is a K3 surface endowed with a canonical 
$H \oplus E_7(-1) \oplus E_7(-1) $ lattice polarization. 
\end{theorem}
\begin{proof}
The conditions imposed on the pairs $(\gamma, \delta)$ and $(\varepsilon, \zeta)$ ensure that singularities of  $ {\rm Q}(\alpha, \beta, \gamma, \delta , \varepsilon, \zeta)$ are rational double points. This fact, in connection with the degree of Equation~(\ref{mainquartic}) being four, guarantees that the minimal resolution  $\mathcal{X}(\alpha, \beta, \gamma, \delta , \varepsilon, \zeta)$ is a K3 surface.   
\par Note that the quartic ${\rm Q}(\alpha, \beta, \gamma, \delta , \varepsilon, \zeta)$ has two special singularities at the following points:
$$ {\rm P}_1 = [0,1,0,0], \ \ \ \ {\rm P}_2 = [0,0,1,0].  $$
One verifies that the singularity at ${\rm P}_1$ is a rational double point of type ${\rm A}_9$  if $ \varepsilon \neq 0$, and of type ${\rm A}_{11}$ if $ \varepsilon = 0$. The singularity at ${\rm P}_2$ is of type ${\rm A}_5$  if $ \gamma  \neq 0$, and of type ${\rm E}_6$ if $ \gamma  = 0$. For a generic sextuple $ (\alpha, \beta, \gamma, \delta , \varepsilon, \zeta)$,  the points ${\rm P}_1$ and ${\rm P}_2$ are the only singularities of Equation~(\ref{mainquartic}). 
\par As a first step in uncovering the claimed lattice polarization, we introduce the following three special lines,  denoted $ {\rm L}_1 $, $ {\rm L}_2 $, ${\rm L}_3$:
$$ \mathbf{X}=\mathbf{W}=0, \ \ \ \mathbf{Z}=\mathbf{W}=0, \ \ 2\varepsilon \mathbf{X}-\zeta \mathbf{W} = \mathbf{Z} = 0 \ . $$
Note that $ {\rm L}_1 $, $ {\rm L}_2 $, ${\rm L}_3$ lie on the quartic in Equation~(\ref{mainquartic}). 
\par Assume the case $ \gamma \varepsilon \neq 0 $. Then $ {\rm L}_1 $, $ {\rm L}_2 $, ${\rm L}_3$  are distinct and concurrent, meeting at ${\rm P}_1$. When taking the minimal resolution 
$\mathcal{X}(\alpha, \beta, \gamma, \delta , \varepsilon, \zeta)$, one obtains a configuration of smooth rational curves intersecting  according to the dual diagram below.    
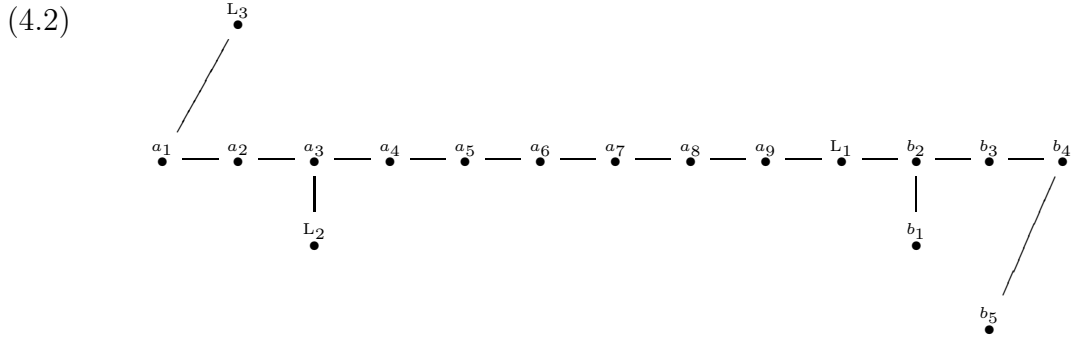
\begin{equation}
\label{diaggl7}
\def\objectstyle{\scriptstyle}
\def\labelstyle{\scriptstyle}
\xymatrix @-0.9pc  {
& & \stackrel{{\rm L}_3}{\bullet} 
& &  & & &    & & & & & 
&  &  \\ 
 & &  & &  & & & &   &    & &  & & \\
& \stackrel{a_1}{\bullet} \ar @{-} [r] 
\ar @{-}[uur] &
\stackrel{a_2}{\bullet} \ar @{-} [r]  &
\stackrel{a_3}{\bullet} \ar @{-} [r]  \ar @{-} [d] &
\stackrel{a_4}{\bullet} \ar @{-} [r] &
\stackrel{a_5}{\bullet} \ar @{-} [r] &
\stackrel{a_6}{\bullet} \ar @{-} [r] &
\stackrel{a_7}{\bullet} \ar @{-} [r] &
\stackrel{a_8}{\bullet} \ar @{-} [r] &
\stackrel{a_9}{\bullet} \ar @{-} [r] &
\stackrel{{\rm L}_1}{\bullet} \ar @{-} [r] &
\stackrel{b_2}{\bullet} \ar @{-} [r] \ar @{-} [d] &
\stackrel{b_3}{\bullet} & \stackrel{b_4}{\bullet} \ar @{-} [l] \ar @{-}[ddl] 
 & \\
 & &  & \stackrel{{\rm L}_2}{\bullet} &  & & & &   &    & & \stackrel{b_1}{\bullet}  & & \\
 & & 
 & &  & & &    & & & & &   \stackrel{b_5}{\bullet} & & \\
} 
\end{equation}
The two sets $ a_1, a_2, \dots, a_9$ and $b_1, b_2, \dots, b_5$ denote the curves appearing from resolving the rational double  point singularities at ${\rm P}_1$ and ${\rm P}_2$, respectively. In the context of diagram (\ref{diaggl7}), the lattice polarization 
$H \oplus E_7(-1) \oplus E_7(-1) $ is generated by:
\begin{align*}
{\rm H}_{ \ } \ =& \ \langle a_7, \ {\rm L}_3 + 2 a_1 + 3 a_2 + 4 a_3 + 2 {\rm L}_2 + 3 a_4 + 2 a_5 + a_6  \rangle \\
{\rm E}_7 \ =& \ \langle \ {\rm L}_3, \  a_1,  \ a_2, \   a_3, \  {\rm L}_2, \  a_4, \ a_5 \ \rangle \\
{\rm E}_7 \ =& \ \langle \ b_5, \  b_4,  \ b_3, \   b_2, \  b_1, \  {\rm L}_1, \ a_9 \ \rangle \ .
\end{align*}
In the case $\gamma=0$, $\varepsilon \neq 0$, the diagram (\ref{diaggl7}) changes to: 
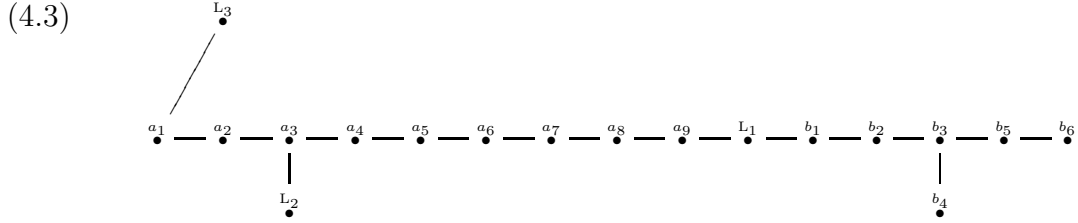
\begin{equation}
\label{diaggl8}
\def\objectstyle{\scriptstyle}
\def\labelstyle{\scriptstyle}
\scalemath{\MyScaleMedium}{
\xymatrix @-0.9pc  {
& & \stackrel{{\rm L}_3}{\bullet} 
& &  & & &    & & & & & 
&  &  \\ 
 & &  & &  & & & &   &    & & & &  & & \\
& \stackrel{a_1}{\bullet} \ar @{-} [r] 
\ar @{-}[uur] &
\stackrel{a_2}{\bullet} \ar @{-} [r]  &
\stackrel{a_3}{\bullet} \ar @{-} [r]  \ar @{-} [d] &
\stackrel{a_4}{\bullet} \ar @{-} [r] &
\stackrel{a_5}{\bullet} \ar @{-} [r] &
\stackrel{a_6}{\bullet} \ar @{-} [r] &
\stackrel{a_7}{\bullet} \ar @{-} [r] &
\stackrel{a_8}{\bullet} \ar @{-} [r] &
\stackrel{a_9}{\bullet} \ar @{-} [r] &
\stackrel{{\rm L}_1}{\bullet} \ar @{-} [r] &
\stackrel{b_1}{\bullet} \ar @{-} [r] &
\stackrel{b_2}{\bullet} \ar @{-} [r] &
\stackrel{b_3}{\bullet} \ar @{-} [r] \ar @{-} [d] &
\stackrel{b_5}{\bullet} & \stackrel{b_6}{\bullet} \ar @{-} [l] 
 & \\
 & &  & \stackrel{{\rm L}_2}{\bullet} &  & & & & &  &   &    & & \stackrel{b_4}{\bullet}  & & 
} }
\end{equation}
One obtains an enhanced lattice polarization of type $H \oplus E_8(-1) \oplus E_7(-1)$ with:
\begin{align*}
{\rm H}_{ \ } \ =& \ \langle a_7, \ {\rm L}_3 + 2 a_1 + 3 a_2 + 4 a_3 + 2 {\rm L}_2 + 3 a_4 + 2 a_5 + a_6  \rangle \\
{\rm E}_7 \ =& \ \langle \ {\rm L}_3, \  a_1,  \ a_2, \   a_3, \  {\rm L}_2, \  a_4, \ a_5 \ \rangle \\
{\rm E}_8 \ =& \ \langle \ b_6, \ b_5, \  b_4,  \ b_3, \   b_2, \  b_1, \  {\rm L}_1, \ a_9 \ \rangle \ .
\end{align*}
In the case $ \gamma \neq 0$, $ \varepsilon = 0$, the lines ${\rm L}_2$ and ${\rm L}_3$ coincide but the rational double point  type at ${\rm P}_1$ changes from ${\rm A}_9$ to ${\rm A}_{11}$. One obtains rational curves on $\mathcal{X}(\alpha, \beta, \gamma, \delta , \varepsilon, \zeta)$ as follows:  
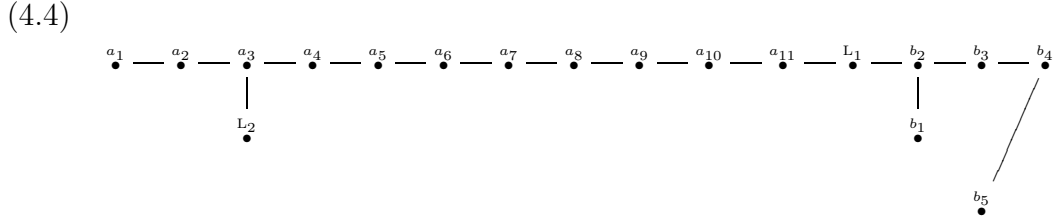
\begin{equation}
\label{diaggl9}
\def\objectstyle{\scriptstyle}
\def\labelstyle{\scriptstyle}
\scalemath{\MyScaleMedium}{
\xymatrix @-0.9pc  {
& \stackrel{a_1}{\bullet} \ar @{-} [r] 
&
\stackrel{a_2}{\bullet} \ar @{-} [r]  &
\stackrel{a_3}{\bullet} \ar @{-} [r]  \ar @{-} [d] &
\stackrel{a_4}{\bullet} \ar @{-} [r] &
\stackrel{a_5}{\bullet} \ar @{-} [r] &
\stackrel{a_6}{\bullet} \ar @{-} [r] &
\stackrel{a_7}{\bullet} \ar @{-} [r] &
\stackrel{a_8}{\bullet} \ar @{-} [r] &
\stackrel{a_9}{\bullet} \ar @{-} [r] &
\stackrel{a_{10}}{\bullet} \ar @{-} [r] &
\stackrel{a_{11}}{\bullet} \ar @{-} [r] &
\stackrel{{\rm L}_1}{\bullet} \ar @{-} [r] &
\stackrel{b_2}{\bullet} \ar @{-} [r] \ar @{-} [d] &
\stackrel{b_3}{\bullet} & \stackrel{b_4}{\bullet} \ar @{-} [l] \ar @{-}[ddl] 
 & \\
 & &  & \stackrel{{\rm L}_2}{\bullet} &  & & & &  & & &    & & \stackrel{b_1}{\bullet}  & & \\
 & & 
 & &  & & &    & & &  & & & &   \stackrel{b_5}{\bullet} & & \\
} }
\end{equation}
This provides a $H \oplus E_8(-1) \oplus E_7(-1)$ polarization as follows:
\begin{align*}
{\rm H}_{ \ } \ =& \ \langle a_9, \ 2 a_1 + 4 a_2 + 6 a_3 + 3 {\rm L}_2 + 5 a_4 + 4 a_5 + 3 a_6 + 2 a_7 + a_8  \rangle \\
{\rm E}_8 \ =& \ \langle \ a_1,  \  a_2, \ a_3, \ {\rm L}_2, \ a_4, \ a_5, \ a_6, \ a_7 \ \rangle \\
{\rm E}_7 \ =& \ \langle \ b_5, \  b_4,  \ b_3, \   b_2, \  b_1, \  {\rm L}_1, \ a_{11} \ \rangle \ .
\end{align*}
Finally, in the case $ \gamma = \varepsilon = 0$, the lines ${\rm L}_2$, ${\rm L}_3$ coincide and the rational double point types  at ${\rm P}_1$ and ${\rm P}_2$ are ${\rm A}_{11}$ and ${\rm E}_6$, respectively. This determines the following diagram of smooth rational  curves on the resolution $\mathcal{X}(\alpha, \beta, \gamma, \delta , \varepsilon, \zeta)$.  
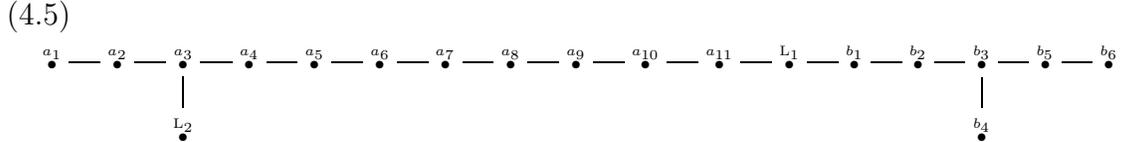
\begin{equation}
\label{diaggl89}
\def\objectstyle{\scriptstyle}
\def\labelstyle{\scriptstyle}
\scalemath{\MyScaleMedium}{
\xymatrix @-0.9pc  {
 \stackrel{a_1}{\bullet} \ar @{-} [r] 
&
\stackrel{a_2}{\bullet} \ar @{-} [r]  &
\stackrel{a_3}{\bullet} \ar @{-} [r]  \ar @{-} [d] &
\stackrel{a_4}{\bullet} \ar @{-} [r] &
\stackrel{a_5}{\bullet} \ar @{-} [r] &
\stackrel{a_6}{\bullet} \ar @{-} [r] &
\stackrel{a_7}{\bullet} \ar @{-} [r] &
\stackrel{a_8}{\bullet} \ar @{-} [r] &
\stackrel{a_9}{\bullet} \ar @{-} [r] &
\stackrel{a_{10}}{\bullet} \ar @{-} [r] &
\stackrel{a_{11}}{\bullet} \ar @{-} [r] &
\stackrel{{\rm L}_1}{\bullet} \ar @{-} [r] &
\stackrel{b_1}{\bullet} \ar @{-} [r] &
\stackrel{b_2}{\bullet} \ar @{-} [r] &
\stackrel{b_3}{\bullet} \ar @{-} [d] \ar @{-} [r] & 
\stackrel{b_5}{\bullet} \ar @{-} [r]  &
\stackrel{b_6}{\bullet}   
 \\
  &  & \stackrel{{\rm L}_2}{\bullet} &  & & & & &   & & & &    & & \stackrel{b_4}{\bullet}  & & \\
} }
\end{equation}
This determines a lattice polarization of type $H \oplus E_8(-1) \oplus E_8(-1)$ polarization as follows:
\begin{align*}
{\rm H}_{ \ } \ =& \ \langle a_9, \ 2 a_1 + 4 a_2 + 6 a_3 + 3 {\rm L}_2 + 5 a_4 + 4 a_5 + 3 a_6 + 2 a_7 + a_8  \rangle \\
{\rm E}_8 \ =& \ \langle \ a_1,  \  a_2, \ a_3, \ {\rm L}_2, \ a_4, \ a_5, \ a_6, \ a_7 \ \rangle \\
{\rm E}_8 \ =& \ \langle \ b_6, \ b_5, \  b_4,  \ b_3, \   b_2, \  b_1, \  {\rm L}_1, \ a_{11} \ \rangle \ .
\end{align*}
\end{proof}
\begin{remark}
The degree-four polarization determined on ${\rm X}(\alpha, \beta, \gamma, \delta , \varepsilon, \zeta)$ by its quartic definition 
is described explicitly in the context of diagrams $(\ref{diaggl7})$-$(\ref{diaggl89})$. For instance, assuming the case  
$ \gamma \varepsilon \neq 0 $, one can write a polarizing divisor as:
\begin{equation}
\label{linepolariz}
\scalemath{\MyScaleMedium}{
\begin{array}{c}
\mathcal{L} \ = \ {\rm L}_2 + 
\left ( 
a_1+2a_2+3a_3+3a_4+3a_5+\cdots 3a_9 
\right ) 
+ 3 {\rm L}_1 + 
\left ( 
2b_1+4b_2+3b_3+2b_4+b_5
\right )
\end{array}}
\end{equation}
Similar formulas hold in the other three cases.
\end{remark}
\noindent Diagrams $(\ref{diaggl7})$, $(\ref{diaggl8})$ and $(\ref{diaggl9})$ from the above proof can be nicely 
augmented. Consider the following complete intersections:
$$ 2\varepsilon \mathbf{X}-\zeta \mathbf{W}  \ = \  
\left (3 \alpha \varepsilon ^2 \zeta + 2 \beta \varepsilon ^3 - \zeta^3 \right ) \mathbf{W}^2 -  
\varepsilon^2 \left ( \delta \varepsilon-\gamma \zeta \right ) \mathbf{Z} \mathbf{W} 
+ 2 \varepsilon^3 \mathbf{Y}^2 \ = \  0 
$$
$$ 
2\gamma \mathbf{X}-\delta \mathbf{W}  \ = \  
\left (3 \alpha \gamma ^2 \delta + 2 \beta \gamma ^3 - \delta^3 \right ) \mathbf{Z} \mathbf{W}^2 -  
\gamma^2 \left ( \gamma \zeta - \delta \varepsilon \right ) \mathbf{W}^3 
+ 2 \gamma^3 \mathbf{Y}^2 \mathbf{Z} \ = \  0 \ .$$
Assuming appropriate generic conditions, the above equations determine two projective curves ${\rm R}_1$, ${\rm R}_2$, of degrees two and three, respectively. 
The conic ${\rm R}_1$ is a (generically smooth) rational curve tangent to ${\rm L}_1$ at ${\rm P}_2$. The cubic ${\rm R}_2$ has 
a double point at ${\rm P}_2$, passes through ${\rm P}_1$ and is generically irreducible. When resolving the quartic surface $(\ref{mainquartic})$, 
these two curves lift to smooth rational curves on ${\rm X}(\alpha, \beta, \gamma, \delta , \varepsilon, \zeta)$, which by a slight 
abuse of notation we shall denote by the same symbol. One obtains the following dual diagrams of rational curves.    
\par Case $ \gamma \neq 0$, $ \varepsilon \neq 0$: 
\begin{equation}
\label{exdiagg77}
\def\objectstyle{\scriptstyle}
\def\labelstyle{\scriptstyle}
\scalemath{\MyScaleMedium}{
\xymatrix @-0.9pc  {
& & \stackrel{{\rm L}_3}{\bullet} \ar @{=}[rrrrrrrrrr]& &  & & &    & & & & & \stackrel{{\rm R}_1}{\bullet} &  &  \\ 
 & &  & &  & & & &   &    & &  & & \\
& \stackrel{a_1}{\bullet} \ar @{-} [r] \ar @{-}[ddr] \ar @{-}[uur] &
\stackrel{a_2}{\bullet} \ar @{-} [r]  &
\stackrel{a_3}{\bullet} \ar @{-} [r]  \ar @{-} [d] &
\stackrel{a_4}{\bullet} \ar @{-} [r] &
\stackrel{a_5}{\bullet} \ar @{-} [r] &
\stackrel{a_6}{\bullet} \ar @{-} [r] &
\stackrel{a_7}{\bullet} \ar @{-} [r] &
\stackrel{a_8}{\bullet} \ar @{-} [r] &
\stackrel{a_9}{\bullet} \ar @{-} [r] &
\stackrel{{\rm L}_1}{\bullet} \ar @{-} [r] &
\stackrel{b_2}{\bullet} \ar @{-} [r] \ar @{-} [d] &
\stackrel{b_3}{\bullet} & \stackrel{b_4}{\bullet} \ar @{-} [l] \ar @{-}[ddl] \ar @{-}[uul]
 & \\
 & &  & \stackrel{{\rm L}_2}{\bullet} &  & & & &   &    & & \stackrel{b_1}{\bullet}  & & \\
 & & \stackrel{{\rm R}_2}{\bullet} \ar @{=}[rrrrrrrrrr]& &  & & &    & & & & &   \stackrel{b_5}{\bullet} & & \\
} }
\end{equation}
\par Case $ \gamma = 0$, $ \varepsilon \neq 0$: 
\begin{equation}
\label{exdiaggl8}
\def\objectstyle{\scriptstyle}
\def\labelstyle{\scriptstyle}
\scalemath{\MyScaleMedium}{
\xymatrix @-0.9pc  {
& & \stackrel{{\rm L}_3}{\bullet} 
\ar @{=}[rrrrrrrrrrrr]
& &  & & &    & & & & &  & & 
\stackrel{{\rm R}_1}{\bullet} 
&  &  \\ 
 & &  & &  & & & &   &    & & & &  & & \\
& \stackrel{a_1}{\bullet} \ar @{-} [r] 
\ar @{-}[uur] &
\stackrel{a_2}{\bullet} \ar @{-} [r]  &
\stackrel{a_3}{\bullet} \ar @{-} [r]  \ar @{-} [d] &
\stackrel{a_4}{\bullet} \ar @{-} [r] &
\stackrel{a_5}{\bullet} \ar @{-} [r] &
\stackrel{a_6}{\bullet} \ar @{-} [r] &
\stackrel{a_7}{\bullet} \ar @{-} [r] &
\stackrel{a_8}{\bullet} \ar @{-} [r] &
\stackrel{a_9}{\bullet} \ar @{-} [r] &
\stackrel{{\rm L}_1}{\bullet} \ar @{-} [r] &
\stackrel{b_1}{\bullet} \ar @{-} [r] &
\stackrel{b_2}{\bullet} \ar @{-} [r] &
\stackrel{b_3}{\bullet} \ar @{-} [r] \ar @{-} [d] &
\stackrel{b_5}{\bullet} & \stackrel{b_6}{\bullet} \ar @{-} [l] \ar @{-} [uul]
 & \\
 & &  & \stackrel{{\rm L}_2}{\bullet} &  & & & & &  &   &    & & \stackrel{b_4}{\bullet}  & & 
} }
\end{equation}
\par Case $ \gamma \neq 0$, $ \varepsilon = 0$: 
\begin{equation}
\label{exdiaggl9}
\def\objectstyle{\scriptstyle}
\def\labelstyle{\scriptstyle}
\scalemath{\MyScaleMedium}{
\xymatrix @-0.9pc  {
& \stackrel{a_1}{\bullet} \ar @{-} [r] 
\ar @{-}[ddr] 
&
\stackrel{a_2}{\bullet} \ar @{-} [r]  &
\stackrel{a_3}{\bullet} \ar @{-} [r]  \ar @{-} [d] &
\stackrel{a_4}{\bullet} \ar @{-} [r] &
\stackrel{a_5}{\bullet} \ar @{-} [r] &
\stackrel{a_6}{\bullet} \ar @{-} [r] &
\stackrel{a_7}{\bullet} \ar @{-} [r] &
\stackrel{a_8}{\bullet} \ar @{-} [r] &
\stackrel{a_9}{\bullet} \ar @{-} [r] &
\stackrel{a_{10}}{\bullet} \ar @{-} [r] &
\stackrel{a_{11}}{\bullet} \ar @{-} [r] &
\stackrel{{\rm L}_1}{\bullet} \ar @{-} [r] &
\stackrel{b_2}{\bullet} \ar @{-} [r] \ar @{-} [d] &
\stackrel{b_3}{\bullet} & \stackrel{b_4}{\bullet} \ar @{-} [l] \ar @{-}[ddl] 
 & \\
 & &  & \stackrel{{\rm L}_2}{\bullet} &  & & & &  & & &    & & \stackrel{b_1}{\bullet}  & & \\
 & & 
 \stackrel{{\rm R}_2}{\bullet} \ar @{=}[rrrrrrrrrrrr]
 & &  & & &    & & &  & & & &   \stackrel{b_5}{\bullet} & & \\
} }
\end{equation}
Note that diagrams $(\ref{exdiaggl8})$ and $(\ref{exdiaggl9})$ are similar in nature. This is not a fortuitous fact, as we shall see next.  
\begin{proposition}
\label{symmetries1}
Let $(\alpha, \beta, \gamma, \delta, \varepsilon, \zeta) \in \mathbb{C}^6 $ with  
$(\gamma, \delta) \neq (0,0)$ and $(\varepsilon, \zeta)  \neq (0,0)$. Then, one 
has the following isomorphisms of $H \oplus E_7(-1) \oplus E_7(-1)$ lattice polarized K3 surfaces:
\begin{itemize}
\item [(a)] $\mathcal{X}(\alpha,\beta, \gamma, \delta, \varepsilon, \zeta) \ \simeq \
\mathcal{X}(t^2 \alpha,  \ t^3 \beta,  \ t^5 \gamma,  \ t^6 \delta,  \ t^{-1} \varepsilon,  \ \zeta  ) $, for any $t \in \mathbb{C}^*$.
\item [(b)] $\mathcal{X}(\alpha,\beta, \gamma, \delta, \varepsilon, \zeta) \ \simeq \  
\mathcal{X}(\alpha,  \beta,  \varepsilon,  \zeta, \gamma,  \delta  )$.
\end{itemize}
\end{proposition}
\begin{proof}
Let $q$ be a square root of $t$. Then, the projective automorphism given by
\begin{equation}
\label{tmor}
\mathbb{P}^3 \ \longrightarrow \ \mathbb{P}^3, \ \ \ [\mathbf{X}: \mathbf{Y}: \mathbf{Z}: \mathbf{W}] \ \mapsto \ [\ q^8\mathbf{X}: \ q^9\mathbf{Y}: \mathbf{Z}: \ q^6\mathbf{W} \ ] \ 
\end{equation}
extends to an isomorphism $\mathcal{X}(\alpha,\beta, \gamma, \delta, \varepsilon, \zeta) \simeq \mathcal{X}(t^2 \alpha,  t^3 \beta,  t^5 \gamma,  t^6 \delta,  t^{-1} \varepsilon,  \zeta  ) $ preserving the lattice polarization. Similarly, the birational involution:
\begin{equation}
\label{invoef}
\mathbb{P}^3 \ \dashrightarrow \ \mathbb{P}^3, \ \ \ [\mathbf{X}: \mathbf{Y}: \mathbf{Z}: \mathbf{W}] \ \mapsto \ [\ \mathbf{X}\mathbf{Z}: \ \mathbf{Y}\mathbf{Z}, \mathbf{W}^2, \ \mathbf{Z}\mathbf{W} \ ] \ .
\end{equation}
extends to an isomorphism between $\mathcal{X}(\alpha,\beta, \gamma, \delta, \varepsilon, \zeta)$ and  $\mathcal{X}(\alpha,  \beta,  \varepsilon,  \zeta, \gamma,  \delta  )$.
\end{proof}
\subsection{Elliptic Fibrations on \texorpdfstring{$\mathcal{X}$}{X}}
By the nature of the $H \oplus E_7(-1) \oplus E_7(-1)$ lattice polarizations, K3 surfaces  $\mathcal{X}(\alpha,\beta, \gamma, \delta, \varepsilon, \zeta)$  carry interesting elliptic fibrations with sections. As discussed in \cite{MR2824841}, there are four non-isomorphic elliptic fibrations with section; three will be important for the  considerations of this article.
\subsubsection{The standard fibration}
The first elliptic fibration with section is canonically associated with the lattice polarization, as the classes of its fiber and section span the hyperbolic factor in $H \oplus E_7(-1) \oplus E_7(-1)$.  Following the terminology of \cite{MR2369941}, we shall refer to this elliptic fibration as \emph{standard} and denote it by
$$
\pi^{\mathcal{X}}_{\mathrm{std}} \colon \mathcal{X}(\alpha,\beta, \gamma, \delta, \varepsilon, \zeta) \rightarrow \mathbb{P}^1 \ .
$$ 
One obtains the fibers of $\pi^{\mathcal{X}}_{\mathrm{std}}$ by considering the pencil of planes in $ \mathbb{P}^3 $ that contain the line $ {\rm L}_1 $, where ${\rm L}$ is the degree-four canonical  hyperplane polarization of $\mathcal{X}(\alpha,\beta, \gamma, \delta, \varepsilon, \zeta)$.
\par It is obtained from residual intersections with the pencil of planes containing the line $ \mathbf{Z}=\mathbf{W}=0$.  Setting 
\begin{equation}
\mathbf{X} = sx, \quad  \mathbf{Y}= y, \quad \mathbf{W} = -4\, s^3, \quad \mathbf{Z} = 4\, s^4,
\end{equation}
in Equation~(\ref{mainquartic}), we obtain the Weierstrass equation
\begin{equation}
\label{eqns_SI_std}
  \mathcal{X}_s: \quad y^2  = x^3 +  f(s) \, x + g(s), 
\end{equation}
with discriminant 
\begin{equation}
 \Delta_{\mathcal{X}}(s) =  4 \, f(s)^3 + 27 \, g(s)^2 =  - 64 \, s^9 P(s)
\end{equation}
and
\begin{equation}
 f(s)  = 4\, s^3 \Big(\gamma \, s^2 - 3 \, \alpha \, s + \varepsilon \Big), \quad
 g(s)  = -8 \, s^5 \Big(\delta \, s^2 + 2 \, \beta \, s + \zeta\big),
\end{equation}
and
\begin{equation}
\begin{array}{rl}
P(s) = &  4 \, \gamma^3 s^6 - 9\, (4\, \alpha \gamma^2-3 \,\delta^2)\, s^5 + 12\, (9 \, \alpha^2 \gamma+9 \, \beta  \delta+\gamma^2 \varepsilon)\, s^4\\[0.5em]
- & 18\,  (6\, \alpha^3+4\, \alpha \gamma \varepsilon-6 \, \beta^2-3\, \delta \zeta) \, s^3 + 12 \, (9\, \alpha^2 \varepsilon+9 \, \beta \zeta +\gamma \varepsilon^2) \, s^2  \\[0.5em]
- & 9 \, (4 \, \alpha \varepsilon^2-3 \,  \zeta^2 ) \, s + 4 \, \varepsilon^3 .
\end{array}
\end{equation}
In this way, we obtain an elliptically fibered K3 surface $\pi^{\mathcal{X}}_{\mathrm{std}} \colon \mathcal{X}(\alpha,\beta, \gamma, \delta, \varepsilon, \zeta) \rightarrow \mathbb{P}^1$ with section given by the point at infinity in each fiber $\mathcal{X}_{s}$ and minimal Weierstrass equation~(\ref{eqns_SI_std}).  The fibration has singularities of Kodaira type $III^*$ over $s=0$ and $s = \infty$. The following lemma is immediate:
\begin{lemma}
\label{lem:Ellstd}
Equation~(\ref{eqns_SI_std}) is a Jacobian elliptic fibration $\pi^{\mathcal{X}}_{\mathrm{std}} \colon \mathcal{X}(\alpha,\beta, \gamma, \delta, \varepsilon, \zeta) \rightarrow \mathbb{P}^1$ with six singular fibers of type $I_1$,  two singular fibers of type $III^*$, and the Mordell-Weil group of sections $\operatorname{MW}(\pi^{\mathcal{X}}_{\mathrm{std}} )=\{ \mathbb{I} \}$.
\end{lemma}
Application of Tate's algorithm shows immediately:
\begin{lemma} \label{lem:std_extension} We have the following:
\begin{itemize}
\item If $ \varepsilon \neq 0 $, there is a singular fiber of Kodaira type $III^*$ at $s = 0$. Otherwise, it is a singular fiber of Kodaira type $II^*$.
\item If $ \gamma \neq 0 $, there is a singular fiber of Kodaira type $III^*$ at $s = \infty$. Otherwise,  it is a singular fiber of  Kodaira type $II^*$.
\end{itemize}
\end{lemma}
\subsubsection{The alternate fibration}
Another elliptic fibration with section is obtained from residual intersections with the pencil of planes containing the line $\mathbf{X}=\mathbf{W}=0$.  Setting 
\begin{equation}
\mathbf{X}= t\, x^3, \quad  \mathbf{Y}=  \sqrt{2}\, x^2y, \quad \mathbf{W} = 2\, x^3, \quad \mathbf{Z} = 2\,x^2(-\varepsilon t+\zeta),
\end{equation}
in Equation~(\ref{mainquartic}),  determines a second Jacobian elliptic fibration $\pi^{\mathcal{X}}_{\mathrm{alt}} \colon \mathcal{X} \rightarrow \mathbb{P}^1$ with fiber $\mathcal{X}_{t}$ given by the Weierstrass equation
\begin{equation}
\label{eqns_SI_alt}
 \mathcal{X}_t: \quad y^2  = x \Big(x^2 + B(t) \, x +A(t)  \Big), 
\end{equation}
with discriminant 
\begin{equation}
 \Delta_{\mathcal{X}}(t) =  A(t)^2 \, \Big(B(t)^2-4\,A(t)\Big)
\end{equation}
and
\begin{equation}
\begin{split}
A(t)  &= (\gamma t-\delta)(\varepsilon t -\zeta)  = \gamma\varepsilon t^2 - (\gamma \zeta + \delta \varepsilon) t + \delta\zeta,\\
B(t) &= t^3-3\alpha  t- 2 \beta.
\end{split}
\end{equation}
Thus, we obtain an elliptically fibered K3 surface $\pi^{\mathcal{X}}_{\mathrm{alt}} \colon \mathcal{X}(\alpha,\beta, \gamma, \delta, \varepsilon, \zeta) \rightarrow \mathbb{P}^1$ which we call the \emph{alternate fibration}, with section given by the point at infinity in each fiber $\mathcal{X}_{t}$ and minimal Weierstrass equation~(\ref{eqns_SI_alt}).  It has a two-torsion section $(x,y)=(0,0)$, two singularities of Kodaira type $I_2$ over $A(t)=0$, and a singularity of Kodaira type $I_8^*$ over $t=\infty$. Therefore, the following is immediate:
\begin{lemma}
\label{lem:Ellalt}
Equation~(\ref{eqns_SI_alt}) defines a Jacobian elliptic fibration $\pi^{\mathcal{X}}_{\mathrm{alt}} \colon \mathcal{X} \rightarrow \mathbb{P}^1$ with six singular fibers of type $I_1$,  two singular fibers of type $I_2$, one singular fibers of type $I_8^*$, and the Mordell-Weil group of sections $\operatorname{MW}(\pi^{\mathcal{X}}_{\mathrm{alt}} )=\mathbb{Z}/2\mathbb{Z}$.
\end{lemma}
\par Setting
\begin{equation}
 x=T, \quad y = \frac{Y}{T^2}, \quad t=\frac{X-\frac{1}{3}\gamma \varepsilon T}{T^2},
\end{equation}
in Equation~(\ref{eqns_SI_alt}) determines another Jacobian elliptic fibration $\check{\pi}^{\mathcal{X}}_{\mathrm{alt}}: \mathcal{X} \to \mathbb{P}^1$ with fiber $\mathcal{X}_{T}$ given by the minimal Weierstrass equation
\begin{equation}
\label{eqn:bfdual}
 \mathcal{X}_{T}: \quad Y^2 = X^3 + \check{f}(T) \, X + \check{g}(T) ,
\end{equation} 
with discriminant 
\begin{equation}
 \Delta_{\mathcal{X}}(T)=4\, \check{f}(T)^3+27\, \check{g}(T)^2
\end{equation}
and
\begin{equation*}
\begin{split}
 \check{f}(T) =& -\frac{1}{3}\, T^2\, \Big(9\, \alpha \, T^2 + 3(\gamma\zeta + \delta \varepsilon)\, T +( \gamma\varepsilon)^2\Big),\\
 \check{g}(T)= & \, \frac{1}{27} \, T^3 \, \Big(27 \, T^4 -54\,  \beta\,  T^3 + 27\,  (\alpha \gamma \varepsilon + \delta \zeta) \, T^2 \\
 &\qquad + 9 \, \gamma\varepsilon\, (\delta\varepsilon+\gamma \zeta) \, T + 2( \gamma\varepsilon)^3\Big) .
\end{split}
\end{equation*}
Thus, we obtain a Jacobian elliptic fibration $\check{\pi}^{\mathcal{X}}_{\mathrm{alt}} \colon \mathcal{X}(\alpha,\beta, \gamma, \delta, \varepsilon, \zeta) \rightarrow \mathbb{P}^1$ which we call the \emph{base-fiber-dual} of the alternate fibration. We have the following:
\begin{lemma}
\label{lem:Ellalt_dual}
Equation~(\ref{eqn:bfdual}) defines a Jacobian elliptic fibration $\check{\pi}^{\mathcal{X}}_{\mathrm{alt}} \colon \mathcal{X} \rightarrow \mathbb{P}^1$ with 6 singular fibers of Kodaira type $I_1$, a singular fibers of Kodaira type $I_2^*$, and a singular fiber of Kodaira type $II^*$, and the Mordell-Weil group of sections $\operatorname{MW}(\check{\pi}^{\mathcal{X}}_{\mathrm{alt}})=\{ \mathbb{I} \}$. 
\end{lemma}

\subsection{Van~Geemen-Sarti involutions and moduli}
Equation~(\ref{eqns_SI_alt}) is a minimal Weierstrass model for the Jacobian elliptic fibration $\pi^{\mathcal{X}}_{\mathrm{alt}}: \mathcal{X} \to \mathbb{P}^1$ with fiber $\mathcal{X}_{t}$ given by
\begin{equation}
\label{Xsurface}
 \mathcal{X}_{t}: \quad y^2  = x \, \Big(x^2 + B(t)  \, x +  A(t) \Big) .
\end{equation} 
The singular fibers of $\mathcal{X}$ are located over the support of $\Delta_{\mathcal{X}}=A(t)^2 (B(t)^2-4  \, A(t))$. A smooth section $\sigma$ is given by the point at infinity in each fiber. A two-torsion section $\tau$ is given by $\tau: t \mapsto (x,y)=(0,0)$ such that $2  \tau=\sigma$. Thus, we have  $\mathbb{Z}/2\mathbb{Z} \subset \operatorname{MW}(\pi^{\mathcal{X}}_{\mathrm{alt}})$. The holomorphic two-form is given by $\omega_{\mathcal{X}}=dt \wedge dx/y$.

A \emph{Nikulin involution} on a K3 surface $\mathcal{X}$ is a symplectic involution $\jmath_{\mathcal{X}}: \mathcal{X} \to \mathcal{X}$, i.e., an involution with $\jmath_{\mathcal{X} }^*(\omega)=\omega$. If a Nikulin involution exists on a K3 surface $\mathcal{X}$, then it necessarily has eight fixed points, and the minimal resolution of the quotient surface is another K3 surface $\mathcal{Y}=\widehat{\mathcal{X}/\{1,\jmath_{\mathcal{X} }\}}$ \cite{MR544937}. Special Nikulin involution are obtained in our situation: the fiberwise translation by the two-torsion section acting by $p \mapsto p+\tau$ for all  $p \in \mathcal{X}_{t}$ extends to a Nikulin involution $\jmath_{\mathcal{X} }$ on $\mathcal{X}$, called \emph{Van~\!Geemen-Sarti} involution. A computation shows that the involution is, on each fiber $\mathcal{X}_t$, given by
\begin{equation}
\label{involution_fiber}
 \jmath_{\mathcal{X}_t}: (x,y) \mapsto (x,y) + (0,0) = \left( \frac{A(t)}{x}, - \frac{A(t)\, y}{x^2} \right)
\end{equation} 
for $p \not \in  \{\sigma, \tau\}$  and interchanges $\sigma$ and $\tau$. It is also easy to check that $\jmath_{\mathcal{X} }$ leaves the holomorphic two-form $\omega_{\mathcal{X}}$ invariant. Using the smooth two-isogeneous elliptic curve $ \mathcal{X}_t/\{ \sigma, \tau\}$ for each smooth fiber, we obtain the new K3 surface $\mathcal{Y}$ equipped with an elliptic fibration $\pi^{\mathcal{Y}}_{\mathrm{alt}}:\mathcal{Y}\to\mathbb{P}^1$ with section $\Sigma$ as the Weierstrass model with fiber $\mathcal{Y}_{t}$ given by
\begin{equation}
\label{Ysurface}
 \mathcal{Y}_{t}: \quad Y^2 = X \Big(X^2- 2 \, B(t) \,    X +  B(t)^2-4 \, A(t)\Big).
\end{equation}
The singular fibers of $\mathcal{Y}$ are located over the support of $\Delta_{\mathcal{Y}}=16\, A(t)\,(B(t)^2-4A(t))^2$. The holomorphic two-form on $\mathcal{Y}$ is $\omega_{\mathcal{Y}}=dt\wedge dX/Y$.  The fiberwise isogeny given by
\begin{equation}
\label{isog}
 \hat{\Phi}\vert_{\mathcal{X}_t}: (x,y) \mapsto (X,Y) = \left( \frac{y^2}{x^2}, \frac{(x^2-A(t))y}{x^2} \right) 
\end{equation}
extends  to a degree-two rational map $\hat{\Phi}: \mathcal{X} \dasharrow \mathcal{Y}$. We observe that the K3 surface $\mathcal{Y}$ satisfies $\mathbb{Z}/2\mathbb{Z} \subset \operatorname{MW}(\pi^{\mathcal{Y}}_{\mathrm{alt}})$ with a two-torsion section $T$ given by $T: t \mapsto (X,Y)=(0,0)$. Therefore, the surface $\mathcal{Y}$ is itself equipped with a Van~\!Geemen-Sarti involution $\jmath_{\mathcal{Y}}$, namely
\begin{equation}
\label{dual_isog_involution}
 \jmath_{\mathcal{Y}_t}: (X,Y) \mapsto (X,Y) + (0,0) = \left( \frac{B(t)^2-4\, A(t)}{X}, - \frac{(B(t)^2-4\,A(t))Y}{X^2} \right).
\end{equation} 
The involution $\jmath_{\mathcal{Y}}$ leaves the holomorphic two-form $\omega_{\mathcal{Y}}$ invariant and covers the map $\Phi$ extending the fiberwise dual isogeny $P \mapsto P + T$ for all $P \in \mathcal{Y}_t$ given by 
\begin{equation}
\label{dual_isog}
  \Phi\vert_{\mathcal{Y}_t}: (X,Y) \mapsto (x,y) = \left( \frac{Y^2}{4X^2}, \frac{Y(X^2-B(t)^2+4\, A(t))}{8\, X^2} \right) .
\end{equation}
The situation is summarized in the following diagram: 
\begin{equation}
\label{pic:VGS}
\xymatrix{
\mathcal{X} \ar @(dl,ul) _{\jmath_{\mathcal{X}}} \ar [dr] _{\pi^{\mathcal{X}}_{\mathrm{alt}}} \ar @/_0.5pc/ @{-->} _{\hat{\Phi}} [rr]
&
& \mathcal{Y} \ar @(dr,ur) ^{\jmath_{\mathcal{Y}}} \ar [dl] ^{\pi^{\mathcal{Y}}_{\mathrm{alt}}} \ar @/_0.5pc/ @{-->} _{\Phi} [ll] \\
& \mathbb{P}^1 }
\end{equation}
We refer to such K3 surfaces $\mathcal{X}$ and $\mathcal{Y}$ as \emph{Van~\!Geemen-Sarti partners}. Therefore, the family of K3 surfaces $\mathcal{Y}$ associated with the double cover of the projective plane branched along the union of six lines equipped with the alternate fibration in Equation~(\ref{eqns_Kummer_alt}) and the Clingher-Doran family of K3 surfaces equipped with the alternate fibration in Equation~(\ref{eqns_SI_alt}) constitute such Van~\!Geemen-Sarti partners; see \cites{MalmendierClingher:2018,Clingher:2017aa}. The notion of Van~\!Geemen-Sarti partners is more general than the one of a Shioda-Inose structure. We make the following:
\begin{remark}
In Picard rank $17$, $\mathcal{X}$ carries a Shioda-Inose structure \cites{MR2369941,MR2824841}. The quotient map $\hat{\Phi}: \mathcal{X} \dashrightarrow \mathcal{Y}=\operatorname{Kum}(A)$ induces a Hodge isometry $T_{\mathcal{X}}(2 ) \cong T_{\operatorname{Kum}(A)}$. In Picard rank 16,  the map $\hat{\Phi}: \mathcal{X} \dashrightarrow \mathcal{Y}$ in Equation~(\ref{pic:VGS}) does NOT induce a Hodge isometry. In Proposition~\ref{lem_6lines} the transcendental lattice $T_{\mathcal{Y}}$ of the family of K3 surfaces $\mathcal{Y}$, and in Proposition~\ref{symmetries1} the lattice polarization of the family of K3 surfaces $\mathcal{X}$ were determined. For generic parameters, we have 
\begin{equation}
\begin{split}
 T_{\mathcal{X}}&= H \oplus H \oplus \langle -2 \rangle^{\oplus 2},\\
 T_{\mathcal{Y}}&= H(2) \oplus H(2) \oplus \langle -2 \rangle^{\oplus 2}.
\end{split}
\end{equation}
Hence, it is no longer the case that $T_{\mathcal{X}}(2 ) \cong T_{\mathcal{Y}}$.  
\end{remark}
In the context of the above results, we have the following:
\begin{proposition}
\label{lem:polar1}
Any $H \oplus E_7(-1) \oplus E_7(-1)$-polarized K3 surface $\mathcal{X}$ given by Equation~(\ref{mainquartic}) is the Van~\!Geemen-Sarti partner of a K3 surface $\mathcal{Y}$ given in Theorem~\ref{thm:6lines_alt} associated with a six-line configuration in $\mathbb{P}^2$ with invariants $J_k$ for $k=2, \dots, 6$. In particular, we have
\begin{equation}
\label{modd}
 \left [ J_2 : \ J_3 : \  J_4: \ J_5 : \ J_6   \right ]  
  =  
 \left [\, \alpha : \ \beta : \  \gamma \cdot \varepsilon: \ \gamma \cdot \zeta + \delta \cdot \varepsilon : \ \delta \cdot \zeta  \right ] 
\end{equation}
as points in the four-dimensional weighted projective space $\mathbb{P}(2,3,4,5,6)$.
\end{proposition}
\begin{proof}
The proof follows directly by comparing Equation~(\ref{Ysurface}) -- obtained by fiberwise two-isogeny from Equation~(\ref{eqns_SI_alt}) -- with Equation~(\ref{eqns_Kummer_alt}). It then follows that $\mathcal{A}(t)=A(t)$ and $\mathcal{B}(t)=B(t)$, and the claim follows.
\end{proof}
We also have the following:
\begin{lemma}
\label{lem:polar2}
The isomorphism classes in the family of K3 surfaces $\mathcal{X}(\alpha,\beta, \gamma, \delta, \varepsilon, \zeta)$ in Equation~(\ref{mainquartic}) are parametrized by the four-dimensional open complex variety $\mathfrak{M}$ defined in Equation~(\ref{modulispace}).
\end{lemma}
\begin{proof}
As a consequence of Theorem~\ref{thm:polarization}, one has an isomorphism of polarized K3 surfaces 
\begin{equation}
\label{eq:isom}
\mathcal{X}(\alpha,\beta, \gamma, \delta, \varepsilon, \zeta) \ \simeq \ \mathcal{X}(\alpha,  \ \beta,  \ t \gamma,  \ t \delta, \ t^{-1} \varepsilon,  \ t^{-1} \zeta  )
\end{equation}
for any $t \in \mathbb{C}^*$.  The conditions imposed on the pairs $( J_4 , \  J_5, \ J_6 ) \neq (0,0,0)$ ensure that singularities of  ${\rm Q}(\alpha, \beta, \gamma, \delta , \varepsilon, \zeta)$ are rational double points. 
\end{proof}
Combing the results of Proposition~\ref{lem:polar1} and Lemma~\ref{lem:polar2} we obtain the following:
\begin{theorem}
\label{cor:K3moduli_space}
The moduli space $\mathfrak{M}$ in Equation~(\ref{modulispace}) is the coarse moduli space of  K3 surfaces endowed with $H \oplus E_7(-1) \oplus E_7(-1)$ lattice polarization.
\end{theorem}
We have the immediate consequence:
\begin{corollary}
The loci of the singular fibers in the alternate fibration on the K3 surface $\mathcal{X}$ are determined by the Satake sextic in Section~\ref{ssec:mod_forms}. That is, if $\modvar \in \mathfrak{M}$ is the point in the moduli space associated with the six-line configuration defining $\mathcal{Y}$ and $\mathcal{X}$, the loci of the fibers of Kodaira type $I_1$ and $I_2$ are given by $\mathcal{S} =\mathcal{B}^2 - 4 \, \mathcal{A} = 0$ and $\mathcal{A} = 0$, respectively, with
\begin{equation}
\label{SatakeSextic2b}
\begin{split}
  \mathcal{B}=t^3- 3 \, J_2(\modvar)\, t - 2 \, J_3(\modvar), \ & \quad 
 \mathcal{A}= J_4(\modvar)\, t^2 - J_5(\modvar) \, t  + J_6(\modvar),
 \end{split}
 \end{equation}
where $J_k(\modvar)$ are the modular forms of weights $2k$ for $k=2, \dots , 6$ in Theorem~\ref{thm4} generating the ring of modular forms relative to $\Gamma_{\mathcal{T}}$.
\end{corollary}
Next, we describe what confluences of singular fibers appear in the three Jacobian elliptic fibrations determined above. We discuss several cases for each fibration where the labelling corresponds to the one used to characterize six-line configurations in Definition~\ref{def:sstable}, Lemma~\ref{lem:Js}, and Corollary~\ref{cor:VanishingLoci}. We have the following:
\begin{lemma} 
\label{lem:extensions_alt}
The Weierstrass model in Equation~(\ref{eqns_SI_alt}) associated with a six-line configuration in $\mathbb{P}^2$ with invariants $J_k$ for $k=2, \dots, 6$ satisfies the following:
\begin{enumerate}
    \item[(0)] In the generic case, there are singular fibers $I_8^* + 2 \, I_2 + 6 \, I_1$. 
    \item[(0b)] If $\operatorname{Res} (\mathcal{A}, \mathcal{B})=0$, one $I_1$ and one $I_2$ fiber coalesce to a $III$ fiber.
    \item[(1)] If $J_4 = 0$, one $I_2$ and the $I_8^*$ fiber coalesce to an $I_{10}^*$ fiber. 
    \item[(2)] If $\operatorname{Disc} (\mathcal{A})=0$, two $I_2$ fibers coalesce to an $I_4$ fiber.
    \item[(2b)] If $\operatorname{Disc} (\mathcal{S})=0$, two $I_1$ fibers coalesce to an $I_2$ fiber.
    \item[(3+4)]  If $\operatorname{Disc} (\mathcal{A})=\operatorname{Res} (\mathcal{A}, \mathcal{B})=0$,  two $I_1$ and two $I_2$ fibers coalesce to an $I_0^*$ fiber.
    \item[(5)] If $J_4 = J_5 = 0$, two $I_2$ fibers and the $I_8^*$ fiber coalesce to an $I_{12}^*$ fiber. 
  \end{enumerate}
\end{lemma}
\begin{proof}
The coefficients of the Weierstrass model in Equation~(\ref{eqns_SI_alt}) can be written in terms of modular forms. The proof follows from the application of Tate's algorithm. Notice that $J_4 =  \operatorname{Disc} (\mathcal{A}) = 0$ is equivalent to $J_4 = J_5 = 0$; and $\operatorname{Disc} (\mathcal{A})=\operatorname{Res} (\mathcal{A}, \mathcal{B})=0$ implies $\operatorname{Disc} (\mathcal{S})=0$.
\end{proof}
We also have the following:
\begin{lemma} 
\label{lem:extensions_alt_FB}
The Weierstrass model in Equation~(\ref{eqn:bfdual}) associated with a six-line configuration in $\mathbb{P}^2$ with invariants $J_k$ for $k=2, \dots, 6$ satisfies the following:
\begin{enumerate}
    \item[(0)] In the generic case, there are singular fibers $II^* +  6 \, I_1 + I_2^*$. 
    \item[(0b)] If $\operatorname{Res} (\check{f}\, T^{-2}, \check{g}\, T^{-3})=0$, two $I_1$ fibers coalesce to a $II$ fiber.
    \item[(1)] If $J_4 = 0$,  one $I_1$ and the $I_2^*$ fiber coalesce to a $III^*$ fiber. 
    \item[(2)] If $\operatorname{Disc} (\mathcal{A})=0$, one $I_1$ and the $I_2^*$ fiber coalesce to an $I_3^*$ fiber. 
    \item[(2b)] If $\operatorname{Disc} (\mathcal{S})=0$, two $I_1$ fibers coalesce to an $I_2$ fiber.
    \item[(3+4)]  If $\operatorname{Disc} (\mathcal{A})=\operatorname{Res} (\mathcal{A}, \mathcal{B})=0$,  two $I_1$ and the $I_2^*$ fiber coalesce to an $I_4^*$ fiber. 
    \item[(5)] If $J_4 = J_5 = 0$, two $I_1$ fibers and the $I_2^*$ fiber coalesce to a $II^*$ fiber. 
\end{enumerate}
\end{lemma}
\begin{proof}
The coefficients of the Weierstrass model in Equation~(\ref{eqn:bfdual}) can be written in terms of modular forms. The proof then follows from the application of Tate's algorithm. 
\end{proof}
Similarly, one proves the following:
\begin{lemma} 
\label{lem:extensions_std}
The Weierstrass model in Equation~(\ref{eqns_SI_std}) associated with a six-line configuration in $\mathbb{P}^2$ with invariants $J_k$ for $k=2, \dots, 6$ satisfies the following:
\begin{enumerate}
    \item[(0)] In the generic case, there are singular fibers $III^* +  6 \, I_1 + III^*$. 
    \item[(0b)] If $\operatorname{Res} (f s^{-2}, g s^{-5})=0$, two $I_1$ fibers coalesce to a $II$ fiber.
    \item[(1)] If $J_4 = 0$, one $I_1$ and one $III^*$ fiber coalesce to a $II^*$ fiber. 
    \item[(2b)] If $\operatorname{Disc} (\mathcal{S})=0$, two $I_1$ fibers coalesce to an $I_2$ fiber.
    \item[(5)] If $J_4 = J_5 = 0$, two pairs of $I_1$ and $III^*$ fiber coalesce each to a $II^*$ fiber. 
\end{enumerate}
\end{lemma}
Notice that cases in which two $I_1$'s coalesce to form a fiber of type $II$ or one $I_1$ fiber and one $I_2$ fiber coalesce to a fiber of type $III$ -- a case we labelled (0b), adding to cases (0) through (5) in Definition~\ref{def:sstable} -- do not affect the lattice polarization. An immediate consequence is the following:
\begin{corollary}
\label{cor:extensions}
The family of K3 surfaces in Equation~(\ref{mainquartic}) satisfies the following:
\begin{enumerate}
\item[(0)] For a generic point in $\mathfrak{M}$, there is a $H \oplus E_7(-1) \oplus E_7(-1)$ polarization.
\item[(1)] If $J_4 = 0$, the polarization extends to $H \oplus E_8(-1) \oplus E_7(-1)$.
\item[(2)] if $\operatorname{Disc} (\mathcal{A})=0$, the polarization extends to $H \oplus E_8(-1) \oplus D_7(-1)$.
\item[(2b)] If $\operatorname{Disc} (\mathcal{S})=0$, the polarization extends to $H \oplus E_7(-1) \oplus E_7(-1) \oplus \langle -2 \rangle$.
\item[(3+4)] If $\operatorname{Disc} (\mathcal{A})=\operatorname{Res} (\mathcal{A}, \mathcal{B})=0$, the polarization extends to $H \oplus E_8(-1) \oplus D_8(-1)$.
\item[(5)] If $J_4 = J_5 = 0$, the polarization extends to $H \oplus E_8(-1) \oplus E_8(-1)$.
\end{enumerate}
\end{corollary}
\begin{proof}
The presence of a singular fiber of Kodaira type $II^*$ in the fibration given by Equation~(\ref{eqn:bfdual}) implies that we have, in all cases, a Mordell-Weil group of sections $\operatorname{MW}(\check{\pi}^{\mathcal{X}}_{\mathrm{alt}})=\{ \mathbb{I} \}$ \cite{MR2732092}*{Lemma~7.3}. Therefore, the lattice polarization coincides with the trivial lattice generated by the singular fibers extended by $H$ generated by the classes of the smooth fiber and the section of the elliptic fibration.
\end{proof}
\subsection{F-theory/heterotic duality}
We determined equations for three important elliptic fibrations with sections on the universal family of such K3 surfaces admitting a $H \oplus E_7(-1) \oplus E_7(-1)$ lattice polarization in Theorem~\ref{thm:polarization} and Lemmas~\ref{lem:Ellstd},~\ref{lem:Ellalt},~\ref{lem:Ellalt_dual}.  We also identified the coarse moduli space $\mathfrak{M}$ of the Clingher-Doran family to be the quotient space of $\mathbf{H}_2$ by the modular group $\Gamma_{\mathcal{T}}$; see Theorem~\ref{thm4} and~\ref{cor:K3moduli_space}.  We will now show, that the ring of modular forms of even characteristic relative to $\Gamma_{\mathcal{T}}$ coincides with the ring of modular forms on the bounded symmetric domain of type $IV$. The latter is known to provide a quantum-exact effective heterotic description of a natural sub-space of the moduli space of non-geometric heterotic models \cites{MR3366121,MR2826187,MR3417046}. Therefore, we will have established a particular F-theory/heterotic string duality map in eight dimensions. The heterotic theories on a torus $\mathbf{T}^2$ in question have two complex moduli and two non-vanishing complex Wilson lines. The duality map is the exact correspondence between the two moduli spaces, known in the large volume limit on the heterotic side, which corresponds to the stable degeneration limit on the F-theory side and established in \cites{MR1468319,MR1643100,MR1697279,MR2126482}. Therefore, our result generalizes earlier results by \cites{MR2369941, MR2826187, MR2824841, MR2935386, MR3366121, MR3417046}.

\par We let $L^{2,4}$ be the lattice of signature $(2,4)$ which is the orthogonal complement of $E_7(-1)\oplus E_7(-1)$ in the unique integral even unimodular lattice $\Lambda^{2,18}$ of signature $(2,18)$, which is
\begin{equation}
\label{k3lattice}
\Lambda^{2,18}=H\oplus H \oplus E_8(-1)\oplus E_8(-1) \;.
\end{equation}
We restrict to the quotient of the symmetric space\footnote{By $\mathcal{D}_{p,q}$ we denote the symmetric space for $O(p,q)$, i.e.,
\begin{equation}
\mathcal{D}_{p,q} = (O(p)\times O(q))\backslash O(p,q).
\end{equation}
}
for $O(2,4)$ by the automorphism group $O(L^{2,4})$, i.e., the space
\begin{equation} \label{moduli_space}
\mathcal{D}_{2,4}/O(L^{2,4}).
\end{equation}
The space $\mathcal{D}_{2,4}$ is also known as \emph{bounded symmetric domain of type $IV$}. We further restrict to a certain index-two sub-group $O^+(L^{2,4}) \subset O(L^{2,4})$  in the construction above with the corresponding degree-two cover given by
\begin{equation} \label{eqn:MSP+} \mathcal{D}_{2,4}/O^+(L^{2,4}).
\end{equation}
The group $O^+(L^{2,4})$ is the maximal sub-group whose action preserves the complex structure on the symmetric space, and thus is the maximal sub-group for which the corresponding modular forms are holomorphic. 
\par We have the following:
\begin{theorem}
The natural sub-space of the moduli space of non-geometric heterotic models whose quantum-exact effective heterotic description is captured by the ring of holomorphic modular forms on the bounded symmetric domain of type $IV$ is isomorphic to the coarse moduli space K3 of surfaces admitting a $H \oplus E_7(-1) \oplus E_7(-1)$ lattice polarization in Theorem~\ref{thm:polarization}.
\end{theorem}
\begin{proof}
In Section~\ref{ssec:modular_description}, we introduced the space $\mathbf{H}_2$ of complex two-by-two matrices $\modvar$ over $\mathbb{C}$ such that the hermitian matrix $(\modvar-\modvar^\dagger)/(2i)$ is positive definite; see~Equation~(\ref{Siegel_tau}). On these elements, the modular group $\Gamma_{\mathcal{T}}$ introduced in Equation~(\ref{modular group_extended}) acts by matrix multiplication for elements in $U(2,2)$ and as matrix transposition, generated by the additional element $\mathcal{T}\cdot \modvar =\modvar^t$. In \cite{MR1204828}*{Prop.~\!1.5.1} it was proved that there is an isomorphism $\Gamma_{\mathcal{T}} \cong O^+(L^{2,4})$  that induces an isomorphism
\begin{equation} \label{eq:h2iso}
\mathbf{H}_2 \cong \mathcal{D}_{2,4}.
\end{equation}
Generally, $O^+(L^{2,n})$ is the index-two sub-group given by the condition that the upper left minor of order two is positive; see \cite{MR2682724} for details. The group $O^+(L^{2,4})$ contains the special orthogonal sub-group $SO^+(L^{2,4})$ of all elements of determinant one.  In our situation, this group $SO^+(L^{2,4})$ is precisely the index-two sub-group $\Gamma^{+}_{\mathcal{T}}$ introduced in Equation~(\ref{index-two}): an isomorphism $SO^+(L^{2,4}) \cong \Gamma^{+}_{\mathcal{T}}$ is given by mapping the generators of $\Gamma^{+}_{\mathcal{T}}$ to generators of $SO^+(L^{2,4})$. In fact, we map the generators $G_1\mathcal{T}$ and $G_2, \dots, G_5$  in Lemma~\ref{lem:generators} to the generators explicitly given in \cite{MR1204828}*{p.~\!393} and denote the latter by $\mathcal{G}_k \in SO^+(L^{2,4})$ for $k=1,\dots, 5$.\footnote{In \cite{MR1204828} $G_1$ was mapped to $\mathcal{G}_1$ which is \emph{not} compatible with the identification $SO^+(L^{2,4}) \cong \Gamma^{+}_{\mathcal{T}}$.} Moreover, the elements $G_1$ and $\mathcal{T}$ are mapped to reflections $\mathcal{R}_{G_1}$ and $\mathcal{R}_{\mathcal{T}}$ in $O^+(L^{2,4})$ associated with roots of square $-2$ and $-4$, respectively, such that $\mathcal{G}_1 =\mathcal{R}_{G_1} \cdot \mathcal{R}_{\mathcal{T}}$. We also find $\mathcal{G}_3=\mathcal{R}_{G_3} \cdot \mathcal{R}_{\mathcal{T}}$ for another reflection $\mathcal{R}_{G_3}$.\footnote{In \cite{MR1204828} the roots associated with $\mathcal{R}_{G_1}$, $\mathcal{R}_{\mathcal{T}}$, and $\mathcal{R}_{G_3}$ were denoted by $\alpha(1,2,3)$, $\beta_1$, and $\beta_6$.}  Note that reflections belong to $O^+(L^{2,4})$, but not to $SO^+(L^{2,4})$. In fact, the generators $\mathcal{R}_{G_1}, \mathcal{R}_{\mathcal{T}}\in O^+(L^{2,4})$ together with $\mathcal{G}_k \in SO^+(L^{2,4})$ for $k=1,\dots, 5$ determine the full isomorphism $\Gamma_{\mathcal{T}} \cong O^+(L^{2,4})$.
\par The element $\mathcal{T}$ acts trivially on the five modular forms $J_k$ of weights $2k$ for $k=2, \dots , 6$. Thus, they all have \emph{even characteristic} with respect to the action of $\mathcal{T}$. We proved in Theorem~\ref{thm4} that they freely generate the ring of modular forms relative to $\Gamma_{\mathcal{T}}$ with character $\chi_{2k}(g)=\det(G)^k$ for all $g = G\,\mathcal{T} ^{n} \in\Gamma_{\mathcal{T}}$. By a result of Vinberg~\cite{MR2682724}, the ring of modular forms relative to $O^+(L^{2,4})$ turns out to be exactly this ring of modular forms relative to $\Gamma_{\mathcal{T}}$ of even characteristic. 
\end{proof}
\par  The space $\mathbf{H}_2$ is a generalization of the Siegel upper-half space $\mathbb{H}_2$. In fact, elements invariant under transposition $\mathcal{T}$ are precisely the two-by-two symmetric matrices over $\mathbb{C}$ whose imaginary part is positive definite, i.e., elements of the Siegel upper-half plane $\mathbb{H}_2\cong \mathcal{D}_{2,3}$, on which the modular group $\operatorname{Sp}_4(\mathbb{Z})\cong SO^+(L^{2,3})$ acts. For the sub-space
\begin{equation} \label{eq:h2iso_rest}
\mathcal{D}_{2,3}/O^+(L^{2,3}) \hookrightarrow \mathcal{D}_{2,4}/O^+(L^{2,4}),
\end{equation}
another result of Vinberg \cite{MR3235787} proves that the ring of $O^+(L^{2,3})$-modular forms corresponds to the ring of Siegel modular forms of \emph{even weight}. Igusa \cite{MR0141643} showed that this ring of even modular forms is generated by the Siegel modular forms $\psi_4$, $\psi_6$, $\chi_{10}$, $\chi_{12}$ of respective weights $4, 6, 10, 12$. 
\par Matrix transposition $\mathcal{T}$ acts as $-1$ on the $\Gamma_{\mathcal{T}}$-modular forms of \emph{odd characteristic}, and the fixed locus of $\mathcal{T}$ must be contained in the vanishing locus of any $\Gamma_{\mathcal{T}}$-modular form of odd characteristic. Modular forms of odd characteristic are generated by the unique (up to scaling) modular form $\Theta(\modvar)$ of odd characteristic introduced in Theorem~\ref{thm1}.  In Theorem~\ref{thm4} we found the relation $J_4(\modvar)=(\Theta(\modvar)/25)^2$. Therefore, the fixed locus of $\mathcal{T}$ coincides with the vanishing locus of $J_4(\modvar)$. In fact, we will show in Proposition~\ref{prop:compare} that in the case $\modvar=\tau \in \mathbb{H}_2$ the form $J_4=(\Theta(\modvar)/25)^2$ vanishes and the other $\Gamma_{\mathcal{T}}$-modular forms restrict to the generators of the ring of Siegel modular forms.
%

\section{Specialization to six lines tangent to a conic}
\label{6Lrestricted}
In this section we consider the specialization of the generic six-line configuration when the six lines are tangent to a common conic. Such a configuration has three moduli which we will denote by $\lambda_1, \lambda_2, \lambda_3$. It follows from \cite{MalmendierClingher:2018}*{Prop.~5.13} that the lines can be brought into the form
\begin{equation}
\label{Eqn:6Ltangent}
\begin{array}{lrcl}
 \ell_1: & z_1 & = & 0,\\
 \ell_2: & z_2 & = & 0,\\
 \ell_3: & z_1 + z_2 - z_3 &=& 0,\\
 \ell_4: & \lambda_1^2 \, z_1 + z_2 -  \lambda_1 z_3 &=& 0,\\
 \ell_5: & \lambda_2^2 \, z_1 + z_2 -  \lambda_2 z_3 &=& 0,\\
 \ell_6: & \lambda_3^2 \, z_1 + z_2 -  \lambda_3 z_3 &=& 0,
\end{array}
\end{equation} 
where $\lambda_i \not = 0, 1, \infty$ and $\lambda_i \not = \lambda_j$ for all $i \not= j$. We have the following:
\begin{lemma}
The six lines in Equation~(\ref{Eqn:6Ltangent}) are tangent to $C: z_3^2-4z_1z_2=0$.
\end{lemma}
\begin{proof}
It is easy to prove that the intersection of the conic $C: z_3^2-4z_1z_2$ with any of the six lines $\ell_i$ for $1\le i \le6$ in Equation~(\ref{Eqn:6Ltangent}) yields a root of order two, that is, a point of tangency.
\end{proof}
\par The following lemma is immediate:
\begin{lemma}
\label{lem:EllLeft_special}
For a configuration of six lines tangent to a conic, the K3 surface $\mathcal{Y}$ satisfies the following:
\begin{enumerate}
\item Equation~(\ref{eqns_Kummer}) is a Jacobian elliptic fibration  $\pi^{\mathcal{Y}}_{\mathrm{nat}}: \mathcal{Y} \to \mathbb{P}^1$ with 6 singular fibers of type $I_2$,  two singular fibers of type $I_0^*$, and the Mordell-Weil group of sections $\operatorname{MW}(\pi^{\mathcal{Y}}_{\mathrm{nat}})=(\mathbb{Z}/2\mathbb{Z})^2+\langle 1 \rangle$.
\item Equation~(\ref{eqns_Kummer_alt}) is a Jacobian elliptic fibration $\pi^{\mathcal{Y}}_{\mathrm{alt}}: \mathcal{Y} \to \mathbb{P}^1$ with 6 singular fibers of type $I_2$, one fiber of type $I_5^*$, one fiber of type $I_1$, and a Mordell-Weil group of sections $\operatorname{MW}(\pi^{\mathcal{Y}}_{\mathrm{alt}})=\mathbb{Z}/2\mathbb{Z}$.
\end{enumerate}
\end{lemma}
\begin{proof}
 The proof of (1) was given in \cite{MalmendierClingher:2018}*{Prop.~5.13}. The proof of (2) was given in \cite{MR3712162}*{Prop.~9}.
\end{proof}
\begin{lemma}
\label{prop_Kummers}
For a configuration of six lines tangent to a conic, the K3 surface $\mathcal{Y}$  is the Kummer surface $\operatorname{Kum}(\operatorname{Jac} \mathcal{C})$ of the principally polarized abelian surface $\operatorname{Jac}(\mathcal{C})$, i.e., the Jacobian variety of a generic genus-two curve $\mathcal{C}$. In particular, the curve $\mathcal{C}$ is given in Rosenhain normal form as
\begin{equation}\label{Rosenhain}
 \mathcal{C}: \quad Y^2 = F(X)=X (X-1) (X-\lambda_1) (X-\lambda_2) (X-\lambda_3) .
\end{equation} 
\end{lemma}
\begin{proof}
All inequivalent elliptic fibrations on a generic Kummer surface where determined explicitly by Kumar in \cite{MR3263663}. In fact, Kumar computed elliptic parameters and Weierstrass equations for all twenty five different fibrations that appear, and analyzed the reducible fibers and Mordell-Weil lattices. Equation~(\ref{eqns_Kummer}) is the Weierstrass model of the  elliptic fibration {\tt (7)} in the list of all possible elliptic fibrations in~\cite{MR3263663}*{Thm.~2}.
\end{proof}
\par The ordered tuple $(\lambda_1, \lambda_2, \lambda_3)$ determines a point in the moduli space of genus-two curves together with a level-two structure,  and, in turn, a level-two structure on the corresponding Jacobian variety, i.e., a point in the moduli space of principally polarized abelian surfaces with level-two structure
\begin{equation}
 \mathfrak{A}_2(2) = \mathbb{H}_2 /  \Gamma_2(2),
\end{equation}
where $\Gamma_2(2)$ is the principal congruence sub-group of level two of the Siegel modular group $\Gamma_2 = \operatorname{Sp}_4(\mathbb{Z})$. In turn, the Rosenhain invariants generate the function field $\mathbb{C}(\lambda_1, \lambda_2, \lambda_3)$ of $\mathfrak{A}_2(2)$.
For a Jacobian variety with level-two structure corresponding to $\tau \in \mathfrak{A}_2(2)$, we have six odd theta characteristics and ten even theta characteristics;  see \cites{MR2367218,MR3782461}   for details.  We denote the even theta characteristics by 
\begin{small}
\[
\begin{split}
	\vartheta_1		=	\begin{bmatrix} 0 		& 0			\\[3pt] 0		 	& 0		 	\end{bmatrix}, \,
	\vartheta_2		=	\begin{bmatrix} 0 		& 0			\\[3pt] \frac{1}{2}	& \frac{1}{2}	\end{bmatrix}, \,
	\vartheta_3	&	=	\begin{bmatrix} 0 		& 0			\\[3pt] \frac{1}{2}	& 0			\end{bmatrix}, \;
	\vartheta_4		=	\begin{bmatrix} 0 		& 0			\\[3pt] 0			& \frac{1}{2}	\end{bmatrix}, \;
	\vartheta_5	\	=	\begin{bmatrix} \frac{1}{2} 	& 0			\\[3pt] 0		 	& 0		 	\end{bmatrix}, \\[5pt]
	\vartheta_6		=	\begin{bmatrix} \frac{1}{2} 	& 0			\\[3pt] 0		 	& \frac{1}{2} 	\end{bmatrix}, \,
	\vartheta_7		=	\begin{bmatrix} 0 		& \frac{1}{2}	\\[3pt] 0			& 0			\end{bmatrix}, \,
	\vartheta_8	&	=	\begin{bmatrix} \frac{1}{2} 	& \frac{1}{2}	\\[3pt] 0			& 0			\end{bmatrix}, \;
	\vartheta_9		=	\begin{bmatrix} 0 		& \frac{1}{2}	\\[3pt] \frac{1}{2}	& 0			\end{bmatrix}, \,
	\vartheta_{10}	=	\begin{bmatrix} \frac{1}{2} 	& \frac{1}{2}	\\[3pt] \frac{1}{2}	& \frac{1}{2}	\end{bmatrix}.
\end{split}
\]
\end{small}
and write
\begin{equation}
\label{Eqn:theta_short}
 \vartheta_i(\tau) \  \text{instead of} \  \vartheta\begin{bmatrix} a^{(i)} \\  b^{(i)} \end{bmatrix}(\tau) \ \text{with $i=1,\dots ,10$,}
\end{equation}
and $\vartheta_i =\vartheta_i(0)$.  Fourth powers of theta constants are modular forms of $\mathfrak{A}_2(2)$ and define the Satake compactification of $\mathfrak{A}_2(2)$ given by $\operatorname{Proj}[\vartheta^4_1: \dots : \vartheta^4_{10}]$. 
\par The three $\lambda$-parameters in the Rosenhain normal~(\ref{Rosenhain}) can be expressed as ratios of even theta constants by Picard's lemma. There are 720 choices for such expressions since the forgetful map, i.e., forgetting the level-two structure, is a Galois covering of degree $720 = |\mathrm{S}_6|$ since $\mathrm{S}_6$ acts on the  roots of $\mathcal{C}$ by permutations. Any of the $720$ choices may be used, we chose the one from \cite{MR0141643}:
\begin{lemma}
\label{lem:ppas}
If $\mathcal{C}$ is a non-singular genus-two curve with period matrix $\tau$ for $\operatorname{Jac}(\mathcal{C})$, then $\mathcal{C}$  is equivalent to the curve~(\ref{Rosenhain}) with Rosenhain parameters $\lambda_1, \lambda_2, \lambda_3$  given by 
\begin{equation}\label{Picard}
\lambda_1 = \frac{\vartheta_1^2\vartheta_4^2}{\vartheta_2^2\vartheta_3^2} \,, \quad \lambda_2 = \frac{\vartheta_4^2\vartheta_7^2}{\vartheta_2^2\vartheta_9^2}\,, \quad \lambda_3 =
\frac{\vartheta_1^2\vartheta_7^2}{\vartheta_3^2\vartheta_9^2}\,.
\end{equation}
Conversely, given three distinct complex numbers $(\lambda_1, \lambda_2, \lambda_3)$ different from $0, 1, \infty$ there is  complex abelian surface $A$ with period matrix $[\mathbb{I}_2 | \tau]$ such that $A =\operatorname{Jac} (\mathcal{C})$ where $\mathcal{C}$ is the genus-two curve with period matrix $\tau$.
\end{lemma}
\begin{proof}
A proof can be found in \cite{MR2367218}*{Lemma~8}.
\end{proof}
We also have the following:
\begin{lemma}
The following equations relate theta functions and branch points:
\begin{equation}\label{Thomae}
\begin{array}{ll}
\vartheta_1^4 = \kappa \, \lambda_1 \lambda_3 (\lambda_2 -1) (\lambda_3 - \lambda_1) &
\vartheta_2^4  = \kappa\, \lambda_3 (\lambda_2 - \lambda_1) (\lambda_3 - 1) \\[5pt]
\vartheta_3^4  = \kappa \, \lambda_2 (\lambda_2 -1) ( \lambda_3 - \lambda_1) &
\vartheta_4^4  = \kappa \, \lambda_1  \lambda_2 (\lambda_2 - \lambda_1) (\lambda_3 - 1) \\[5pt]
\vartheta_5^4  =  \kappa \, \lambda_2 (\lambda_1 - 1) (\lambda_3 - \lambda_1) (\lambda_3 - 1)&
\vartheta_6^4  =\kappa \, \lambda_3  (\lambda_1 - 1) (\lambda_2 - 1) (\lambda_2 - \lambda_1) \\[5pt]
\vartheta_7^4  =  \kappa\, \lambda_2 \lambda_3 (\lambda_1 - 1) (\lambda_3 - \lambda_2)  &
\vartheta_8^4  = \kappa \, \lambda_1  (\lambda_2 - 1) (\lambda_3 - 1) (\lambda_3 - \lambda_2)\\[5pt]
\vartheta_9^4  = \kappa \, \lambda_1  (\lambda_1 - 1) (\lambda_3 - \lambda_2), &
\vartheta_{10}^4  = \kappa \,  (\lambda_2 - \lambda_1) (\lambda_3 - \lambda_1) (\lambda_3 - \lambda_2),
\end{array}
\end{equation}
where $\kappa\not = 0$ is a non-zero constant. 
\end{lemma}
\begin{proof}
The proof follows immediately using Thomae's formula.
\end{proof}
We can now express the invariants $t_i$ in terms of theta functions:
\begin{proposition}\label{prop-5.6}
For a configuration of six lines tangent to a conic associated with a genus-two curve $\mathcal{C}$ with level-two structure, the period matrix $\tau \in \mathfrak{A}_2(2)$ determines a point $\modvar \in \mathbf{H}_2/\Gamma_{\mathcal{T}}(1+i)$ such that
\begin{equation}
\label{modd_restriced_thetas}
  \Big[ t_1(\modvar):    \dots :  t_{10}(\modvar)\Big] = \Big[ \vartheta_1^4(\tau): \ \vartheta_2^4(\tau): \ \dots : \vartheta_{10}^4(\tau)\Big] \in \mathbb{P}^9 , \quad R=0.
\end{equation}
\end{proposition}
\begin{proof}
For the lines in Equations~(\ref{Eqn:6Ltangent}) we compute the period matrix $\tau$ for $\operatorname{Jac}(\mathcal{C})$ using Lemma~(\ref{lem:ppas}). Setting $\modvar = \tau$  yields a $\mathcal{T}$-invariant point in $\mathbf{H}_2/\Gamma_{\mathcal{T}}(1+i)$. By construction, the modular forms $\theta_i^2(\modvar)$ equal $t_i$ for $1 \le i \le 10$ and can be computed directly from Equations~(\ref{eqn:DOcoords}) for the lines in Equations~(\ref{Eqn:6Ltangent}). On the other hand, we can also compute the fourth powers of theta functions directly using Equations~(\ref{Thomae}) to confirm Equation~(\ref{modd_restriced_thetas}). 
\end{proof}
\begin{remark}
Proposition~\ref{prop-5.6} is a special case of a statement in \cite{MR1204828}*{Lemma~2.1.1(vi)} where it was shown that under the restriction to $\mathbb{H}_2/\Gamma_2(2)$ we have $\theta_i(\modvar)=\vartheta_i^2(\tau)$.
\end{remark}
For the Siegel three-fold $\mathfrak{A}_2=\mathbb{H}_2/\Gamma_2$, i.e., the set of isomorphism classes of principally polarized abelian surfaces, the even Siegel modular forms of $\mathfrak{A}_2$ form a polynomial ring in four free generators of degrees $4$, $6$, $10$ and $12$ usually denoted by $\psi_4, \psi_6, \chi_{10}$ and $\chi_{12}$, respectively. Igusa showed in \cite{MR0229643} that for the full ring of modular forms, one needs an additional generator $\chi_{35}$ which is algebraically dependent on the others.  In fact, its square 
is a polynomial in the even generators given in \cite{MR0229643}*{p.~849}.
\par Let $I_2, I_4, I_6 , I_{10}$ denote Igusa invariants of the binary sextic $Y^2=F(X)$ as defined in \cite{MR3731039}*{Sec.~2.3}. Igusa \cite{MR0229643}*{p.~\!848} proved that the relation between the Igusa invariants of a binary sextic $Y^2=F(X)$ defining a genus-two curve $\mathcal{C}$ with period matrix $\tau$ for $\operatorname{Jac}(\mathcal{C})$ and the even Siegel modular forms are as follows:
\begin{equation}
\label{invariants}
\begin{split}
 I_2(F) & = -2^3 \cdot 3 \, \dfrac{\chi_{12}(\tau)}{\chi_{10}(\tau)} \;, \\
 I_4(F) & = \phantom{-} 2^2 \, \psi_4(\tau) \;,\\
 I_6(F) & = -\frac{2^3}3 \, \psi_6(\tau) - 2^5 \,  \dfrac{\psi_4(\tau) \, \chi_{12}(\tau)}{\chi_{10}(\tau)} \;,\\
 I_{10}(F) & = -2^{14} \, \chi_{10}(\tau) \;.
\end{split}
\end{equation}
Conversely, the point $[I_2 : I_4 : I_6 : I_{10}]\in \mathbb{P}(2,4,6,10)$ in weighted projective space equals
\begin{equation}
\label{IgusaClebschProjective}
 \big[ 2^3 3^2 \chi_{12} : 2^2 3^2 \psi_4 \chi_{10}^2 : 2^3 3^2 \big(12 \psi_4 \chi_{12}+ \psi_6 \chi_{10} \big) \chi_{10}^2: 2^2 \chi_{10}^6 \big] .
\end{equation}
We have the following:
\begin{proposition}
\label{prop:compare}
For a configuration of six lines tangent to a conic associated with the binary sextic $Y^2=F(X)$ defining a genus-two curve $\mathcal{C}$, the period matrix $\tau$ determines a point $\modvar \in \mathbf{H}_2/\Gamma_{\mathcal{T}}$ such that
\begin{equation}
\label{modd_restriced}
\scalemath{\MyScaleMedium}{
\begin{array}{c}
\Big[  J_2(\modvar) : J_3(\modvar) : J_4(\modvar): J_5(\modvar) : J_6(\modvar)   \Big]   = \Big[ \psi_4(\tau) : \  \psi_6(\tau) : \ 0: \ 2^{12} 3^5 \chi_{10}(\tau) : \ 2^{12}3^6 \chi_{12}(\tau)\Big]\\[1em]
 =  \Big[ \frac{1}{4} I_4(F): \frac{1}{8} (I_2 I_4-3 I_6)(F):   0: -\frac{243}{4} I_{10}(F):  \frac{243}{32} I_2 I_{10}(F)\Big] 
\end{array}}
\end{equation}
as points in $\mathbb{P}(2,3,4,5,6)$. The discriminant of the Satake sextic restricts to
\begin{equation}
\operatorname{Disc}(\mathcal{S}) = 2^{64} 3^{30} \frac{\chi_{35}^2(\tau)}{\chi_{10}(\tau)} .
\end{equation}
\end{proposition}
\begin{proof}
For the lines in Equations~(\ref{Eqn:6Ltangent}) we compute the period matrix $\tau$ for $\operatorname{Jac}(\mathcal{C})$ using Lemma~(\ref{lem:ppas}). Setting $\modvar = \tau$ and forgetting the level-two structure, yields a $\mathcal{T}$-invariant point in $\mathbf{H}_2/\Gamma_{\mathcal{T}}$. By construction, the modular forms $J_k(\modvar)$ equal $J_k$ for $2 \le k \le 6$ and can be computed directly from Equations~(\ref{eqn:Jinvariants}) for the lines in Equations~(\ref{Eqn:6Ltangent}). On the other hand, we can compute the Igusa invariants $I_2, I_4, I_6 , I_{10}$ of the binary sextic $Y^2=F(X)$ as defined in \cite{MR3731039}*{Sec.~2.3} for the genus-two curve~(\ref{Rosenhain}) to confirm Equation~(\ref{modd_restriced}). We then use Equation~(\ref{IgusaClebschProjective}) to convert to expressions in terms of $\psi_4, \psi_6, \chi_{10}$ and $\chi_{12}$.
\end{proof}
\par To summarize, when the six lines are tangent to a conic, the K3 surface $\mathcal{Y}$ becomes the Kummer surface $\mathrm{Kum}(\operatorname{Jac}\mathcal{C})$ of the Jacobian variety $\operatorname{Jac}(\mathcal{C})$ of a generic genus-two curve $\mathcal{C}$. In \cites{MR2824841,MR2935386} it was proved that the K3 surface $\mathcal{X}$ in turn is the Shioda-Inose surface $\mathrm{SI}(\operatorname{Jac}\mathcal{C})$, i.e., a K3 surface which carries a Nikulin involution such that quotienting by this involution and blowing up the fixed points, recovers the Kummer surface $\mathcal{Y}$ \emph{and} the rational quotient map of degree two induces a Hodge isometry\footnote{A Hodge isometry between two transcendental lattices is an isometry preserving the Hodge structure.} between the transcendental lattices $T_{\mathcal{X}}(2)$\footnote{The notation $T_{\mathcal{X}}(2)$ indicates that the bilinear pairing on the transcendental lattice $T_{\mathcal{X}}$ is multiplied by $2$.} and  $T_{\operatorname{Kum}(\operatorname{Jac}\mathcal{C})}$. In particular, the Shioda-Inose surface $\mathcal{X}$ and the Kummer surface $\operatorname{Kum}(\operatorname{Jac}\mathcal{C})$ have Picard rank greater or equal to $17$. Proposition~\ref{prop:compare} then has the following corollary:
\begin{corollary}
Configurations of six lines tangent to a conic give rise to a three-parameter family of Kummer surfaces $\mathrm{Kum}(\operatorname{Jac}\mathcal{C})$ of the Jacobian varieties $\operatorname{Jac}(\mathcal{C})$ of generic genus-two curves $\mathcal{C}$. Moreover, the corresponding three-parameter family of Shioda-Inose surfaces $\mathrm{SI}(\operatorname{Jac}\mathcal{C})$ associated with $\mathrm{Kum}(\operatorname{Jac}\mathcal{C})$ is obtained by setting $\epsilon=0$ and $\zeta=1$ in Equation~(\ref{mainquartic}).
\end{corollary}
We also have the following:
\begin{corollary}
Along the locus $J_4=0$, the lattice polarization of the K3 surfaces $\mathcal{X}(\alpha,\beta, \gamma, \delta, \varepsilon=0, \zeta=1)$ extends to a canonical 
$H \oplus E_8(-1) \oplus E_7(-1)$ lattice polarization. 
\end{corollary}
\begin{proof}
It was proved in \cite{MR2824841} that the family in Equation~(\ref{mainquartic}) with $\varepsilon=0, \zeta=1$ is endowed with a canonical 
$H \oplus E_8(-1) \oplus E_7(-1)$ lattice polarization. They also found the parameters $(\alpha,\beta,\gamma,\delta)$ in terms of the standard even Siegel modular forms $\psi_4, \psi_6, \chi_{10}, \chi_{12}$ (cf.~\cite{MR0141643}) given by
\begin{equation}
 (\alpha,\beta,\gamma,\delta) = \left(\psi_4, \psi_6, 2^{12}3^5 \, \chi_{10}, 2^{12}3^6 \, \chi_{12}\right) \;,
\end{equation}
which agrees with Equation~(\ref{modd_restriced}) and Equation~(\ref{modd}).
\end{proof}
\noindent
The different Jacobian elliptic fibrations, the Satake sextic, and further confluences of singular fibers were investigated in \cites{MR3712162, MR3731039}.
\bibliographystyle{amsplain}
\bibliography{ref}{}
\newpage
\begin{appendix}
\section{Invariants of the quintic pencils}
Using a 2-neighbor-step procedure twice starting with the natural fibration in Equation~(\ref{eqns_Kummer}), we constructed on the K3 surface $\mathcal{Y}$ associated with the double cover branched along the six lines given by Equations~(\ref{lines}) the following Weierstrass model:
\begin{equation}
\label{eqns_Kummer_alt_b}
  Y^2   = X \Big(X^2 - 2 \, \mathcal{B}(t)  X + \mathcal{B}(t)^2 - 4 \, \mathcal{A}(t) \Big), 
\end{equation} 
where $\mathcal{B}(t)=t^3- J'_2 \, t - J'_3$ and $\mathcal{A}(t) = J'_4 t^2 - J'_5 t  + J'_6$, and
\begin{equation}
\label{QP_inv}
\scalemath{\MyScaleTiny}{
\begin{array}{rcl}
 2 J'_2 & = & 2\,{a}^{2}{d}^{2}-2\,abcd+2\,{b}^{2}{c}^{2}-2\,{a}^{2}d+bca+adb+adc-2\,a{d}^{2}-2\,{b}^{2}c-2\,b{c}^{2}+bcd+2\,{a}^{2}\\
&&-2\,ba-2\,ca +ad+2\,{b}^{2}+bc-2\,db+2\,{c}^{2}-2\,cd+2\,{d}^{2},\\
-4 J'_3 &= & -4\,{a}^{3}{d}^{3}+6\,{a}^{2}bc{d}^{2}+6\,a{b}^{2}{c}^{2}d-4\,{b}^{3}{c}^{3}+6\,{a}^{3}{d}^{2}-6\,{a}^{2}bcd-3\,{a}^{2}b{d}^{2}\\
&&-3\,{a}^{2}c{d}^{2}+6\,{a}^{2}{d}^{3}-3\,a{b}^{2}{c}^{2}-6\,a{b}^{2}cd-6\,ab{c}^{2}d-6\,abc{d}^{2}+6\,{b}^{3}{c}^{2}+6\,{b}^{2}{c}^{3}\\
&&-3\,{b}^{2}{c}^{2}d+6\,{a}^{3}d-3\,{a}^{2}bc-6\,{a}^{2}bd-6\,{a}^{2}cd-6\,{a}^{2}{d}^{2}-6\,a{b}^{2}c-3\,a{b}^{2}d-6\,ab{c}^{2}+60\,abcd\\
&&-6\,ab{d}^{2}-3\,a{c}^{2}d-6\,ac{d}^{2}+6\,a{d}^{3}+6\,{b}^{3}c-6\,{b}^{2}{c}^{2}-6\,{b}^{2}cd+6\,b{c}^{3}-6\,b{c}^{2}d-3\,bc{d}^{2}-4\,{a}^{3}\\
&&+6\,{a}^{2}b+6\,{a}^{2}c-3\,{a}^{2}d+6\,a{b}^{2}-6\,bca-6\,adb+6\,a{c}^{2}-6\,adc-3\,a{d}^{2}-4\,{b}^{3}-3\,{b}^{2}c\\
&&+6\,{b}^{2}d-3\,b{c}^{2}-6\,bcd+6\,b{d}^{2}-4\,{c}^{3}+6\,{c}^{2}d+6\,c{d}^{2}-4\,{d}^{3},\\
16 \, J'_4 & = &  81 \, \left( bca-adb-adc+bcd+ad-bc \right) ^{2},\\
-\frac{8}{81}J'_5 & = & -2\,{b}^{2}{c}^{2}d{a}^{3}+{b}^{2}c{d}^{2}{a}^{3}+{b}^{2}{d}^{3}{a}^{3}+b{c}^{2}{d}^{2}{a}^{3}-4\,bc{d}^{3}{a}^{3}+{c}^{2}{d}^{3}{a}^{3}+{b}^{3}{c}^{3}{a}^{2}+{b}^{3}{c}^{2}d{a}^{2}\\
&&-2\,{b}^{3}c{d}^{2}{a}^{2}+{b}^{2}{c}^{3}d{a}^{2}+4\,{b}^{2}{c}^{2}{d}^{2}{a}^{2}+{b}^{2}c{d}^{3}{a}^{2}-2\,b{c}^{3}{d}^{2}{a}^{2}+b{c}^{2}{d}^{3}{a}^{2}-4\,{b}^{3}{c}^{3}da\\
&&+{b}^{3}{c}^{2}{d}^{2}a+{b}^{2}{c}^{3}{d}^{2}a-2\,{b}^{2}{c}^{2}{d}^{3}a+{b}^{3}{c}^{3}{d}^{2}+{b}^{2}{c}^{2}{a}^{3}+{b}^{2}cd{a}^{3}-2\,{b}^{2}{d}^{2}{a}^{3}+b{c}^{2}d{a}^{3}\\
&&+4\,bc{d}^{2}{a}^{3}+b{d}^{3}{a}^{3}-2\,{c}^{2}{d}^{2}{a}^{3}+c{d}^{3}{a}^{3}-2\,{b}^{3}{c}^{2}{a}^{2}+{b}^{3}cd{a}^{2}+{b}^{3}{d}^{2}{a}^{2}-2\,{b}^{2}{c}^{3}{a}^{2}\\
&&-4\,{b}^{2}{c}^{2}d{a}^{2}-4\,{b}^{2}c{d}^{2}{a}^{2}-2\,{b}^{2}{d}^{3}{a}^{2}+b{c}^{3}d{a}^{2}-4\,b{c}^{2}{d}^{2}{a}^{2}+4\,bc{d}^{3}{a}^{2}+{c}^{3}{d}^{2}{a}^{2}\\
&&-2\,{c}^{2}{d}^{3}{a}^{2}+{b}^{3}{c}^{3}a+4\,{b}^{3}{c}^{2}da+{b}^{3}c{d}^{2}a+4\,{b}^{2}{c}^{3}da-4\,{b}^{2}{c}^{2}{d}^{2}a+{b}^{2}c{d}^{3}a+b{c}^{3}{d}^{2}a\\
&&+b{c}^{2}{d}^{3}a+{b}^{3}{c}^{3}d-2\,{b}^{3}{c}^{2}{d}^{2}-2\,{b}^{2}{c}^{3}{d}^{2}+{b}^{2}{c}^{2}{d}^{3}-4\,bcd{a}^{3}+b{d}^{2}{a}^{3}+c{d}^{2}{a}^{3}-2\,{a}^{3}{d}^{3}\\
&&+{b}^{2}{c}^{2}{a}^{2}+4\,{b}^{2}cd{a}^{2}+{b}^{2}{d}^{2}{a}^{2}+4\,b{c}^{2}d{a}^{2}-4\,{a}^{2}bc{d}^{2}+b{d}^{3}{a}^{2}+{c}^{2}{d}^{2}{a}^{2}+c{d}^{3}{a}^{2}+{b}^{3}{c}^{2}a\\
&&-4\,{b}^{3}cda+{b}^{2}{c}^{3}a-4\,a{b}^{2}{c}^{2}d+4\,{b}^{2}c{d}^{2}a-4\,b{c}^{3}da+4\,b{c}^{2}{d}^{2}a-4\,bc{d}^{3}a-2\,{b}^{3}{c}^{3}+{b}^{3}{c}^{2}d\\
&&+{b}^{2}{c}^{3}d+{b}^{2}{c}^{2}{d}^{2}+{a}^{3}{d}^{2}+{a}^{2}bcd-2\,{a}^{2}b{d}^{2}-2\,{a}^{2}c{d}^{2}+{a}^{2}{d}^{3}-2\,a{b}^{2}{c}^{2}+a{b}^{2}cd\\
&&+ab{c}^{2}d+abc{d}^{2}+{b}^{3}{c}^{2}+{b}^{2}{c}^{3}-2\,{b}^{2}{c}^{2}d,\\
\frac{16}{81}J'_6 & = & -4\,{b}^{2}{c}^{2}d{a}^{4}+4\,b{c}^{2}d{a}^{4}+4\,{b}^{2}cd{a}^{4}-10\,bcd{a}^{4}+4\,b{c}^{2}{d}^{4}{a}^{3}-22\,b{c}^{2}{d}^{3}{a}^{3}-4\,{b}^{3}c{d}^{3}{a}^{3}\\
&&-22\,{b}^{2}c{d}^{3}{a}^{3}-10\,{b}^{2}{c}^{3}{d}^{2}{a}^{3}+16\,b{c}^{3}{d}^{2}{a}^{3}-10\,{b}^{3}{c}^{2}{d}^{2}{a}^{3}+4\,{b}^{2}{c}^{2}{d}^{2}{a}^{3}+16\,{b}^{3}c{d}^{2}{a}^{3}-4\,{b}^{3}{c}^{3}d{a}^{3}\\
&&+16\,{b}^{2}{c}^{3}d{a}^{3}-10\,b{c}^{3}d{a}^{3}+16\,{b}^{3}{c}^{2}d{a}^{3}-10\,{b}^{3}cd{a}^{3}+4\,{b}^{2}{c}^{2}{d}^{4}{a}^{2}-10\,b{c}^{2}{d}^{4}{a}^{2}-10\,{b}^{2}c{d}^{4}{a}^{2}\\
&&+4\,{b}^{2}c{d}^{4}{a}^{3}+12\,bc{d}^{4}{a}^{3}-4\,b{c}^{3}{d}^{3}{a}^{3}+12\,{b}^{2}{c}^{2}{d}^{3}{a}^{3}+12\,bc{d}^{4}{a}^{2}-10\,{b}^{2}{c}^{3}{d}^{3}{a}^{2}+16\,b{c}^{3}{d}^{3}{a}^{2}\\
&&-10\,{b}^{3}{c}^{2}{d}^{3}{a}^{2}+4\,{b}^{4}c{d}^{2}a-10\,{b}^{4}{c}^{4}da+12\,{b}^{3}{c}^{4}da+12\,{b}^{2}{c}^{4}da-10\,b{c}^{4}da+12\,{b}^{4}{c}^{3}da+12\,{b}^{4}{c}^{2}da\\
&&-10\,{b}^{4}cda+4\,{b}^{2}{c}^{2}{d}^{3}{a}^{2}+16\,{b}^{3}c{d}^{3}{a}^{2}+4\,{b}^{2}{c}^{4}{d}^{2}{a}^{2}-4\,b{c}^{4}{d}^{2}{a}^{2}+12\,{b}^{3}{c}^{3}{d}^{2}{a}^{2}+4\,{b}^{2}{c}^{3}{d}^{2}{a}^{2}\\
&&+4\,{b}^{4}{c}^{2}{d}^{2}{a}^{2}+4\,{b}^{3}{c}^{2}{d}^{2}{a}^{2}-4\,{b}^{4}c{d}^{2}{a}^{2}+4\,{b}^{3}{c}^{4}d{a}^{2}-10\,{b}^{2}{c}^{4}d{a}^{2}+4\,b{c}^{4}d{a}^{2}+4\,{b}^{4}{c}^{3}d{a}^{2}\\
&&-22\,{b}^{3}{c}^{3}d{a}^{2}-10\,{b}^{4}{c}^{2}d{a}^{2}+4\,{b}^{4}cd{a}^{2}-4\,{b}^{2}{c}^{2}{d}^{4}a+4\,b{c}^{2}{d}^{4}a+4\,{b}^{2}c{d}^{4}a-10\,bc{d}^{4}a-4\,{b}^{3}{c}^{3}{d}^{3}a\\
&&+16\,{b}^{2}{c}^{3}{d}^{3}a-10\,b{c}^{3}{d}^{3}a+16\,{b}^{3}{c}^{2}{d}^{3}a-10\,{b}^{3}c{d}^{3}a+4\,{b}^{3}{c}^{4}{d}^{2}a-10\,{b}^{2}{c}^{4}{d}^{2}a+4\,b{c}^{4}{d}^{2}a+4\,{b}^{4}{c}^{3}{d}^{2}a\\
&&-22\,{b}^{3}{c}^{3}{d}^{2}a-10\,{b}^{4}{c}^{2}{d}^{2}a-10\,bc{d}^{4}{a}^{4}+4\,b{c}^{2}{d}^{3}{a}^{4}+4\,{b}^{2}c{d}^{3}{a}^{4}+12\,bc{d}^{3}{a}^{4}+4\,{b}^{2}{c}^{2}{d}^{2}{a}^{4}\\
&&-10\,b{c}^{2}{d}^{2}{a}^{4}-10\,{b}^{2}c{d}^{2}{a}^{4}+12\,bc{d}^{2}{a}^{4}+4\,{b}^{2}c{d}^{2}{a}^{3}+12\,b{c}^{2}d{a}^{3}+12\,{b}^{2}cd{a}^{3}+4\,b{c}^{2}{d}^{3}{a}^{2}+4\,{b}^{2}c{d}^{3}{a}^{2}\\
&&+4\,{b}^{2}{c}^{3}d{a}^{2}+12\,b{c}^{3}d{a}^{2}+4\,{b}^{3}{c}^{2}d{a}^{2}+12\,{b}^{3}cd{a}^{2}+12\,b{c}^{2}{d}^{3}a+12\,{b}^{2}c{d}^{3}a+4\,{b}^{2}{c}^{3}{d}^{2}a+12\,b{c}^{3}{d}^{2}a\\
&&+4\,{b}^{3}{c}^{2}{d}^{2}a+12\,{b}^{3}c{d}^{2}a+4\,bc{d}^{3}{a}^{3}-22\,bc{d}^{2}{a}^{3}-22\,{b}^{2}{c}^{2}d{a}^{3}+4\,bcd{a}^{3}-22\,bc{d}^{3}{a}^{2}\\
&&-22\,b{c}^{3}{d}^{2}{a}^{2}+12\,{b}^{2}{c}^{2}{d}^{2}{a}^{2}+4\,b{c}^{2}{d}^{2}{a}^{2}-22\,{b}^{3}c{d}^{2}{a}^{2}+4\,{b}^{2}c{d}^{2}{a}^{2}+4\,{b}^{2}{c}^{2}d{a}^{2}-10\,b{c}^{2}d{a}^{2}\\
&&-10\,{b}^{2}cd{a}^{2}-22\,{b}^{2}{c}^{2}{d}^{3}a+4\,bc{d}^{3}a+4\,{b}^{2}{c}^{2}{d}^{2}a-10\,b{c}^{2}{d}^{2}a-10\,{b}^{2}c{d}^{2}a+4\,{b}^{3}{c}^{3}da-22\,{b}^{2}{c}^{3}da\\
&&+16\,{a}^{2}bc{d}^{2}+16\,a{b}^{2}{c}^{2}d+4\,b{c}^{3}da-22\,{b}^{3}{c}^{2}da+4\,{b}^{3}cda+4\,b{c}^{2}{d}^{2}{a}^{3}-4\,{b}^{3}{c}^{2}a-4\,{b}^{2}{c}^{3}d\\
&&+4\,{c}^{3}{d}^{2}{a}^{2}-10\,{b}^{2}{d}^{3}{a}^{2}+4\,{b}^{2}{c}^{2}{a}^{3}+4\,{b}^{2}{d}^{2}{a}^{2}-10\,{b}^{3}{c}^{2}{d}^{2}+4\,{b}^{2}{c}^{2}{a}^{2}+16\,{b}^{3}{c}^{3}a+4\,{c}^{2}{d}^{2}{a}^{2}\\
&&-10\,{b}^{2}{c}^{3}{d}^{2}-4\,b{d}^{3}{a}^{2}+4\,{b}^{3}{d}^{2}{a}^{2}+16\,c{d}^{3}{a}^{3}-10\,{b}^{3}{c}^{2}{a}^{2}+16\,b{d}^{3}{a}^{3}-4\,c{d}^{2}{a}^{3}+16\,{b}^{2}{d}^{3}{a}^{3}\\
&&-4\,c{d}^{3}{a}^{2}-10\,{b}^{2}{c}^{3}{a}^{2}+4\,{b}^{2}{c}^{2}{d}^{3}-10\,{c}^{2}{d}^{3}{a}^{2}+16\,{b}^{3}{c}^{3}d-10\,{c}^{2}{d}^{2}{a}^{3}+16\,{b}^{3}{c}^{3}{d}^{2}-4\,{b}^{3}{c}^{2}d\\
&&+4\,{b}^{2}{c}^{2}{d}{2}-4\,b{d}^{2}{a}^{3}-4\,{b}^{2}{c}^{3}a+16\,{b}^{3}{c}^{3}{a}^{2}-10\,{b}^{2}{d}^{2}{a}^{3}+16\,{c}^{2}{d}^{3}{a}^{3}+{b}^{4}{c}^{2}+4\,{b}^{4}{c}^{4}-4\,{b}^{3}{c}^{4}\\
&&-4\,{b}^{4}{c}^{3}+{d}^{2}{a}^{4}+{d}^{4}{a}^{2}+4\,{d}^{4}{a}^{4}-4\,{d}^{3}{a}^{4}-4\,{d}^{4}{a}^{3}+{b}^{2}{c}^{4}+2{a}^{3}{d}^{3}+2{b}^{3}{c}^{3}+4\,{b}^{4}{c}^{2}d\\
&&+2{b}^{3}{c}^{3}{a}^{3}+{b}^{4}{c}^{4}{d}^{2}-10\,b{d}^{3}{a}^{4}-4\,{b}^{3}{d}^{2}{a}^{3}+4{b}^{3}{d}^{3}{a}^{3}+{b}^{2}{c}^{2}{a}^{4}+4\,{b}^{2}{c}^{4}a+4\,b{d}^{4}{a}^{4}\\
&&-10\,c{d}^{3}{a}^{4}-4\,{c}^{3}{d}^{3}{a}^{2}+4\,{b}^{2}{d}^{4}{a}^{2}+{b}^{4}{d}^{2}{a}^{2}-10\,{b}^{3}{c}^{4}a-4\,{b}^{3}{c}^{4}{a}^{2}+4\,b{d}^{2}{a}^{4}+4\,{b}^{2}{d}^{2}{a}^{4}\\
&&+2{b}^{3}{c}^{3}{d}^{3}-10\,c{d}^{4}{a}^{3}-4\,{b}^{3}{c}^{2}{a}^{3}+4\,{c}^{2}{d}^{4}{a}^{2}+4\,c{d}^{4}{a}^{2}-4\,{b}^{4}{c}^{3}{a}^{2}-10\,{b}^{3}{c}^{4}d+4\,c{d}^{4}{a}^{4}-10\,b{d}^{4}{a}^{3}\\
&&+4\,b{d}^{4}{a}^{2}+2\,{c}^{3}{d}^{3}{a}^{3}+4\,{b}^{2}{c}^{4}d+4\,c{d}^{2}{a}^{4}-4\,{b}^{3}{d}^{3}{a}^{2}+4\,{b}^{4}{c}^{2}a-10\,{b}^{4}{c}^{3}d-4\,{c}^{2}{d}^{4}{a}^{3}+{b}^{4}{c}^{4}{a}^{2}\\
&&-4\,{b}^{3}{c}^{4}{d}^{2}-4\,{b}^{2}{c}^{3}{d}^{3}+4\,{b}^{4}{c}^{2}{a}^{2}-4\,{b}^{2}{c}^{3}{a}^{3}+4\,{b}^{4}{c}^{2}{d}^{2}-4\,{b}^{3}{c}^{2}{d}^{3}+4\,{b}^{4}{c}^{4}a-10\,{b}^{4}{c}^{3}a+4\,{b}^{2}{c}^{4}{d}^{2}\\
&&+2\,{c}^{4}{d}^{2}{a}^{2}-4\,{b}^{4}{c}^{3}{d}^{2}-4\,{b}^{2}{d}^{3}{a}^{4}-4\,{b}^{2}{d}^{4}{a}^{3}-4\,{c}^{2}{d}^{3}{a}^{4}+4\,{c}^{2}{d}^{2}{a}^{4}\\
&&+{b}^{2}{d}^{4}{a}^{4}+{b}^{2}{c}^{2}{d}^{4}+4\,{b}^{4}{c}^{4}d-4\,{c}^{3}{d}^{2}{a}^{3}+{c}^{2}{d}^{4}{a}^{4}+4\,{b}^{2}{c}^{4}{a}^{2}.
\end{array}}
\end{equation}

\end{appendix}
\end{document}